\documentclass[reqno,a4paper]{amsart}
\usepackage[utf8]{inputenc}
\usepackage{hyperref,xy,doi}
\usepackage{amsmath}
\usepackage{amsthm}
\usepackage{amssymb}
\usepackage{mathrsfs}
\usepackage{graphicx}
\usepackage{mathtools}
\usepackage{enumerate}
\usepackage{multicol}
\usepackage{multirow}
\usepackage{xcolor}
\usepackage[parfill]{parskip}
\usepackage{stix}
\usepackage{csquotes}

\numberwithin{equation}{subsection}

\xyoption{all}
\usepackage{ytableau}

\usepackage{tikz}
\usepackage{tikz-cd}
\usetikzlibrary{positioning, arrows, patterns,shapes,backgrounds, decorations.pathreplacing}

\hypersetup{
	colorlinks   = true, %Colours links instead of ugly boxes
	urlcolor     = blue, %Colour for external hyperlinks
	linkcolor    = purple, %Colour of internal links
	citecolor   = blue %Colour of citations
}
\usepackage{geometry}
\newtheorem{theorem}{Theorem}[section]
\newtheorem{corollary}[theorem]{Corollary}
\newtheorem{lemma}[theorem]{Lemma}
\newtheorem{proposition}[theorem]{Proposition}
\theoremstyle{definition}
\newtheorem{definition}[theorem]{Definition}
\theoremstyle{remark}
\newtheorem{remark}[theorem]{Remark}
\newtheorem{example}[theorem]{Example}

\DeclareMathOperator{\Res}{Res}
\DeclareMathOperator{\Ind}{Ind}

\DeclareMathOperator{\hd}{hd}
\DeclareMathOperator{\Vect}{Vect}
\DeclareMathOperator{\Ob}{Ob}
\DeclareMathOperator{\soc}{soc}
\DeclareMathOperator{\Coind}{Coind}
\DeclareMathOperator{\Hom}{Hom}
\DeclareMathOperator{\Tor}{Tor}
\DeclareMathOperator{\colim}{colim}
\DeclareMathOperator{\End}{End}
\DeclareMathOperator{\Rep}{Rep}
\DeclareMathOperator{\Mod}{\mathrm{-Mod_{lfd}}}
\DeclareMathOperator{\BMod}{\mathrm{-Mod}}
\DeclareMathOperator{\Proj}{\mathrm{-Proj}}
\DeclareMathOperator{\fdMod}{\mathrm{-Mod_{fd}}}

\DeclareMathOperator{\infl}{infl}
\DeclareMathOperator{\sgn}{sgn}

\DeclareMathOperator{\cont}{content}

\DeclareMathOperator{\G}{\overline{G}}
\DeclareMathOperator{\LR}{LR}
\DeclareMathOperator{\SW}{SW}
\DeclareMathOperator{\RS}{RS}
\DeclareMathOperator{\STab}{STab}
\DeclareMathOperator{\pr}{pr}
\DeclareMathOperator{\spin}{sp}
\DeclareMathOperator{\sh}{sh}
\DeclareMathOperator{\Rtab}{RTab}
\DeclareMathOperator{\SRTab}{SRTab}
\DeclareMathOperator{\Tab}{Tab}
\DeclareMathOperator{\bumpout}{\mathbf{bumpout}}
\DeclareMathOperator{\firstr}{\mathbf{firstr}}
\DeclareMathOperator{\nextr}{\mathbf{nextr}}

\newcommand{\CC}{\Bbbk }
\newcommand{\C}{ \mathbf C}
\newcommand{\B}{\mathbf B}
\newcommand{\FF}{\Bbbk}

\newcommand{\KK}{\mathbb K}
\newcommand{\LL}{\mathcal L}
\newcommand{\ZZ}{\mathbb{Z}}

\newcommand{\cPar}{\mathrm{CPar}}
\newcommand{\parti}{\mathrm{Par}}
\newcommand{\cSym}{\mathrm{CSym}}

\newcommand{\T}{\mathrm{T}}
\newcommand{\Kar}{\mathrm{Kar}}
\newcommand{\pn}{\mathrm{rn}}
\newcommand{\op}{\mathrm{op}}

\newcommand{\OV}[1]{\overrightarrow{#1}}
\newcommand{\Inv}{\mathpzc{i}}
\newcommand{\CPar}{\mathpzc{CPar}}
\newcommand{\CSym}{\mathpzc{CSym}}
\newcommand{\Par}{\mathpzc{Par}}
\newcommand{\Gr}{\mathpzc{G}}

\newcommand{\Cstar}{\begin{tikzpicture}
		\node[circle, draw,inner sep=-0.9pt]{$\star$};
\end{tikzpicture}}

\newcommand{\cstar}{\,\Cstar\,}

\newcommand\opendot[1]{\filldraw[fill=white,draw=black] (#1) circle (2pt)}

\newcommand\blackdot[1]{\filldraw[fill=black,draw=black] (#1) circle (2pt)}

\newcommand{\bluecircle}{
	\begin{tikzpicture}[scale=1,mycirc/.style={circle,fill=blue, minimum size=0.1mm, inner sep = 1.1pt}]
		\node[mycirc] at (0,0) {};
	\end{tikzpicture}
}
\newcommand{\redcircle}{
	\begin{tikzpicture}[scale=1,mycirc/.style={circle,fill=red, minimum size=0.1mm, inner sep = 1.1pt}]
		\node[mycirc] at (0,0) {};
	\end{tikzpicture}
}

\newcommand{\merge}{\resizebox{0.8 cm}{!}{
		\begin{tikzpicture}
			\draw [black, line width=4 pt] (0,0)--(2,2);
			\draw [black, line width=4pt] (4,0)--(2,2);
			\draw [black,line width=4pt] (2,2)--(2,4);
		\end{tikzpicture}
}}

\newcommand\spt{\resizebox{0.8 cm}{!}{
		\begin{tikzpicture}
			\draw [black, line width=4 pt] (0,0)--(0,2);
			\draw [black, line width=4pt] (0,2)--(2,4);
			\draw [black,line width=4pt] (0,2)--(-2,4);
		\end{tikzpicture}
}}

\newcommand\cross{\resizebox{0.8 cm}{!}{
		\begin{tikzpicture}
			\draw [black, line width=2 pt] (0,0)--(2,2);
			\draw [black, line width=2pt] (2,0)--(0,2);
		\end{tikzpicture}
}}

\newcommand\diaga{\resizebox{0.8cm}{!}{
		\begin{tikzpicture}
			\draw [black, line width=4 pt] (0,0)--(2,2);
			\draw [black,line width=4 pt] (3.1,.6)--(2,2);
			\draw [black,line width=4 pt] (2,2)--(2,4);
			\filldraw[fill=white,draw=black, very thick] (3.1,.6) circle (9pt);
\end{tikzpicture}}}

\newcommand\diagb{\resizebox{0.8 cm}{!}{
		\begin{tikzpicture}
			\draw [black, line width=4 pt] (.6,.6)--(2,2);
			\draw [black,line width=4 pt] (4,0)--(2,2);
			\draw [black,line width=4 pt] (2,2)--(2,4);
			\filldraw[fill=white,draw=black, very thick] (.6,.6) circle (9pt);
\end{tikzpicture} }}

\newcommand\diagc{\resizebox{0.8 cm}{!}{
		\begin{tikzpicture}
			\draw [black, line width=4 pt] (0,0)--(0,2);
			\draw [black, line width=4pt] (0,2)--(1.4,3.4);
			\draw [black,line width=4pt] (0,2)--(-2,4);
			\filldraw[fill=white,draw=black, very thick] (1.4,3.4) circle (9pt);
\end{tikzpicture}}}

\newcommand\diagd{	\resizebox{0.8 cm}{0.8 cm}{
		\begin{tikzpicture}
			\draw [black, line width=4 pt] (0,0)--(0,2);
			\draw [black, line width=4pt] (0,2)--(2,4);
			\draw [black,line width=4pt] (0,2)--(-1.4,3.4);
			\filldraw[fill=white,draw=black, very thick] (-1.4,3.4) circle (9pt);
\end{tikzpicture}}}

\newcommand\diage{\resizebox{0.6cm}{1.2cm}{\begin{tikzpicture}
			\draw [black, line width=1 pt] (0,-1)--(0,1);
			\blackdot{0,0};
			\node[scale=1.5] at (0.26,0) {$a$};
\end{tikzpicture}}}

\newcommand\diagf{\resizebox{1cm}{1.3cm}{\begin{tikzpicture}
			\draw [black, line width=1 pt] (2.2,-1)--(2.2,1);
			\blackdot{2.2,0.7};
			\blackdot{2.2,0.4};
			\blackdot{2.2,0.1};
			\blackdot{2.2, -0.6};
			\node at (2.1,-0.25) {$\vdots$};
			\draw [decorate,decoration={brace,amplitude=5pt,mirror},xshift=-4pt,yshift=0pt]
			(2,0.7) -- (2,-0.6) node[scale=1.5] at (1.6, 0.05) {$a$};
\end{tikzpicture}}}

\newcommand\diagg{\resizebox{0.7cm}{1cm}{
		\begin{tikzpicture}
			\draw [black, line width= 1.5 pt] (0,-1)--(1,1);
			\draw [black, line width= 1.5pt] (0.7,-0.7)--(0,1);
			\filldraw[fill=white,draw=black] (.7,-.7) circle (3pt);
\end{tikzpicture}}}

\newcommand\diagh{\resizebox{0.6cm}{1cm}{
		\begin{tikzpicture}
			\draw[black, line width=1.2 pt] (2.2,-0.7)--(2.2,1);
			\filldraw[fill=white,draw=black] (2.2,-0.7) circle (3pt);
			\draw[black, line width=1.2 pt] (2.8,-1)--(2.8,1);
\end{tikzpicture}}}

\newcommand\diagi{\resizebox{0.7cm}{1cm}{
		\begin{tikzpicture}
			\draw [black, line width= 1.5 pt] (0,-1)--(1,1);
			\draw [black, line width= 1.5pt] (0.7,-1)--(0,0.7);
			\filldraw[fill=white,draw=black] (0,0.7) circle (3pt);
\end{tikzpicture}}}

\newcommand\diagj{\resizebox{0.6cm}{1cm}{
		\begin{tikzpicture}
			\draw[black, line width=1.2 pt] (2.2,-1)--(2.2,0.7);
			\filldraw[fill=white,draw=black] (2.2,0.7) circle (3pt);
			\draw[black, line width=1.2 pt] (2.8,-1)--(2.8,1);
\end{tikzpicture}}}

\newcommand\diagk{\resizebox{0.7cm}{1.1cm}{
		\begin{tikzpicture}
			\draw [black, line width= 1.2pt] (0,-1)--(0,1);
			\draw [black, line width= 1.2pt] (1,-1)--(1,1);
			\draw [black, line width= 1.2pt] (0,0.5)--(1,-0.5);
\end{tikzpicture}}}

\newcommand\diagl{\resizebox{0.7cm}{1.1cm}{
		\begin{tikzpicture}
			\draw [black, line width=1.2 pt] (2.5,-1)--(3,-0.5);
			\draw [black, line width=1.2 pt] (3.5,-1)--(3,-0.5);
			\draw [black, line width=1.2 pt] (3,-0.5)--(3,0.5);
			\draw [black, line width=1.2 pt] (3,0.5)--(3.5,1);
			\draw [black, line width=1.2 pt] (3,0.5)--(2.5,1);
\end{tikzpicture}}}

\newcommand\diagm{\resizebox{0.7cm}{1.1cm}{
		\begin{tikzpicture}	
			\draw [black, line width= 1.2pt] (4.5,-1)--(4.5,1);
			\draw [black, line width= 1.2pt] (5.5,-1)--(5.5,1);
			\draw [black, line width= 1.2pt] (4.5,-0.5)--(5.5,0.5);
\end{tikzpicture}}}

\newcommand\diagn{\resizebox{0.7cm}{1cm}{\begin{tikzpicture}
			\draw [black, line width=1.2 pt] (0,0)--(0.5,0.5);
			\draw [black, line width=1.2 pt] (1,0)--(0.5,0.5);
			\draw [black,line width=1.2 pt] (0.5,0.5)--(0.5,1);
			\draw [black, line width=1.2 pt] (0,-1)--(1,0);
			\draw [black, line width=1.2 pt] (1,-1)--(0,0);
\end{tikzpicture}}}

\newcommand\diagp{\resizebox{0.7cm}{1cm}{\begin{tikzpicture}
			\draw [black, line width=1.2 pt] (0,0)--(0.5,0.5);
			\draw [black, line width=1.2 pt] (1,0)--(0.5,0.5);
			\draw [black,line width=1.2 pt] (0.5,0.5)--(0.5,1);
			\draw [black, line width=1.2 pt] (0,0)--(0.5,-0.5);
			\draw [black, line width=1.2 pt] (1,0)--(0.5,-0.5);
			\draw [black, line width=1.2 pt] (0.5,-0.5)--(0.5,-1);
\end{tikzpicture}}}

\newcommand\diagq{\resizebox{0.03 cm}{1cm}{\begin{tikzpicture}
			\draw [black,line width=1pt] (2.2,-1)--(2.2,1);
\end{tikzpicture}}}

\newcommand\diagr{\resizebox{0.7 cm}{1 cm}{
		\begin{tikzpicture}
			\draw [black, line width=1.2 pt] (0,0)--(1,1);
			\draw [black, line width=1.2pt] (1,0)--(0,1);
			\draw [black, line width=1.2pt] (0,-1)--(1,0);
			\draw [black, line width=1.2pt] (1,-1)--(0,0);
\end{tikzpicture}}}

\newcommand\diags{\resizebox{0.7cm}{1cm}{
		\begin{tikzpicture}
			\draw [black, line width=1.2pt](2.5,-1)--(2.5,1);
			\draw [black, line width=1.2pt](3.5,-1)--(3.5,1);
\end{tikzpicture}}}

\newcommand\diagt{\resizebox{0.8cm}{1cm}{
		\begin{tikzpicture}
			\draw [black, line width= 1.5 pt] (0,-1)--(2,1);
			\draw [black, line width= 1.5pt] (2,-1)--(0,1);
			\draw [black, line width= 1.5 pt] (1,-1)..controls (0.2,0) ..(1,1);
\end{tikzpicture}}}

\newcommand\diagu{\resizebox{0.8cm}{1cm}{
		\begin{tikzpicture}
			\draw [black, line width= 1.5 pt] (3.5,-1)--(5.5,1);
			\draw [black, line width= 1.5pt] (5.5,-1)--(3.5,1);
			\draw [black, line width= 1.5 pt] (4.5,-1)..controls (5.3,0) ..(4.5,1);
\end{tikzpicture}}}

\newcommand\diagv{\resizebox{0.5cm}{1cm}{\begin{tikzpicture}
			\draw [black, line width=1 pt] (0,-1)--(0,1);
			\blackdot{0,0};
			\node[scale=1.5] at (0.25,0) {$r$};
\end{tikzpicture}}}

\newcommand\diagw{\resizebox{0.8cm}{1cm}{
		\begin{tikzpicture}
			\draw [black, line width=1 pt] (0,-1)--(1,1);
			\draw [black, line width=1 pt] (1,-1)--(0,1);
			\blackdot{0.29,-0.4};
\end{tikzpicture}}}	

\newcommand\diagx{\resizebox{0.8cm}{1cm}{
		\begin{tikzpicture}
			\draw [black, line width=1 pt] (2.2,-1)--(3.2,1);
			\draw [black, line width=1 pt] (3.2,-1)--(2.2,1);
			\blackdot{2.9,0.4};
\end{tikzpicture}}}

\newcommand\diagy{\resizebox{0.8cm}{1cm}{
		\begin{tikzpicture}
			\draw [black, line width=1 pt] (0,-1)--(1,1);
			\draw [black, line width=1 pt] (1,-1)--(0,1);
			\blackdot{0.7,-0.4};
\end{tikzpicture}}}	

\newcommand\diagz{\resizebox{0.8cm}{1cm}{
		\begin{tikzpicture}	
			\draw [black, line width=1 pt] (2.2,-1)--(3.2,1);
			\draw [black, line width=1 pt] (3.2,-1)--(2.2,1);
			\blackdot{2.49,0.4};
\end{tikzpicture}}}

\newcommand\diagaa{\resizebox{0.8cm}{1cm}{\begin{tikzpicture}
			\draw [black, line width=1 pt] (0,-1)--(0.5,0);
			\draw [black, line width=1 pt] (1,-1)--(0.5,0);
			\draw [black,line width=1 pt] (0.5,0)--(0.5,1);
			\blackdot{0.29,-0.4};
\end{tikzpicture}}}			

\newcommand\diagab{\resizebox{0.8cm}{1cm}{\begin{tikzpicture}			
			\draw [black, line width=1 pt] (2.2,-1)--(2.7,0);
			\draw [black, line width=1 pt] (3.2,-1)--(2.7,0);
			\draw [black,line width=1 pt] (2.7,0)--(2.7,1);
			\blackdot{2.91,-0.4};
\end{tikzpicture}}}	

\newcommand\diagac{\resizebox{0.8cm}{1cm}{\begin{tikzpicture}				
			\draw [black, line width=1 pt] (4.4,-1)--(4.9,0);
			\draw [black, line width=1 pt] (5.4,-1)--(4.9,0);
			\draw [black,line width=1 pt] (4.9,0)--(4.9,1);
			\blackdot{4.9,0.5};
\end{tikzpicture}}}

\newcommand\diagba{\resizebox{0.8cm}{1cm}{\begin{tikzpicture}
			\draw [black, line width=1 pt] (0.5,0)--(0.5,-1);
			\draw [black, line width=1 pt] (1,1)--(0.5,0);
			\draw [black,line width=1 pt] (0,1)--(0.5,0);
			\blackdot{0.25,0.5};
\end{tikzpicture}}}

\newcommand\diagbb{\resizebox{0.8cm}{1cm}{\begin{tikzpicture}		
			\draw [black, line width=1 pt] (2.2,1)--(2.7,0);
			\draw [black, line width=1 pt] (3.2,1)--(2.7,0);
			\draw [black,line width=1 pt] (2.7,0)--(2.7,-1);
			\blackdot{2.95,0.5};
\end{tikzpicture}}}			

\newcommand\diagbc{\resizebox{0.8cm}{1cm}{\begin{tikzpicture}			
			\draw [black, line width=1 pt] (4.4,1)--(4.9,0);
			\draw [black, line width=1 pt] (5.4,1)--(4.9,0);
			\draw [black,line width=1 pt] (4.9,0)--(4.9,-1);
			\blackdot{4.9,-0.5};
\end{tikzpicture}}}

\newcommand\diagcc{\resizebox{0.6cm}{1.2cm}{\begin{tikzpicture}
			\draw [black, line width=1 pt] (0,-1)--(0,1);
			\opendot{0,-1};
			\opendot{0,1};
			\blackdot{0,0};
			\node[scale=1.5] at (0.25,0) {$i$};
\end{tikzpicture}}}

\newcommand\diagda{\resizebox{0.8cm}{1cm}{\begin{tikzpicture}
			\draw [black, line width=1.4 pt] (-0.2,-1)--(0.5,0);
			\draw [black, line width=1.4 pt] (1.2,-1)--(0.5,0);
			\draw [black,line width=1.4 pt] (0.5,0)--(0.5,1);
			\blackdot{0.2,-0.4};
			\blackdot{0.8,-0.4};
			\node at (-0.01,-0.4) {$i$};
			\node at (1.1,-0.4) {$j$};
\end{tikzpicture}}}			

\newcommand\diagdb{\resizebox{0.8cm}{1cm}{\begin{tikzpicture}			
			\draw [black, line width=1.4 pt] (3.0,-1)--(3.9,0);
			\draw [black, line width=1.4 pt] (4.8,-1)--(3.9,0);
			\draw [black,line width=1.4 pt] (3.9,0)--(3.9,1);
			\blackdot{3.9,0.5};
			\node at (4.5,0.5) {$i+j$};
\end{tikzpicture}}}

\newcommand\diagea{\resizebox{0.8cm}{1cm}{\begin{tikzpicture}
			\draw [black, line width=1.4 pt] (0.5,0)--(0.5,-1);
			\draw [black, line width=1.4 pt] (1.2,1)--(0.5,0);
			\draw [black,line width=1.4 pt] (-0.2,1)--(0.5,0);
			\blackdot{0.15,0.5};
			\blackdot{0.85,0.5};
			\node at (-0.08,0.5) {$i$};
			\node at (1.08,0.5) {$j$};
\end{tikzpicture}}}	

\newcommand\diageb{\resizebox{0.8cm}{1cm}{\begin{tikzpicture}
			\draw [black, line width=1.4 pt] (3.0,1)--(3.9,0);
			\draw [black, line width=1.4 pt] (4.8,1)--(3.9,0);
			\draw [black,line width=1.4 pt] (3.9,0)--(3.9,-1);
			\blackdot{3.9,-0.5};
			\node at (4.5,-0.5) {$i+j$};
\end{tikzpicture}}}

\newcommand\diagfa{\resizebox{0.8cm}{1.2cm}{\begin{tikzpicture}	
			\draw [black, line width=1 pt] (0,-1)--(0.5,0);
			\draw [black, line width=1 pt] (1,-1)--(0.5,0);
			\draw [black,line width=1 pt] (0.5,0)--(0.5,1);
			\filldraw[fill=white,draw=black]  (0.5,0.9) circle (3pt);
\end{tikzpicture}}}

\newcommand\diagfb{\resizebox{0.8cm}{0.5cm}{\begin{tikzpicture}	
			\draw[black, line width=4.5 pt] (3.3,-1) arc (1:180:3);
\end{tikzpicture}}}

\newcommand\diagga{\resizebox{0.8cm}{1.2cm}{\begin{tikzpicture}
			\draw [black, line width=1 pt] (0.5,0)--(0.5,-1);
			\draw [black, line width=1 pt] (1,1)--(0.5,0);
			\draw [black,line width=1 pt] (0,1)--(0.5,0);
			\filldraw[fill=white,draw=black]  (0.5,-0.9) circle (3pt);
\end{tikzpicture}}}

\newcommand\diaggb{\resizebox{0.8cm}{0.5cm}{\begin{tikzpicture}
			\draw[black, line width=4.5 pt] (2.2,1) arc (-180:0:3);
\end{tikzpicture}}}	

\newcommand\diagha{\resizebox{4cm}{3cm}{
		\begin{tikzpicture}
			\draw [black, line width= 1pt] (-0.8,-1.3)..controls(-0.9,-1).. (-0.8,1);
			\draw [black, line width= 1pt] (-0.6,-1.3)..controls(-0.7,-1).. (-0.6,1);
			\draw [black, line width= 1pt] (-0.4,-1.3)..controls(-0.5,-1).. (-0.4,1);
			\node [scale=1 ] at (-0.6, 1.2) {$c$};
			\node [scale=1] at (-0.6, -1.5) {$c$};
			\draw [black, line width= 1pt]	(1.7,0.5)--(1.7,1);
			\draw [black, line width= 1pt] (-0.1,-1.3)..controls(0.1,-0.3)..(1.7,0.5);
			\draw [black, line width= 1pt] (3,-1.3)..controls(3.1,-0.3)..(1.7,0.5);
			\draw [black, line width= 1pt]	(1.5,0.3)--(1.5,1);
			\draw [black, line width= 1pt] (0.3,-1.3)..controls(0.2,-0.5)..(1.5,0.3);
			\draw [black, line width= 1pt] (2.7,-1.3)..controls(2.8,-0.5)..(1.5,0.3);
			\draw [black, line width= 1pt]	(1.3,0)--(1.3,1);
			\draw [black, line width= 1pt] (0.6,-1.3)..controls(0.5,-0.8)..(1.3,0);
			\draw [black, line width= 1pt] (2.3,-1.3)..controls(2.4,-0.8)..(1.3,0);
			\node [scale=1 ] at (1.5, 1.2) {$t$};
			\node [scale=1 ] at (0.25, -1.5) {$t$};
			\node [scale=1 ] at (2.8, -1.5) {$t$};
			\draw [black, line width= 1pt]  (0.8, -1.3)..controls(1.5,-0.4) .. (2,-1.3);
			\draw [black, line width= 1pt]  (1, -1.3)..controls(1.5,-0.6) .. (1.8,-1.3);
			\node [scale=1 ] at (0.9, -1.5) {$a$};
			\node [scale=1 ] at (1.9, -1.5) {$a$};
			\draw [black, line width= 1pt] (3.6,-1.3)..controls(3.7,-1).. (3.6,1);
			\draw [black, line width= 1pt] (3.9,-1.3)..controls(4,-1).. (3.9,1);
			\draw [black, line width= 1pt] (4.1,-1.3)..controls(4.2,-1).. (4.1,1);
			\node [scale=1 ] at (3.85, 1.2) {$b$};
			\node [scale=1 ] at (3.85, -1.5) {$b$};
\end{tikzpicture}}}

\newcommand\tiaga{\resizebox{0.7cm}{1cm}{
		\begin{tikzpicture}
			\draw [black, line width=1.2 pt] (0,-1)--(1,0);
			\draw [black, line width=1.2 pt] (1,-1)--(0,0);
			\draw [black, line width=1.2 pt] (1,-1)--(0.5,-1.5);
			\draw [black, line width=1.2 pt] (0,-1)--(0.5,-1.5);
			\draw [black,line width=1.2 pt] (0.5,-1.5)--(0.5,-2);
\end{tikzpicture}}}

\newcommand\tiagba{\resizebox{1cm}{1.3cm}{
		\begin{tikzpicture}
			\draw [black, line width=1.2 pt] (0,0)--(0.5,-0.5);
			\draw [black, line width=1.2pt] (1,0)--(0.5,-0.5);
			\draw [black,line width=1.2pt] (1.5,0)--(1.5,-1);
			\draw [black,line width=1.2pt] (0.5,-0.5)--(0.5,-1);
			\draw [black,line width=1.2pt] (0.5,-1)--(1,-1.5);
			\draw [black,line width=1.2pt] (1.5,-1)--(1,-1.5);
			\draw [black,line width=1.2pt] (1,-2)--(1,-1.5);
\end{tikzpicture}}}

\newcommand\tiagbb{\resizebox{1cm}{1.3cm}{
		\begin{tikzpicture}
			\draw [black, line width=1.2 pt] (0,0)--(0,-1);
			\draw [black, line width=1.2pt] (0.5,0)--(1,-0.5);
			\draw [black,line width=1.2pt] (1.5,0)--(1,-0.5);
			\draw [black,line width=1.2pt] (1,-0.5)--(1,-1);
			\draw [black,line width=1.2pt] (1,-1)--(0.5,-1.5);
			\draw [black,line width=1.2pt] (0,-1)--(0.5,-1.5);
			\draw [black,line width=1.2pt] (0.5,-2)--(0.5,-1.5);
\end{tikzpicture}}}

\newcommand\tiagca{\resizebox{1cm}{1.3cm}{
		\begin{tikzpicture}
			\draw [black, line width=1.2 pt] (0,2)--(0,1.5);
			\draw [black, line width=1.2pt] (-0.5,1)--(0,1.5);
			\draw [black,line width=1.2pt] (0.5,1)--(0,1.5);
			\draw [black,line width=1.2pt] (0.5,1)--(0.5,0.5);
			\draw [black,line width=1.2pt] (0.5,0.5)--(0,0);
			\draw [black,line width=1.2pt] (0.5,0.5)--(1,0);
			\draw [black,line width=1.2pt] (-0.5,1)--(-0.5,0);
\end{tikzpicture}}}

\newcommand\tiagcb{\resizebox{1cm}{1.3cm}{
		\begin{tikzpicture}
			\draw [black, line width=1.2 pt] (1,2)--(1,1.5);
			\draw [black, line width=1.2pt] (0.5,1)--(1,1.5);
			\draw [black,line width=1.2pt] (1.5,1)--(1,1.5);
			\draw [black,line width=1.2pt] (0.5,1)--(0.5,0.5);
			\draw [black,line width=1.2pt] (0,0)--(0.5,0.5);
			\draw [black,line width=1.2pt] (1,0)--(0.5,0.5);
			\draw [black,line width=1.2pt] (1.5,0)--(1.5,1);
\end{tikzpicture}}}

\newcommand\tiagda{\resizebox{1cm}{1.3cm}{
		\begin{tikzpicture}
			\draw [black, line width=1.2 pt] (0,0)--(0,-1);
			\draw [black, line width=1.2pt] (0.5,0)--(1,-0.5);
			\draw [black,line width=1.2pt] (1.5,0)--(1,-0.5);
			\draw [black,line width=1.2pt] (1,-1)--(1,-0.5);
			\draw [black,line width=1.2pt] (1,-1)--(0,-2);
			\draw [black,line width=1.2pt] (0,-1)--(1,-2);
\end{tikzpicture}}}

\newcommand\tiagdb{\resizebox{1cm}{1.3cm}{
		\begin{tikzpicture}
			\draw [black, line width=1.2 pt] (1,0.5)--(0,1.5);
			\draw [black, line width=1.2pt] (0,0.5)--(1,1.5);
			\draw [black,line width=1.2pt] (1.5,0.5)--(1.5,1.5);
			\draw [black,line width=1.2pt] (1.5,0.5)--(1,0);
			\draw [black,line width=1.2pt] (1,0.5)--(1.5,0);
			\draw [black,line width=1.2pt] (0,0.5)--(0,0);
			\draw [black,line width=1.2pt] (0,0)--(0.5,-0.5);
			\draw [black,line width=1.2pt] (1,0)--(0.5,-0.5);
			\draw [black,line width=1.2pt] (0.5,-1)--(0.5,-0.5);
			\draw [black,line width=1.2pt] (1.5,0)--(1.5,-1);
\end{tikzpicture}}}

\newcommand\tiagea{\resizebox{1cm}{1.3cm}{
		\begin{tikzpicture}
			\draw [black, line width=1.2 pt] (-0.5,2)--(0.5,1);
			\draw [black, line width=1.2pt] (-0.5,1)--(0.5,2);
			\draw [black,line width=1.2pt] (0.5,1)--(0.5,0.5);
			\draw [black,line width=1.2pt] (0.5,0.5)--(1,0);
			\draw [black,line width=1.2pt] (0.5,0.5)--(0,0);
			\draw [black,line width=1.2pt] (-0.5,1)--(-0.5,0);
\end{tikzpicture}}}

\newcommand\tiageb{\resizebox{1cm}{1.3cm}{
		\begin{tikzpicture}
			\draw [black, line width=1.2 pt] (0.5,2.5)--(0.5,2);
			\draw [black, line width=1.2pt] (0.5,2)--(0,1.5);
			\draw [black,line width=1.2pt] (0.5,2)--(1,1.5);
			\draw [black,line width=1.2pt] (1.5,2.5)--(1.5,1.5);
			\draw [black,line width=1.2pt] (1.5,1.5)--(1,1);
			\draw [black,line width=1.2pt] (1,1.5)--(1.5,1);
			\draw [black,line width=1.2pt] (0,1.5)--(0,1);
			\draw [black,line width=1.2pt] (1,0)--(0,1);
			\draw [black,line width=1.2pt] (1,1)--(0,0);
			\draw [black,line width=1.2pt] (1.5,0)--(1.5,1);
\end{tikzpicture}}}

\newcommand\tiagfa{\resizebox{1cm}{1.3cm}{
		\begin{tikzpicture}
			\draw [black, line width=1.2 pt] (-0.5,2)--(0,1.5);
			\draw [black, line width=1.2pt] (0.5,2)--(0,1.5);
			\draw [black,line width=1.2pt] (0,1)--(0,1.5);
			\draw [black,line width=1.2pt] (1,2)--(1,1);
			\draw [black,line width=1.2pt] (0,1)--(1,0);
			\draw [black,line width=1.2pt] (1,1)--(0,0);
\end{tikzpicture}}}

\newcommand\tiagfb{\resizebox{1cm}{1.3cm}{
		\begin{tikzpicture}
			\draw [black, line width=1.2 pt] (-0.5,1.5)--(-0.5,0.5);
			\draw [black, line width=1.2pt] (1,0.5)--(0,1.5);
			\draw [black,line width=1.2pt] (0,0.5)--(1,1.5);
			\draw [black,line width=1.2pt] (-0.5,0.5)--(0,0);
			\draw [black,line width=1.2pt] (0,0.5)--(-0.5,0);
			\draw [black,line width=1.2pt] (1,0.5)--(1,0);
			\draw [black,line width=1.2pt] (-0.5,0)--(-0.5,-1);
			\draw [black,line width=1.2pt] (0,0)--(0.5,-0.5);
			\draw [black,line width=1.2pt] (1,0)--(0.5,-0.5);
			\draw [black,line width=1.2pt] (0.5,-0.5)--(0.5,-1);
\end{tikzpicture}}}

\newcommand\tiagga{\resizebox{1cm}{1.3cm}{
		\begin{tikzpicture}
			\draw [black, line width=1.2 pt] (0.5,2)--(1.5,1);
			\draw [black, line width=1.2pt] (0.5,1)--(1.5,2);
			\draw [black,line width=1.2pt] (0.5,1)--(0.5,0.5);
			\draw [black,line width=1.2pt] (0.5,0.5)--(0,0);
			\draw [black,line width=1.2pt] (0.5,0.5)--(1,0);
			\draw [black,line width=1.2pt] (1.5,1)--(1.5,0);
\end{tikzpicture}}}

\newcommand\tiaggb{\resizebox{1cm}{1.3cm}{
		\begin{tikzpicture}
			\draw [black, line width=1.2 pt] (-0.5,2.5)--(-0.5,1.5);
			\draw [black, line width=1.2pt] (0.5,2.5)--(0.5,2);
			\draw [black,line width=1.2pt] (0.5,2)--(1,1.5);
			\draw [black,line width=1.2pt] (0.5,2)--(0,1.5);
			\draw [black,line width=1.2pt] (-0.5,1.5)--(0,1);
			\draw [black,line width=1.2pt] (0,1.5)--(-0.5,1);
			\draw [black,line width=1.2pt] (1,1.5)--(1,1);
			\draw [black,line width=1.2pt] (0,0)--(1,1);
			\draw [black,line width=1.2pt] (1,0)--(0,1);
			\draw [black,line width=1.2pt] (-0.5,1)--(-0.5,0);
\end{tikzpicture}}}

\makeatletter
\@namedef{subjclassname@2020}{%
	\textup{2020} Mathematics Subject Classification}
\makeatother

\DeclareMathAlphabet{\mathpzc}{OT1}{pzc}{m}{it}

\title[Multiparameter colored partition category]{Multiparameter colored  partition category\\ and the product of the reduced Kronecker coefficients}

\author[Mazorchuk]{Volodymyr Mazorchuk}
\address{Uppsala Universitet, Sweden}
\email{mazor@math.uu.se}

\author[Srivastava]{Shraddha Srivastava}
\address{Uppsala Universitet, Sweden}
\email{maths.shraddha@gmail.com}

\begin{document}
	
	\begin{abstract}
		We introduce and study a multiparameter colored partition category  $\CPar(\textbf{x})$ by extending the construction of the partition category, over an algebraically closed field $\CC$ of characteristic zero and for a multiparameter $\textbf{x}\in \CC^{r}$. The morphism spaces in $\CPar(\textbf{x})$  have  bases in terms of partition diagrams whose parts are colored by  elements of the multiplicative cyclic group $C_r$. We show that the endomorphism spaces of $\CPar(\textbf{x})$  and additive Karoubi envelope of $\CPar(\textbf{x})$ are generically semisimple. The category $\CPar(\textbf{x})$ is rigid symmetric strict monoidal and we give a presentation of $\CPar(\textbf{x})$ as a monoidal category. The path algebra of $\CPar(\textbf{x})$ admits a triangular decomposition with Cartan subalgebra being equal to the direct sum of the group algebras of  complex reflection groups $G(r,n)$. We compute the structure constants for the classes of simple modules in the split Grothendieck ring of the category of modules over the path algebra of the downward partition subcategory of $\CPar(\textbf{x})$ in two ways. Among other things, this gives a closed formula for the product of the reduced Kronecker coefficients  in terms of the Littlewood--Richardson coefficients for $G(r,n)$ and certain Kronecker coefficients for the wreath product $(C_r  \times C_r)\wr S_n$. For $r=1$, this formula reduces to a formula for the reduced Kronecker coefficients given by Littlewood. We also give two analogues of the Robinson--Schensted correspondence for colored partition diagrams and, as an application, we classify the equivalence classes of Green's left, right and two-sided relations for the colored partition monoid in terms of these correspondences.
	\end{abstract}
	
	
	\maketitle
	\section{Introduction}
	Jones~\cite{Jones} and Martin~\cite{Martin}, independently, introduced partition algebras as generalizations of Temperley--Lieb algebras  to study the Potts model in non-planar statistical mechanics.  Let $\CC$ be an algebraically closed field of characteristic zero. For a nonnegative integer $l$ and $t\in \CC$, the partition algebra $\parti_l(t)$ has a basis in terms of certain graphs, called partition diagrams and the multiplication of two partition diagrams is given by concatenation up to a (certain power of ) the scalar parameter $t$.
	
	Symmetric groups play important roles in the study of the partition algebra in two different ways, one  as the Schur--Weyl dual of the partition algebra and the other as certain maximal subgroups embedded into the partition algebra.  For a nonnegative integer $n$, $\parti_l(n)$ acts on $(\CC^n)^{\otimes l}$, where $\CC^n$ is the defining representation of the symmetric group $S_n$. Moreover, these actions of $\parti_l(n)$ and the group algebra $\CC[S_n]$ generate each others centralizers. This is referred to as the Schur--Weyl duality between the partition algebra $\parti_l(n)$ and the symmetric group $S_n$. For any $t\in\CC$, the partition algebra $\parti_l(t)$ contains the symmetric group $S_i$, for $0\leq i\leq l$, as a subgroup. Martin's works, e.g. see~\cite{Mar96}, exploit these copies of the symmetric groups inside $\parti_l(t)$ for classifying simple modules and determining the structure of $\parti_l(t)$. The partition algebra $\parti_l(t)$ can be viewed as the twisted semigroup algebra of the partition monoid whose maximal subgroups are exactly these symmetric group $S_i$, for $0\leq i\leq l$, see  \cite{Wilcox}. From the perspective of the Schur--Weyl duality, there have been various generalizations of partition algebras, see for example,~\cite{Tanabe,Bloss,Kosuda,MS20,CPV}. This manuscript focuses on generalizing the other role played by the symmetric groups in the context of partition algebras. 
	
	Let $t\in\CC$. In~\cite{Deligne}, Deligne defined a monoidal category, depending on $t$ as a parameter, which, in some sense, serves as an interpolation of the categories of representations of the symmetric groups $S_n$, for the special cases when $t=n$. In~\cite{CO}, Deligne's category was realized as the additive Karoubi envelope of the partition category $\Par(t)$, for the same $t$. The objects of $\Par(t)$ are indexed by nonnegative integers. The morphisms in $\Par(t)$ are given by partition diagrams. The endomorphism of an object $l$ in $\Par(t)$ is the partition algebra $\parti_{l}(t)$.  In~\cite{Comes}, one can find a diagrammatic approach to $\Par(t)$ as a rigid symmetric strict monoidal category generated by a Frobenius object of categorical dimension $t$ and also a concise and explicit presentation of $\Par(t)$. The partition category is an example of a so called diagram category. A unified approach with rigorous proofs for finding  presentations for diagram categories as monoidal categories was given in~\cite{East20}.
	
	Motivated by the triangular decomposition of the universal enveloping algebra of a finite-dimensional complex semisimple Lie algebra, \cite{Strop} proposed a notion of  triangular decomposition for a locally finite-dimensional and locally unital algebra with a fixed choice of pairwise orthogonal idempotents. In such a setup, one has natural analogues of the notion of a Cartan subalgebra, a Borel subalgebra, and also an  analogue of a Verma module (which is called a standard module in this setup) which have simple quotients giving rise to  a classification of simple modules. The path algebra of $\Par(t)$ is a locally finite-dimensional and locally unital algebra which admits a triangular decomposition, see~\cite{Brv}. The Cartan subalgebra of the path algebra of $\Par(t)$ is the direct sum of the group algebras of all symmetric groups and the (positive) Borel subalgebra is the path algebra of the  downward partition category.
	
	The representation category of a $\CC$-linear category $\mathcal{C}$ is the functor category whose objects are $\CC$-linear covariant functors from $\mathcal{C}$ to the category of  vector spaces over $\CC$. The representation category is equivalent to the category of all locally finite-dimensional modules over the path algebra of $\mathcal{C}$. Since the partition category and the downward partition category are monoidal, their representation categories are also monoidal via the Day convolution. So, the categories of locally finite-dimensional modules over their path algebras are monoidal. The split Grothendieck group of the category of finitely generated projective modules over the path algebra of the downward partition category admits the structure of a ring. Similarly, we have the Grothendieck ring of the category of all finite-dimensional modules over the path algebra of the downward partition category and the Grothendieck ring of the category of finitely generated modules admitting a filtration by the standard modules. Furthermore, the Grothendieck ring of the category of finitely generated projective modules over the path algebra of the partition category and the aforementioned Grothendieck rings can be naturally identified with the ring of symmetric functions, see \cite[Theorem 3.12]{Brv}. In the next paragraph, we elaborate some more on the images of some modules in appropriate Grothendieck rings and their connection to  symmetric functions.
	
	The images of indecomposable projective modules over the path algebra of the downward partition category in the split Grothendieck ring correspond to the Schur functions. It is well-known that the structure constants for the basis consisting of Schur functions are given by the Littlewood--Richardson coefficients. So it follows that the structure constants for the basis consisting of the images of these indecomposable projective modules are given by the Littlewood--Richardson coefficients as well. Moreover,  the images of simple modules over the path algebra of the downward partition category, in the Grothendieck ring, correspond to  the deformed Schur functions. The latter symmetric functions were defined in \cite{OZ} such that the structure constants for the basis consisting of these functions are given by the reduced Kronecker coefficients (these are the limits of certain Kronecker coefficients, see Section~\ref{sec:wreathfun}). Utilizing Littlewood's formula \cite{Littlewood} for the reduced Kronecker coefficients together with Lemma~\cite[Lemma 3.7]{Brv},  in \cite[Theorem 3.8]{Brv} the structure constants for the basis given by the images of the simple modules were given by the reduced Kronecker coefficients.  The structure constants for the images of the standard modules over the path algebra of the partition category were also given by the reduced Kronecker coefficients, see~\cite[Theorem 3.11]{Brv}. For Deligne category, the latter result was obtained in~\cite{Inna}. A similar result for partition algebras can be found in \cite{BMR}. Moreover, utilizing the Schur--Weyl duality between the partition algebra and the symmetric group, \cite{BMR} found a closed positive formula for the  Kronecker coefficients in certain cases.

	A complex reflection group is a finite group generated by pseudo-reflections on a finite-dimensional complex vector space. Shephard and Todd gave a classification of irreducible complex reflection groups in \cite{Shephard}. They are classified by an infinite family of groups $G(r,p,n)$ depending on three parameters $r,p$ and $n$ such that $p$ is a divisor of $r$, together with $34$ ``sporadic'' groups (which are numbered from $4$ to $37$ in \cite{Shephard}). When $r=1$, the group $G(r,p,n)$ is the symmetric group $S_n$. When $p=1$, the group $G(r,p,n)$ is the wreath product $G(r,n)=C_r\wr S_n$, where $C_r$ is the multiplicative cyclic group of order $r$. Diagrammatically, an element of $S_n$ can be represented as a graph called a permutation diagram. Then an element in $G(r,n)$ is a colored permutation diagram, i.e., a permutation diagram whose edges are labelled by elements of $C_r$.  
	
	A permutation diagram is a partition diagram, and we show that labelling of connected components of any partition diagram by the elements of $C_r$ gives rise to an associative algebra and a category which generalize the constructions of the partition algebra and the partition category, respectively. These two generalizations are the main objects of the study in this manuscript.

	{\bf Multiparameter colored partition category.} Let ${\bf x}=(x_0,x_1,\ldots,x_{r-1})\in\CC^r$. In Section~\ref{sec:multicat}, we define the mutliparameter colored partition category $\CPar({\bf x})$.  A morphism in $\CPar({\bf x})$ is a linear combination of colored partition diagrams, i.e., its connected components are labelled by elements of $C_r$. The category $\CPar({\bf x})$ contains the partition category $\Par({x_0})$ as a wide subcategory. As a monoidal category, the category $\CPar({\bf x})$ is generated by a special commutative Frobenius object of categorical dimension $x_0$ with a particular choice of an order $r$ automorphism of this object. In Proposition~\ref{prop:present}, we give a presentation of $\CPar({\bf x})$ as monoidal category. The endomorphism algebra of an object $l$ of  $\CPar({\bf x})$ is called the multiparameter colored partition algebra corresponding to the nonnegative integer $l$ and the multiparameter ${\bf x}$. The multiparameter colored partition algebra can also be viewed as  the twisted semigroup algebra of the colored partition monoid. The maximal subgroups of this monoid are the complex reflection groups $G(r,n)$, for $n\leq l$.  In Proposition~\ref{prop:presentalg}, we give a presentation for this monoid. 
	
	{\bf Triangular decomposition of the path algebra.} In Theorem~\ref{thm:td}, we show that the path algebra of $\CPar({\bf x})$ admits a triangular decomposition in the sense of \cite[Definition 5.31]{Strop}. The Cartan subalgebra is given by the direct sum of the group algebras of the complex reflection group $G(r,n)$, for all nonnegative integers $n$. The positive Borel subalgebra is the path algebra of the colored downward partition category. Since the representation theory of $G(r,n)$ over $\CC$ is well-understood, this triangular decomposition together with \cite[Theorem 5.38]{Strop} allows us to describe standard modules and to classify the simple modules of the path algebra of $\CPar({\bf x})$, see Proposition~\ref{prop:upperfinite}.  
	
	{\bf Grothendieck rings and connection to the reduced Kronecker coefficients, wreath product symmetric functions.}
	As discussed previously, various Grothendieck rings related to the partition category are identified with the ring of symmetric functions. When we consider the analogue of those Grothendieck groups for the multiparameter colored partition category, the role of symmetric functions is replaced by the so called wreath product symmetric functions defined in Section~\ref{sec:wreathfun}.  It is well-known, see~\cite[(7.3)]{Mac}, that the ring of symmetric functions is isomorphic, via the Frobenius characteristic map, to the direct sum of the character rings of all symmetric groups.  In~\cite{Remmel}, the Frobenius characteristic map for wreath product groups was defined for which the irreducible characters map to the so called wreath product Schur functions (this terminology is addressed in \cite{Wreathsym}). A slight variation of the same definition was given in~\cite[Appendix B]{Mac}. In both cases, the Frobenius characteristic map provides an isomorphism between to the direct sum of the character rings of $G(r,n)$, for all $n$, and the $r$-fold tensor product of the ring of symmetric functions.
	
	We  consider the split Grothendieck group of the  category of all finitely generated projective modules over the path algebra of the colored downward partition category, the Grothendieck group of the category of all finite-dimensional modules over the path algebra of the colored downward partition category, and the split Grothendieck group of the category of all finitely generated projective modules over the path algebra of the multiparameter colored partition category. All these Grothendieck groups are, in fact, rings and  in Proposition~\ref{prop:lambdaring}, we conclude that these Grothendieck rings are isomorphic to the $r$-fold tensor product of the ring of symmetric functions. The images of indecomposable projective modules over the path algebra of colored downward partition category correspond to wreath product Schur functions defined in Section~\ref{sec:wreathfun}. In Proposition~\ref{prop:projinde}, we observe that the structure constants for the basis given by the images of indecomposable projective modules over the path algebra of the colored downward partition category are described in terms of products of the Littlewood--Richardson coefficients. The images of simple modules over the path algebra of the colored downward partition category correspond to the wreath product deformed Schur functions. In Proposition~\ref{prop:struct} (respectively Proposition~\ref{prop:delta}), the structure constants for the basis given by the  simple modules (respectively standard modules) over the path algebra of the colored downward partition category (respectively the multiparameter colored partition category) are given as products of the reduced Kronecker coefficients. The proof of the latter result is a consequence of Proposition~\ref{thm:basesym} and an analysis of the Cartan matrix, which we address next.

	In Lemma~\ref{lm:cartan}, the entries of the Cartan matrix for the path algebra of the colored downward partition category are given by the dimensions of certain submodules of the path algebra of the colored downward partition category involving certain primitive idempotents for the group algebras of complex reflection groups. In Corollary~\ref{cor:cartantensor}, we prove that this Cartan matrix is an $r$-fold tensor power of the Cartan matrix of the path algebra of the downward partition category. Our key ingredient in this proof is an alternative description of the path algebra of the colored downward partition category. The paper \cite{VS} introduced a groupoid whose path algebra is isomorphic to the group algebra of the complex reflection group $G(r,n)$. In Section~\ref{sec:downalt}, we generalize this and construct a category $\mathcal{G}$ whose objects are colored dots and morphism are color preserving downward partition diagrams. The groupoid constructed in~\cite{VS} is a subcategory of the category $\mathcal{G}$. In Theorem~\ref{thm:groupoid}, we show that the path algebra of $\mathcal{G}$ is isomorphic to the path algebra of the colored downward partition category.

	The structure constants for the basis given by the simple modules over the path algebra of the colored downward partition category are also computed in Theorem~\ref{thm:multi} by a method used in the proof of~\cite[Theorem 3.8]{Brv}. Combining Proposition~\ref{prop:struct} and Theorem~\ref{thm:multi}, in Theorem~\ref{thm:formula}, we obtain a formula for the $r$-fold product of the reduced Kronecker coefficients in terms of products of Littlewood--Richardson coefficients and Kronecker coefficients for the wreath product $(C_r\times C_r)\wr S_n$. For $r=1$, this formula reduces to Littlewood's formula \cite{Littlewood} for the reduced Kronecker coefficients.

	{\bf Generic semisimplicity.} In~\cite{Martin}, partition algebras were shown to be semisimple for all but finitely many values of the parameter. In Theorem~\ref{cor:gens}, we prove that the multiparameter colored partition algebras are generically semisimple. This means that there exists a polynomial in $r$ variables such that the multiparameter colored partition algebras are semisimple, for all values of the parameters ${\bf x}\in\CC^r$ outside the zero set of this polynomial. For $r=1$, the additive Karoubi envelope of $\CPar({\bf x})$ is the Deligne category for symmetric groups. This one is known to be semisimple if $x_0$ is not a nonnegative integer~\cite{Deligne}. Over the field $\C$ of complex numbers, in Theorem~\ref{thm:karoubi}, we prove generic semisimplicity of the additive Karoubi envelope of $\CPar({\bf x})$. This means that there exists a measure zero subset of $\C^r$ such that, for all ${\bf{x}}\in\C^r$ which are outside this measure zero subset, the additive Karoubi envelope of $\CPar({\bf x})$ is semisimple. 
	
	{\bf Analogues of the Robinson--Schensted correspondence.}  The Robinson--Schensted correspondence~\cite{Sch} is a bijection between permutations and pairs of standard Young tableaux of the same shape. Over the years, this was generalized to several setups. In~\cite{Stanton}, a bijection between elements of $G(r,n)$ and pairs of $r$-tuples of tableaux of the same shape was given. In~\cite{Shimozono}, a bijection between elements of $G(r,n)$ and pairs of $r$-ribbon tableaux of the same shape was given. For partition diagrams, such a bijection was first described in~\cite{Rollet}. Later on, a bijection between partition diagrams and pairs of so called set-partition tableaux was also given in~\cite{TT}. More recently, in~\cite{Colmen}, a bijection was given involving two row arrays of multisets and multiset-partition tableaux such that it recovers the bijection in~\cite{TT}.
	
	We give two such bijections for the colored partition diagrams. 
	The first bijection, given in Proposition~\ref{prop:Rs}, involves $r$-tuples of standard set-partition tableaux. The second bijection, given in Proposition~\ref{prop:SW}, involves standard set-partition $r$-ribbon tableaux. As an application of these bijections, in Proposition~\ref{prop-new-45}, we characterize equivalence classes of Green's relations for the colored partition monoid, see \cite{Green} for more information
	about Green's relations. 
	
	\section{Preliminaries}
	In this section, we give a brief overview of partition diagrams, triangular decomposition of a locally finite-dimensional locally unital algebra  with a fixed choice of pairwise orthogonal idempotents and a result on classification of simple modules of it, induction product, triangular structure of a category, necessary results on wreath product groups and wreath product symmetric functions.
	
	\subsection{Partition diagram}
	Let $k$ and $l$ be nonnegative integers. A set-partition of the set $\{1,2,\ldots,k,1',2',\ldots,l'\}$ is a set consisting of nonempty disjoint subsets of $\{1,2,\ldots,k,1',2',\ldots,l'\}$ whose union is the whole set $\{1,2,\ldots,k,1',2',\ldots,l'\}$. When both $k$ and $l$ are equal to zero, then by the convention empty set is the only set-partition. We represent a set-partition, for which either $k\neq 0$ or $l\neq 0$, by an unoriented graph whose vertices are drawn in two rows. There are $k$ vertices in the top row and they are  indexed by $1,2,\ldots, k$. There are $l$ vertices in the bottom row and they are indexed by $1',2',\ldots,l'$. Edges in the graph are chosen such that the connected components of the graph give exactly the parts of our partition. Equivalently, there is a path between two vertices if and only if they lie in the same subset in the set-partition.
	Such a graph is called a $(k,l)$-partition diagram. Usually, there are more than one graph representing the same set-partition. We say that two such graphs are equivalent provided that they represent the same set-partition.

	\begin{example}
		For $k=7$ and $l=8$, the following graph represents the set-partition 
		\begin{displaymath}
			\{\{1,2'\},\{1',2\},\{3,4,4',5'\},\{3',5\},\{6,7',8'\},\{7\},\{6'\}\}:
		\end{displaymath}
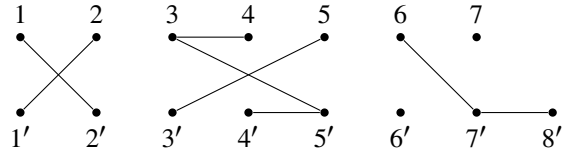
		\begin{align*} 
			\begin{tikzpicture}[scale=1,mycirc/.style={circle,fill=black, minimum size=0.1mm, inner sep = 1.1pt}]
				\node[mycirc,label=above:{$1$}] (n1) at (0,1) {};
				\node[mycirc,label=above:{$2$}] (n2) at (1,1) {};
				\node[mycirc,label=above:{$3$}] (n3) at (2,1) {};
				\node[mycirc,label=above:{$4$}] (n4) at (3,1) {};
				\node[mycirc,label=above:{$5$}] (n5) at (4,1) {};
				\node[mycirc,label=above:{$6$}] (n6) at (5,1) {};
				\node[mycirc,label=above:{$7$}] (n7) at (6,1) {};
				\node[mycirc,label=below:{$1'$}] (n1') at (0,0) {};
				\node[mycirc,label=below:{$2'$}] (n2') at (1,0) {};
				\node[mycirc,label=below:{$3'$}] (n3') at (2,0) {};
				\node[mycirc,label=below:{$4'$}] (n4') at (3,0) {};
				\node[mycirc,label=below:{$5'$}] (n5') at (4,0) {};
				\node[mycirc,label=below:{$6'$}] (n6') at (5,0) {};
				\node[mycirc,label=below:{$7'$}] (n7') at (6,0) {};
				\node[mycirc,label=below:{$8'$}] (n8') at (7,0) {};
				\path[-, draw](n1) to (n2');
				\path[-,draw](n2) to (n1');
				\path[-,draw](n3) to (n4);
				\path[-,draw](n5) to (n3');
				\path[-,draw](n4') to (n5');
				\path[-,draw](n7') to (n8');
				\path[-,draw](n3) to (n5');
				\path[-,draw](n6) to (n7');
			\end{tikzpicture}
		\end{align*}
	\end{example}
	
	\subsubsection{Multiplication} Given a $(k,l)$-partition diagram $d_1$ and $(l,m)$-partition diagram $d_2$, the composition $d_1\circ d_2$ is a $(k,m)$-partition diagram obtained as follows. Put $d_1$ on the top of $d_2$ and 
	identify the vertices of the bottom row of $d_1$ with the vertices of the top row of $d_2$. Then the multiplication $d_1\circ d_2$ is obtained by the concatenation, reading from the bottom to the top, where we ignore the components lying entirely in the middle.
	
	\begin{example} Let $k=7$ and $l=m=8$. For the following $d_1$ and $d_2$
		
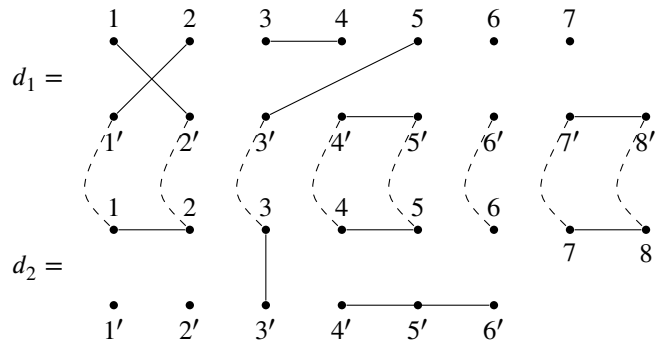
\begin{align*} 
			\begin{tikzpicture}[scale=1,mycirc/.style={circle,fill=black, minimum size=0.1mm, inner sep = 1.1pt}]
				\node (1) at (-1,0.5) {$d_1 =$};
				\node[mycirc,label=above:{$1$}] (n1) at (0,1) {};
				\node[mycirc,label=above:{$2$}] (n2) at (1,1) {};
				\node[mycirc,label=above:{$3$}] (n3) at (2,1) {};
				\node[mycirc,label=above:{$4$}] (n4) at (3,1) {};
				\node[mycirc,label=above:{$5$}] (n5) at (4,1) {};
				\node[mycirc,label=above:{$6$}] (n6) at (5,1) {};
				\node[mycirc,label=above:{$7$}] (n7) at (6,1) {};
				\node[mycirc,label=below:{$1'$}] (n1') at (0,0) {};
				\node[mycirc,label=below:{$2'$}] (n2') at (1,0) {};
				\node[mycirc,label=below:{$3'$}] (n3') at (2,0) {};
				\node[mycirc,label=below:{$4'$}] (n4') at (3,0) {};
				\node[mycirc,label=below:{$5'$}] (n5') at (4,0) {};
				\node[mycirc,label=below:{$6'$}] (n6') at (5,0) {};
				\node[mycirc,label=below:{$7'$}] (n7') at (6,0) {};
				\node[mycirc,label=below:{$8'$}] (n8') at (7,0) {};
				\path[-, draw](n1) to (n2');
				\path[-,draw](n2) to (n1');
				\path[-,draw](n3) to (n4);
				\path[-,draw](n5) to (n3');
				\path[-,draw](n4') to (n5');
				\path[-,draw](n7') to (n8');
				\node (1) at (-1,-2) {$d_2 =$};
				\node[mycirc,label=above:{$1$}] (n7) at (0,-1.5) {};
				\node[mycirc,label=above:{$2$}] (n8) at (1,-1.5) {};
				\node[mycirc,label=above:{$3$}] (n9) at (2,-1.5) {};
				\node[mycirc,label=above:{$4$}] (n10) at (3,-1.5) {};
				\node[mycirc,label=above:{$5$}] (n11) at (4,-1.5) {};
				\node[mycirc,label=above:{$6$}] (n12) at (5,-1.5) {};
				\node[mycirc,label=below:{$7$}] (n13) at (6,-1.5) {};
				\node[mycirc,label=below:{$8$}] (n14) at (7,-1.5) {};
				\node[mycirc,label=below:{$1'$}] (n7'') at (0,-2.5) {};
				\node[mycirc,label=below:{$2'$}] (n8'') at (1,-2.5) {};
				\node[mycirc,label=below:{$3'$}] (n9') at (2,-2.5) {};
				\node[mycirc,label=below:{$4'$}] (n10') at (3,-2.5) {};
				\node[mycirc,label=below:{$5'$}] (n11') at (4,-2.5) {};
				\node[mycirc,label=below:{$6'$}] (n12') at (5,-2.5) {};
				\path[-,draw] (n7) to (n8);
				\path[-,draw] (n9) to (n9');
				\path[-,draw] (n10) to (n11);
				\path[-,draw] (n13) to (n14);
				\path[-,draw] (n10') to (n12');
				\draw[dashed] (n1') .. controls(-0.5,-1)..(n7);
				\draw[dashed] (n2') .. controls(0.5,-1)..(n8);
				\draw[dashed] (n3') .. controls(1.5,-1)..(n9);
				\draw[dashed] (n4') .. controls(2.5,-1)..(n10);
				\draw[dashed] (n5') .. controls(3.5,-1)..(n11);
				\draw[dashed] (n6') .. controls(4.5,-1)..(n12);
				\draw[dashed] (n7') .. controls(5.5,-1)..(n13);
				\draw[dashed] (n8') .. controls(6.5,-1)..(n14);
			\end{tikzpicture}
		\end{align*}
		the multiplication $d_1\circ d_2$ is:
		
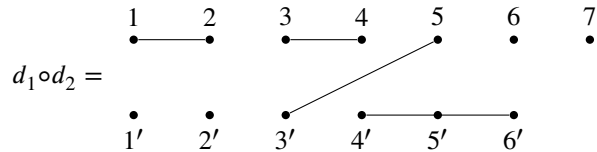
\begin{align*}
			\begin{tikzpicture}[scale=1,mycirc/.style={circle,fill=black, minimum size=0.1mm, inner sep = 1.1pt}]
				\node (1) at (-1,0.5) {$d_1\circ d_2 =$};
				\node[mycirc,label=above:{$1$}] (n1) at (0,1) {};
				\node[mycirc,label=above:{$2$}] (n2) at (1,1) {};
				\node[mycirc,label=above:{$3$}] (n3) at (2,1) {};
				\node[mycirc,label=above:{$4$}] (n4) at (3,1) {};
				\node[mycirc,label=above:{$5$}] (n5) at (4,1) {};
				\node[mycirc,label=above:{$6$}] (n6) at (5,1) {};
				\node[mycirc,label=above:{$7$}] (n7) at (6,1) {};
				\node[mycirc,label=below:{$1'$}] (n7'') at (0,0) {};
				\node[mycirc,label=below:{$2'$}] (n8'') at (1,0) {};
				\node[mycirc,label=below:{$3'$}] (n9') at (2,0) {};
				\node[mycirc,label=below:{$4'$}] (n10') at (3,0) {};
				\node[mycirc,label=below:{$5'$}] (n11') at (4,0) {};
				\node[mycirc,label=below:{$6'$}] (n12') at (5,0) {};
				\path[-,draw] (n1) to (n2);
				\path[-,draw] (n3) to (n4);
				\path[-,draw] (n5) to (n9');
				\path[-,draw] (n10') to (n12');
			\end{tikzpicture}
		\end{align*}
	\end{example}
	\subsection{ Triangular decomposition}
	Let $A$ be a locally finite-dimensional, locally unital algebra. Let $I$ be a nonempty set. 
	Let $\{e_i\in A\mid i\in I\}$ be a distinguished set of pairwise orthogonal idempotents such that
	\begin{displaymath}
		A=\bigoplus_{i, j\in I} e_i A e_j.
	\end{displaymath}
	Then a triangular decomposition (see \cite[Definition 5.31]{Strop}) of $A$ consists of the following data:
	\begin{enumerate}
		\item [(TD1)] an upper finite poset $(\Lambda, \leq )$, i.e., a poset such that  the principal filter $\{b\in\Lambda\mid a\leq b \}$ is finite, for every $a\in\Lambda$;
		
		\item[(TD2)] a function $\partial:I\to \Lambda$ with finite fibres;
		
		\item[(TD3)] locally unital subalgebras $A^\flat$ and $A^\sharp$, each containing all $e_i$, for $i\in I$;
	\end{enumerate}
	satisfying the following axioms:
	\begin{enumerate}
		\item[(TD4)] for $A^0=A^{\flat}\cap A^{\sharp}$, the algebra $A^{\flat}$ is  projective as a right $A^0$-module and the algebra $A^{\sharp}$ is  projective as a left $A^{0}$-module;
		
		\item[(TD5)] the multiplication map $A^{\flat}\otimes_{A^{0}}A^{\sharp}\to A$ is a linear isomorphism;
		
		\item[(TD6)] for $i, j\in I$, the spaces $e_jA^\flat e_i$ and $e_iA^{\sharp} e_j$ are zero unless
		$\partial(i)\leq \partial(j)$, moreover, if $\partial(i)=\partial(j)$, then  $e_iA^\flat e_j=e_iA^\sharp e_j=e_i A^{0} e_j$.
	\end{enumerate}
	
	The subalgebras $A^{\flat}$, $A^{\sharp}$, and $A^{0}$ are called the negative Borel, the positive Borel, and the Cartan subalgebras of $A$, respectively.  A triangular structure on $A$ enables to define two natural functors (\cite[p. 100]{Strop}) which we discuss in the next two sections. Let $A\Mod$ and $A\fdMod$ denote the category of locally finite-dimensional left $A$-modules and the category of finite-dimensional left $A$-modules, respectively.
	
	\subsubsection{Global standardization functor}\label{sub:stand}
	Using (TD6), we have the natural projection map
	\begin{displaymath}
		\pi^\sharp: A^\sharp\twoheadrightarrow A^{0},
	\end{displaymath}
	whose kernel is given by the direct sum  of all 
	$e_iA^\sharp e_j$, for which $\partial(i)<\partial(j)$.
	This map is a homomorphism  of locally unital algebras.
	Thus we have the following functor given by inflation along the map $\pi^\sharp$:
	\begin{displaymath}
		\infl^{\sharp}: A^{0}\fdMod \rightarrow A^{\sharp}\Mod.
	\end{displaymath}
	The global standardization functor
	\begin{align}\label{al:stdfunct}
		j_{!}: A^0\fdMod \longrightarrow A\Mod
	\end{align}
	is defined as the composition
	\begin{displaymath}
		j_{!}= \Ind^A_{A^\sharp}\circ \infl^{\sharp}, 
	\end{displaymath}
	where $\Ind^{A}_{A^{\sharp}}:=A\otimes_{A^{\sharp}}-$. The functor $\infl^{\sharp}$ is obviously exact and, from the combination of (TD4) and (TD5), we see that $A$ is projective as a right module over $A^{\sharp}$, so the functor $j_{!}$ is exact.
	\subsubsection{Global costandardization functor}\label{sub:costand}
	Similarly, we have a surjective homomorphism of locally unital algebras
	\begin{displaymath}
		\pi^\flat: A^\flat\twoheadrightarrow A^{0}.
	\end{displaymath}
	This gives rise to the following functor given by inflation along the map $\pi^\sharp$:
	\begin{displaymath}
		\infl^{\flat}: A^{0}\fdMod \rightarrow A^{\flat}\Mod.
	\end{displaymath}
	The global costandardization functor
	\begin{align}\label{al:costdfunct}
		j_\ast: A^0\fdMod \longrightarrow A\Mod
	\end{align}
	is defined as the composition
	\begin{displaymath}
		j_\ast= \Coind^{A}_{A^\flat}\circ\infl^{\flat},
	\end{displaymath}
	where $\Coind^{A}_{A^{\flat}}= \displaystyle{\bigoplus_{i\in I}}\Hom_{A^{\flat}}(Ae_i,-)$. The functor $\infl^{\flat}$ is obviously exact and from the combination of (TD4) and (TD5) we see that $A$ is projective as a left $A^{\flat}$-module, so the functor $j_\ast$ is exact.
	
	For a module $M$, let $\hd M$ and $\soc M$ denote the head (or top) and the socle of $M$, respectively. The following theorem is given in \cite[Thereom 5.38]{Strop} in more general setup. Here we write the part that is relevant for this paper.
	
	\begin{theorem}\label{thm:classification}
		Let $A$ be a locally finite-dimensional algebra admitting a triangular decomposition. 
		Let, further, $\{V(\lambda)\mid \lambda\in \B \}$ be a complete and irredudant set of simple $A^0$-modules. Then 
		\begin{align*}
			\{L(\lambda):=\hd j_{!}(V(\lambda))\cong \soc(j_{*} V(\lambda))\mid \lambda\in\B\}
		\end{align*}
		is a complete and irredundant set of simple $A$-modules.
	\end{theorem}
	\subsection{Triangular structure}\label{sec:triancat}
	Recall that a subcategory of a category is called wide if it contains all the objects of the original category. A triangular structure on a $\CC$-linear category is  defined in \cite{SamSnow}.
	
	Let $\mathfrak{G}$ be a $\CC$-linear category satisfying the following condition:
	\begin{enumerate}
		\item [(T0)] $\mathfrak{G}$ is an essentially small category with finite-dimensional spaces of homomorphisms.
	\end{enumerate}
	A pair $(\mathfrak{U},\mathfrak{D})$ of wide subcategories of $\mathfrak{G}$ is called a {\em triangular structure} if the following conditions are satisfied:
	\begin{enumerate}
		\item [(T1)] For all objects $x$ in $\mathfrak{G}$, $\End_{\mathfrak{U}}(x)=\End_{\mathfrak{D}}(x)$ and it is a semisimple ring.
		\item[(T2)] There exists a partial order on the set $\lvert\mathfrak{G}\rvert$ of isomorphism classes of objects in $\mathfrak{G}$ with the following properties:
		\begin{enumerate}[$($a$)$]
			\item For all $x\in\lvert\mathfrak{G}\rvert$, there are only finitely many $y\in\lvert\mathfrak{G}\rvert$ such that $y\leq x$.
			\item The category $\mathfrak{U}$ is upwards with respect to $\leq$, i.e., if there is a nonzero morphism from $x$ to $y$ in $\mathfrak{U}$, then $x\leq y$.
			\item The category $\mathfrak{D}$ is downwards with respect to $\leq$, i.e., if there is a nonzero morphism from $x$ to $y$ in $\mathfrak{D}$, then $x\geq y$.
		\end{enumerate}
		\item[(T3)] For all objects $x,z$ in $\mathfrak{G}$, the natural map
		\begin{displaymath}
			\bigoplus_{y\in\lvert \mathfrak{G}\rvert}\Hom_{\mathfrak{U}}(y,z)\otimes_{\End_{\mathfrak{U}}(y)}\Hom_{\mathfrak{D}}(x,y) \rightarrow \Hom_{\mathfrak{G}}(x,z)
		\end{displaymath}
		is an isomorphism.
	\end{enumerate}
	\subsection{Induction product}\label{sec:indupro}
	Let $\mathcal{C}$ be a $\CC$-linear essentially small category. Then the path algebra of $\mathcal{C}$ is defined as 
	\begin{align}\label{al:pathalg}
		C:=\bigoplus_{X,Y\in \Ob(\mathcal{C})}\Hom_{\mathcal{C}}(X,Y),
	\end{align}
	with the multiplication induced by the composition of morphisms in $\mathcal{C}$.
	For $X$ in the object class $\Ob(\mathcal{C})$ of $\mathcal{C}$, let $1_{X}$ denote the identity morphism on $X$. Then $C$ admits the following decomposition
	\begin{displaymath}
		C:=\bigoplus_{X,Y\in\Ob(\mathcal{C})}1_{X} C 1_{Y}.
	\end{displaymath}
	Denote by $\Vect_{\CC}$ the category of vector spaces over $\CC$.  Let $\Rep(\mathcal{C})$ denote the category of representations of $\mathcal{C}$, i.e., $\CC$-linear covariant functors from $\mathcal{C}$ to $\Vect_{\CC}$. Then the  functor
	\begin{equation}\label{al:equiv}
		\begin{array}{ccc}
			\Rep(\mathcal{C}) &\to& C\BMod; \\
			F &\mapsto& 
			\displaystyle\bigoplus_{X\in\Ob(\mathcal{C})} F(X), 
		\end{array}
	\end{equation}
	is an equivalence of categories.
	Furthermore, assume that the category $\mathcal{C}$ has a strict monoidal structure given by the bifunctor  $\star$.
	Then $\Rep(\mathcal{C})$ inherits from $\mathcal{C}$ a monoidal structure given by the bifunctor which we denote by $\Cstar$ such that the following contravariant Yoneda embedding becomes a monoidal functor:
	\begin{align*}
		\mathcal{C}&\hookrightarrow \Rep{(\mathcal{C})}\\
		X&\mapsto \Hom_{\mathcal{C}}(X,-). 
	\end{align*}
	This gives:
	\begin{align}\label{al:repfunct}
		\Hom_{\mathcal{C}}(X,-)\,\Cstar\,\Hom_{\mathcal{C}}(Y,-):=\Hom_{\mathcal{C}}(X\star Y,-).
	\end{align}
	From the Yoneda lemma it follows that  every object $F$ in $\Rep(\mathcal{C})$ can be written as a colimit of representable functors $\Hom_{\mathcal{C}}(X,-)$, where $X$ in $\Rep(\mathcal{C})$, see \cite[III.7, Theorem 1]{MacLane}:
	\begin{displaymath}
		F\cong \underset{\Hom_{\mathcal{C}}(X,-)\to F}{\colim} \Hom_{\mathcal{C}}(X,-).
	\end{displaymath}
	Using this, for $F,G\in\Rep{\mathcal{C}}$, we have:
	\begin{displaymath}
		F\,\Cstar \,G:=\underset{\Hom_{\mathcal{C}}(X,-)\to F}{\colim}\, \left(\underset{\Hom_{\mathcal{C}}(Y,-)\to G}{\colim}\Hom_{\mathcal{C}}(X\star Y,-)\right).
	\end{displaymath}
	
	For the modules over the path algebra $C$, the corresponding  monoidal structure can also be described. Indeed, the monoidal structure $\Cstar$ on $\Rep{\mathcal{C}}$ can be transported to the category $C\BMod$ via the equivalence \eqref{al:equiv}. We denote the bifunctor giving the monoidal structure on $C\BMod$ also by $\Cstar$. Let $M$ and $N$ be two left modules over $C$. 
	Note that $M\otimes N$ is a module over $C\otimes C$. From \cite{Hovey} (also see \cite{XAC}), there exists a $C$-$C\otimes C$-bimodule $B$ such that
	\begin{align}\label{al:cstar}
		M\cstar N = B\otimes_{C\otimes C} (M\otimes N).
	\end{align}
	For the specific values $M=N=C$, we get $B=C\, \Cstar \, C$. 
	
	Under the equivalence \eqref{al:equiv}, $C$ corresponds to $\displaystyle{\bigoplus_{X\in\Ob{\mathcal{C}}}}\Hom_{\mathcal{C}}(X,-)$. Then, using \eqref{al:repfunct} and applying the functor in \eqref{al:equiv}, it follows that
	\begin{displaymath}
		B=\bigoplus_{X,Y,Z\in\Ob{\mathcal{C}}}\Hom_{\mathcal{C}}(X\star Y, Z).
	\end{displaymath}
	Note that $B$ is equal to $C1_{\star}:=\displaystyle{\bigoplus_{X,Y\in\Ob{\mathcal{C}}}}C1_{X\star Y}$ so that  \eqref{al:cstar} can be rewritten as follows:
	\begin{align}\label{al:rcstar}
		M\, \Cstar \, N=C1_{\star}\otimes_{C\otimes C}(M\otimes N),
	\end{align}
	and it is referred to as the induction product in \cite[Section 2.4]{Brv}. 
	
	Every left $C$-module $M$ decomposes as $\displaystyle{\bigoplus_{X\in\Ob{\mathcal{C}}}}1_X M$. We say that $M$ is {\em locally finite-dimensional} if $1_X M$ is finite-dimensional, for all $X\in\Ob({\mathcal{C}})$. 
	
	\subsubsection{Various categories}\label{sec:categ} We fix the notation for various categories that appear in this manuscript. Consider the path algebra $C$ given in \eqref{al:pathalg}. Let $C\BMod$ denote the category of left $C$-modules. Let $C\Mod$ denote the full subcategory of $C\BMod$ whose objects are locally finite-dimensional, $C\fdMod$ denote the full subcategory of $C\BMod$ whose objects are finite-dimensional and $C\Proj$ denote the full subcategory of  $C\BMod$ whose objects are finitely generated projective.  
	
	\subsubsection{Split Grothendieck group}\label{sec:Groth}
	For the path algebra $C$ given in \eqref{al:pathalg}, let $K_0(C)$ denote the split Grothendieck group of the category finitely generated projective left $C$-modules. 
	The monoidal structure on $C\BMod$ given by \eqref{al:rcstar} induces the following ring structure on $K_0(C)$:
	\begin{displaymath}
		[M][N]:=[M\cstar N].
	\end{displaymath}
	
	\subsection{Wreath products}\label{sec:wreath}
	Throughout the manuscript $C_r$ denotes the multiplicative cyclic group of order $r$ generated by  a fixed primitive $r$-th root of unity $\zeta$. In this article, wreath products $C_r\wr S_n$ and $(C_r\times C_r) \wr S_n$ appear and below we discuss their irreducible representations. Unless stated otherwise, we always consider left action of groups. We refer~\cite[Chapter 4]{JK} for details of these results.
	
	\begin{enumerate}[(i)]
		\item Consider the wreath product $C_r\wr S_n$, which is also the complex reflection group $G(r,n)$. An element of $G(r,n)$ is of the form $(h,\sigma)$, where $\sigma\in S_n$ and $h:\{1,2,\ldots,n\}\to C_r$ is a map. For $\tau\in S_n$, let $h_\tau:\{1,2,\ldots,n\}\to C_r$ be defined by $h_{\tau}(i)=h(\tau^{-1}(i))$.  For  $(f,\tau)$ and $(h,\sigma)$ in $G(r,n)$, the multiplication is given by $$(f,\tau)(h,\sigma)=(fh_\tau,\tau\sigma).$$

		{\bf Construction of the irreducible representations.} For a partition $\mu=(\mu_1,\ldots,\mu_l)$, define the weight of $\mu$ to $\lvert \mu \rvert:=\mu_1+\cdots+\mu_l$. Let $\mathcal{P}_{r,n}$ denote the $r$-tuples of partitions of total weight $n$, i.e.,
		\begin{displaymath}
			\mathcal{P}_{r,n}:=\big\{\OV{\lambda}:=(\lambda^{(1)},\ldots,\lambda^{(r)})\mid \lambda^{(i)}\in\mathcal{P} \text{ for all } 1\leq i\leq r\text{ and } \sum_{i=1}^{r}\lvert \lambda^{(i)}\rvert=n \big\}.
		\end{displaymath}
		
		The irreducible representations of $G(r,n)$ over $\CC$ are indexed by the elements of $\mathcal{P}_{r,n}$. For $\lambda=(\lambda^{(1)},\ldots,\lambda^{(r)})$, let $k_1=\lvert \lambda^{(1)}\rvert,\ldots, k_r=\lvert \lambda^{(r)}\rvert$. Let $V^{\lambda_i}$ be the Specht module of $S_{k_i}$ corresponding to $\lambda_i$. Extend the action of $S_{k_{i}}$ on $V^{\lambda_i}$ to $G(r,k_i)$ as follows. For $(h,\sigma)\in G(r,k_i)$ and $v\in V^{\lambda_i}$, let
		\begin{displaymath}
			(h,\sigma)v:=\phi_i\big(\displaystyle{\prod_{j=1}^{k_i}h(j)}\big)  \sigma v,
		\end{displaymath}
		where $\phi_i:C_r\to \CC^{*}$ is the one-dimensional representation of $C_r$ such that $\phi_i(\zeta)=\zeta^{i-1}$. The corresponding irreducible representation $S(\OV{\lambda})$ is the induced representation 
		\begin{displaymath}
			\Ind_{G(r,k_1)\times\cdots\times G(r,k_r)}^{G(r,n)} (V^{\lambda_1}\otimes\cdots \otimes V^{\lambda_r}). 
		\end{displaymath}
		
		\subsubsection{Pulling back representations}\label{sec:twist} For groups $G, H$, a representation $M$ of $H$ and a group homomorphism $\rho: G\to H$, the pull back representation $M^{\rho}$ of $G$ is given by
		$g\cdot m=\rho(g)m$, where $g\in G$ and $m\in M$.
		
		For a vector space $V$, let $V^{*}$ denote the $\CC$-linear dual of $V$. In below, we describe two known examples of pulling back representations. More instances of such representations occur in Section~\ref{sec:stab}.
		
		{\bf Dual representations.}
		\label{dual} The assignment of an element of the group $G(r,n)$ to its inverse defines an anti-involution $\Inv_n$ on $G(r,n)$. Let $G(r,n)^{\op}$ denote the opposite group of $G(r,n)$. Then $\Inv_n$ is a group isomorphism from $G(r,n)$ to $G(r,n)^{\op}$.

		For a representation $V$ of $G(r,n)$, the $\CC$-linear dual $V^{*}$ is a representation of $G(r,n)^{\op}$. Then the pull back representation  $(V^{*})^{\Inv_n}$ is the well-known dual representation of $G(r,n)$. Let $\omega:\mathcal{P}_{r,n}\to \mathcal{P}_{r,n}$ be given by \begin{displaymath}
			\omega(\OV{\lambda})=(\lambda^{(1)},\lambda^{(r)},\lambda^{(r-2)},\ldots,\lambda^{(2)}),
		\end{displaymath}
		where $\OV{\lambda}=(\lambda^{(1)},\ldots, \lambda^{(r)})\in \mathcal{P}_{r,n}$.
		Then $(S(\OV{\lambda})^{*})^{\Inv_n}$ is isomorphic to $S(\omega(\OV{\lambda}))$. Note that for $r>2$, the irreducible representations of $G(r,n)$ may not be self-dual.

		\label{anotherdual} \textbf{ A simple-preserving duality.} For $(h,\sigma)\in G(r,n)$, the assignment $(h,\sigma)\mapsto (h_{\sigma^{-1}},\sigma^{-1})$ is yet another anti-involution $\bar{\Inv}_n$ on $G(r,n)$. (Note that this anti-involution also appeared in proving cellularity~\cite[Theorem~5.5]{GL} for Hecke algebras for $G(r,n)$, i.e., Ariki--Koike algebras.) For a finite-dimensional $G(r,n)$-module $V$, define $\widetilde{\Theta}_n(V)$ to be the pull back representation $(V^{*})^{\bar{\Inv}_n}$ of $G(r,n)$.  So we have a contravariant functor $$\widetilde{\Theta}_n:\CC[G(r,n)]\fdMod\to\CC[G(r,n)]\fdMod$$ such that $\widetilde{\Theta}_n^{2}$ is isomorphic to the identity functor on $\CC[G(r,n)]\fdMod$. Since the elements of $(h,\sigma)$ and $(h_{\sigma^{-1}},\sigma^{-1})$ have the same cycle type (see~\cite[Section~4.2]{JK} for the definition of the cycle type  and~\cite[Theorem~4.2.8]{JK} for the characterization of conjugate elements in terms of their cycle type for wreath products), these elements are conjugate to each other and so that for $\OV{\lambda}\in\mathcal{P}_{r,n}$, $\widetilde{\Theta}_n({S(\OV{\lambda})})$ is isomorphic to $S(\OV{\lambda})$.
		
		\item Consider the wreath product $(C_r\times C_r)\wr S_n$, which we denote by $H(r,n)$. An element of $H(r,n)$ is of the form $((h_1,h_2),\sigma)$, where $\sigma\in S_n$ and $h_1:\{1,2,\ldots,n\}\to C_r$, $h_2:\{1,2,\ldots,n\}\to C_r$ are maps.
		
		The irreducible representations of $H(r,n)$ are indexed by 
		\begin{align}\label{al:set}
			\left\{\left(\mu^{(p,q)}\right)_{1\leq p,q\leq r}\mid \mu_{(p,q)}\in\mathcal{P}  \text{ and }\sum_{1\leq p,q\leq r}\lvert\mu_{(p,q)}\rvert=n\right\}.
		\end{align}
	\end{enumerate}
	\subsection{Wreath product symmetric functions}\label{sec:wreathfun}
	Let $\Lambda$ denote the ring of symmetric functions in the variables $x_1,x_2,\ldots$ over $\CC$ (see for example \cite[Chapter I]{Mac}). As given in \cite{Remmel},
	using $\lambda$-ring notation for symmetric function one can elegantly express irreducible characters of wreath products. 
	
	\subsubsection{\texorpdfstring{$\lambda$}{}-ring notation}
	Let $A$ be a set of formal commuting variables. In what follows, by a ``word'' in $A$ we mean an equivalence class of words in $A$ with respect to the relations that two words are equivalent if they can be obtained from one another by permuting the letters. 
	
	Let $\ZZ_{>0}$ and $\ZZ_{\geq 0}$ denote the set of positive integers and nonnegative integers, respectively. Let $y=a_1a_2\cdots a_n$ be a word in the alphabet $A$ (i.e. a monomial in the polynomial algebra over $A$). Let $Y,Y_i$, where $i\in\ZZ_{>0}$, be formal sums of words in $A$ with complex coefficients. Let $\mathcal{P}$ denote the set of all partitions of all nonnegative integers. For a nonnegative integer $l$, consider the power symmetric function $p_l=x_1^l+x_2^l+\cdots$. Then the $\lambda$-ring notation for power symmetric function is given by
	\begin{align*}
		p_l[0]&=0;  & p_l[1]&=1;\\
		p_l[y]&=y^l=a_1^la_2^l\cdots a_n^l;  & p_l[cY]&=c^lY, \text{ where } c\in\CC;\\
		p_l\left[\sum_i Y_i\right]&=\sum_i p_l[Y_i];  & p_{\gamma}[Y]&=p_{\gamma_1}[Y]\cdots p_{\gamma_r}[Y], \text{ where } \gamma\in\mathcal{P}.
	\end{align*}
	We know that the $\{p_{\alpha}\mid \alpha\in\mathcal{P}\}$ forms a basis for $\Lambda$ (see \cite{Mac}), so any symmetric function $f\in\Lambda$ can be uniquely written as $f=\displaystyle{\sum_{\alpha\in\mathcal{P}}}a_{\alpha}p_{\alpha}$. Then the $\lambda$-ring notation for $f$ is given by 
	\begin{displaymath}
		f[Y]:=\sum_{\alpha\in\mathcal{P}}a_{\alpha}p_{\alpha}[Y].
	\end{displaymath}
	For symmetric functions $f_1,f_2\in\Lambda$, one can derive from the above definitions that 
	\begin{align}\label{al:twofunct}
		(f_1f_2)[Y]:=f_1[Y]f_{2}[Y].    
	\end{align}
	Let $\Lambda[Z]$ denote the space of symmetric functions in the variables $Z$. Note that when $Z=\{x_1,x_2,\ldots\}$, then $\Lambda[Z]=\Lambda$. As usual, we will identify a formal sum of words in $A$ with the set of monomials which appear in this sum with nonzero coefficients. In particular, we can think of 
	$Z$ as $x_1+x_2+\cdots$.
	
	Let $X^{(i)}$ be the set $\{x_1^{(i)},x_2^{(i)},\dots\}$, for $i=1,2,\dots,r$.	The $r$-fold tensor product $\displaystyle{\bigotimes_{i=1}^{r}}\Lambda[X^{(i)}]$ of the rings of symmetric functions may be called as the ring of wreath product symmetric function as suggested in~\cite{Wreathsym}. In~\cite{Remmel}, a Frobenius characteristic map was defined such that the ring of wreath product symmetric functions is isomorphic to the direct sum of the character ring of $G(r,n)$ for all nonnegative integers $n$. We are interested in two bases for this ring and the structure constants for each basis. These are the easy observations from $r=1$ case, which we recall below.
	
	{\bf Schur functions and the Littlewood--Richardson coefficients.}  Let $s_{\lambda}$ denote the Schur function corresponding to a partition $\lambda$ and it is well-known that $\{s_{\lambda}\mid \lambda\in\mathcal{P}\}$ is a basis of $\Lambda$, see~\cite[Section I.3]{Mac}. The structure constants with respect to this basis are given by the Littlewood--Richardson coefficients, which are defined below.
	
	For $\lambda$ and $\mu$ in $\mathcal{P}$ with $\lvert \lambda\rvert =m$ and $\lvert \mu\rvert=n$, the tensor product $S(\lambda)\otimes S(\mu)$ is an irreducible $S_m\times S_n$-module. Then the multiplicity of $S(\nu)$, for $\nu\in\mathcal{P}$ with $\lvert \nu\rvert=m+n$, in the induced representation $$\Ind_{S_m\times S_n}^{S_{n+m}}(S(\lambda)\otimes S(\mu))$$ is the Littlewood--Richardson coefficient $\LR^{\nu}_{\lambda,\mu}$. 
	
	For $\OV{\lambda}=(\lambda^{(1)},\ldots,\lambda^{(r)})\in\mathcal{P}_{r,m}$ and $\OV{\mu}=(\mu^{(1)},\ldots,\mu^{(r)})\in \mathcal{P}_{r,n}$, the tensor product $S(\OV{\lambda})\otimes S(\OV{\mu})$ is an irreducible module over $G(r,m)\times G(r,n)$. Then the multiplicity of $S(\OV{\nu})$, for $\OV{\nu}=(\nu^{(1)},\ldots,\nu^{(r)})\in\mathcal{P}_{r,n+m}$ in the induced representation \begin{displaymath}
		\Ind_{G(r,m)\times G(r,n)}^{G(r,m+n)}(S(\OV{\lambda})\otimes S(\OV{\mu}))
	\end{displaymath} is given by
	\begin{equation}\label{eq:LR}
		\LR^{\OV{\nu}}_{\OV{\lambda},\OV{\mu}}=\prod_{i=1}^{r}\LR^{\nu^{(i)}}_{\lambda^{(i)},\mu^{(i)}}.
	\end{equation}
	
	For $\OV{\lambda}\in\mathcal{P}_{r,m}$, $\OV{\mu}\in \mathcal{P}_{r,n}$, $\OV{\gamma}\in\mathcal{P}_{r,l}$ and $\nu\in\mathcal{P}_{r,m+n+l}$, we have
	\begin{equation}\label{eq:LRtriple}
		\LR^{\OV{\nu}}_{\OV{\lambda},\OV{\mu},\OV{\alpha}}=\sum_{\beta\in \mathcal{P}_r} \LR^{\OV{\nu}}_{\OV{\lambda},\OV{\beta}}\LR^{\OV{\beta}}_{\OV{\mu},\OV{\alpha}}
	\end{equation}
	is the multiplicity of $S(\OV{\nu})$ in the induction $\Ind_{G(r,m)\times G(r,n)\times G(r,l)}^{G(r,m+n+l)}(S(\OV{\lambda})\otimes S(\OV{\mu})\otimes S(\OV{\alpha}))$. By Frobenius reciprocity, $\LR^{\OV{\nu}}_{\OV{\lambda},\OV{\mu},\OV{\alpha}}$ is the multiplicity of the irreducible module  $(S(\OV{\lambda})\otimes S(\OV{\mu})\otimes S(\OV{\alpha}))$ in the restriction
	$$\Res^{G(r,m+n+l)}_{G(r,m)\times G(r,n)\times G(r,l)}S(\OV{\nu}).$$ The following theorem is well-known. The key point is to use \eqref{al:twofunct} together with the fact that the structure constants for the basis consisting of Schur functions are given by the Littlewood--Richardson coefficients. Recall that $\mathcal{P}_{r,n}$ denote the set of all $r$-tuples of partitions of total weight $n$. Let  
	$$\mathcal{P}_r= \displaystyle{\bigcup_{n\in\ZZ_{\geq 0}}}\mathcal{P}_{r,n}.$$
	\begin{proposition}\label{thm:basesymlittle}
		Let $\chi_{i}^{j}$ denote the character of $C_r$ indexed by $j$ evaluated at $\zeta^{i}$.  The set
		\begin{align*}
			\mathcal{B}_1:=\bigg\{s_{\OV{\lambda}}:=\prod_{i=1}^r s_{\lambda^{(i)}}\bigg[\sum_{j=1}^r\chi_j^{i} X^{(j)}\bigg]\mid \OV{\lambda}:=(\lambda^{(1)},\ldots,\lambda^{(r)})\in\mathcal{P}_{r}\bigg\}
		\end{align*}
		is a basis for $\displaystyle{\bigotimes_{i=1}^{r}}\Lambda[X^{(i)}]$. The structure constant of $s_{\OV{\nu}}$ in the product $s_{\OV{\lambda}}s_{\OV{\mu}}$ is are given by the product \eqref{eq:LR} of the Littlewood--Richardson coefficients. 
	\end{proposition}			
	For $\OV{\lambda}\in\mathcal{P}_r$, following \cite{Wreathsym}, the basis element $s_{\OV{\lambda}}$ is called the \emph{wreath product Schur function}.
	
	{\bf Deformed Schur functions and the reduced Kronecker coefficients.} For a partition $\lambda$, the corresponding deformed Schur function (this term is used and motivated in \cite[Section 3]{Brv}) $\tilde{s}_{\lambda}$ was given in~\cite{OZ} and it was shown that $\{\tilde{s}_{\lambda}\mid \lambda\in\mathcal{P}\}$ is a basis of $\Lambda$. The structure constants with respect to this basis are given by the reduced Kronecker coefficients, which are defined below. 
	
	For a partition $\mu=(\mu_1,\ldots,\mu_l)$ and a positive integer $n$ such that $n-|\mu|\geq \mu_1$, the padded partition $\mu[n]$ is $$(n-|\mu|, \mu_1,\ldots, \mu_l).$$ Let $G^{\mu[n]}_{\lambda[n],\nu[n]}$ denote the multiplicity, called as the Kronecker coefficient, of the Specht module $S(\mu[n])$ in the tensor product $S(\lambda[n])\otimes S(\nu[n])$ with respect to the diagonal action of $S_n$. Then the limit $\overline{G}^{\mu}_{\lambda,\nu}$ of the sequence of the Kronecker coefficients $G^{\mu[n]}_{\lambda[n],\nu[n]}$, as $n\to \infty$, is finite, see~\cite{Murnaghan}. This limit is called the reduced Kronecker coefficient. 
	
	For $\OV{\lambda}=(\lambda^{(1)},\ldots,\lambda^{(r)}), \OV{\mu}=(\mu^{(1)},\ldots,\mu^{(r)})$ and $\OV{\nu}= (\nu^{(1)},\ldots,\nu^{(r)})$ in $\mathcal{P}_r$, let 
	\begin{equation}\label{eq:Kr}
		\overline{G}^{\OV{\mu}}_{\OV{\lambda},\OV{\nu}}:=\prod_{i=1}^{r}\overline{G}^{{\mu}^{(i)}}_{{\lambda}^{(i)},{\nu}^{(i)}}.
	\end{equation}
	\begin{proposition}\label{thm:basesym}
		Let $\chi_{i}^{j}$ denote the character of $C_r$ indexed by $j$ evaluated at $\zeta^{i}$. 
		The set
		\begin{align*}
			\mathcal{B}_2:=\bigg\{\tilde{s}_{\OV{\lambda}}:=\prod_{i=1}^r \tilde{s}_{\lambda^{(i)}}\bigg[\sum_{j=1}^r\chi_j^{i} X^{(j)}\bigg]\mid \OV{\lambda}:=(\lambda^{(1)},\ldots,\lambda^{(r)})\in\mathcal{P}_{r}\bigg\}
		\end{align*}
		is a basis for $\displaystyle{\bigotimes_{i=1}^{r}}\Lambda[X^{(i)}]$. The structure constant of $\tilde{s}_{\OV{\nu}}$ in the product $\tilde{s}_{\OV{\lambda}}\tilde{s}_{\OV{\mu}}$ is given by the product of \eqref{eq:Kr} of the reduced Kronecker coefficients. 
	\end{proposition}
	\begin{proof}
		Let $M$ denote the transformation matrix from the basis $\{s_{\lambda}\mid \lambda\in\mathcal{P}\}$ to the basis $\{\tilde{s}_{\lambda}\mid \lambda\in\mathcal{P}\}$. Then $M^{\otimes r}$ is the transformation matrix from $\mathcal{B}_1$ to $\mathcal{B}_2$ (see Theorem~\ref{thm:basesymlittle} for the definition of $\mathcal{B}_1$). Since $M$ is invertible, $M^{\otimes r}$ is invertible and so $\mathcal{B}_2$ is also a basis.
		
		From \cite{OZ}, we know that the structure constants for the basis $\{\tilde{s}_{\lambda} \mid \lambda\in \mathcal{P}\}$ are given by the reduced Kronecker coefficients. Then it follows from \eqref{al:twofunct} that the structure constants for the basis $\mathcal{B}_2$ are given by the product of the reduced Kronecker coefficients.
	\end{proof}
	
	For $\OV{\lambda}\in\mathcal{P}_{r}$,  we call $\tilde{s}_{\OV{\lambda}}$ as the \emph{deformed wreath product Schur function}.
	
	\section{Multiparameter colored partition category}\label{sec:multicat}
	In this section, we define the multiparameter colored partition category and gives its presentation as monoidal category, and also we give triangular decomposition of the path algebra of this category.  We also define the multiparameter colored partition algebra and give a realization of this algebra as the twisted semigroup algebra of the colored partition monoid. A presentation of this monoid is also given.
	
	Recall that $C_r$ denotes the multiplicative cyclic group of order $r$ generated by  a fixed primitive $r$-th root of unity $\zeta$.
	
	A {\em colored set-partition} of $\{1,2,\ldots,k,1',2',\ldots,l'\}$ is a set $\{(B_1,\zeta^{i_1}),\ldots, (B_s,\zeta^{i_s})\}$ where $\{B_1,\ldots, B_s\}$ forms a set-partition of $\{1,2,\ldots,k,1',2',\ldots,l'\}$. A colored set-partition also has a graphical interpretation that we describe below.

	A {\em colored $(k,l)$-partition diagram}
	is a $(k,l)$-partition diagram whose parts
	are labeled by some elements of  $C_r$.  
	We will often omit the labels that are given
	by the identity element. In this way, each $(k,l)$-partition diagram can be considered as
	a colored $(k,l)$-partition diagram  in which 
	all labels are equal to the identity element of $C_r$. Two colored $(k,l)$-partition diagrams are equivalent if and only if they represent the same colored set-partitions.
	
	Given a colored partition diagram $d$, a part of $d$ is called \emph{propagating} if it intersect both the top and the bottom rows. The number of propagating parts of $d$, denoted $\pn(d)$,  is called \emph{the rank} of $d$.
	\begin{example}
		Let $r=5,k=7$ and $l=8$. The following colored partition diagram has rank $5$:
		
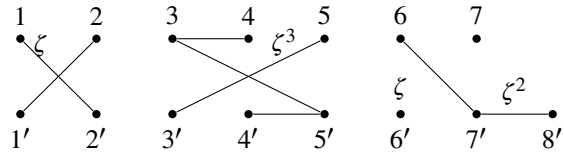
\begin{align*} 
			\begin{tikzpicture}[scale=1,mycirc/.style={circle,fill=black, minimum size=0.1mm, inner sep = 1.1pt}]
				\node[mycirc,label=above:{$1$}] (n1) at (0,1) {};
				\node[mycirc,label=above:{$2$}] (n2) at (1,1) {};
				\node[mycirc,label=above:{$3$}] (n3) at (2,1) {};
				\node[mycirc,label=above:{$4$}] (n4) at (3,1) {};
				\node[mycirc,label=above:{$5$}] (n5) at (4,1) {};
				\node[mycirc,label=above:{$6$}] (n6) at (5,1) {};
				\node[mycirc,label=above:{$7$}] (n7) at (6,1) {};
				\node[mycirc,label=below:{$1'$}] (n1') at (0,0) {};
				\node[mycirc,label=below:{$2'$}] (n2') at (1,0) {};
				\node[mycirc,label=below:{$3'$}] (n3') at (2,0) {};
				\node[mycirc,label=below:{$4'$}] (n4') at (3,0) {};
				\node[mycirc,label=below:{$5'$}] (n5') at (4,0) {};
				\node[mycirc,label=below:{$6'$}, label=above:{$\zeta$}] (n6') at (5,0) {};
				\node[mycirc,label=below:{$7'$}] (n7') at (6,0) {};
				\node[mycirc,label=below:{$8'$}] (n8') at (7,0) {};
				\path[-, draw](n1) edge node[near start, above=-0.09cm] {$\zeta$} (n2');
				\path[-,draw](n2) to (n1');
				\path[-,draw](n3) to (n4);
				\path[-,draw](n5)edge node[near start, above=-0.09cm] {$\zeta^3$}(n3');
				\path[-,draw](n4') to (n5');
				\path[-,draw](n7') edge node[midway,above=0.001cm] {$\zeta^2$} (n8');
				\path[-,draw](n3) to (n5');
				\path[-,draw](n6) to (n7');
			\end{tikzpicture}
		\end{align*}
	\end{example}
	
	From now, throughout this manuscript we fix multiparameter $\textbf{x}=(x_0,x_1,\ldots,x_{r-1})\in\CC^{r}$.
	
	\begin{definition} Let ${\CPar}(\textbf{x})$ be the category whose objects are nonnegative integers and the morphism space from $l$ to $k$ is the free $\CC$-module with the basis consisting of all colored $(k,l)$-partition diagrams. 
		
		\textbf{Composition of morphisms.}	Now we explain how to compose  a colored $(k,l)$-partition diagram $d_1$ with a colored $(l,m)$-partition diagram $d_2$. Let $d_1'$ and $d_2'$, respectively, be the partition diagrams obtained from $d_1$ and $d_2$ by forgetting their colors.
		Let $L$ be a part of the multiplication $d_1'\circ d_2'$ or a part that lie entirely in the middle while composing $d_1'$ with $d_2'$.  We color $L$ by the product of all colors of the parts of both  $d_1$ and  $d_2$ which contributed to $L$. Now by coloring each part of the multiplication $d_1'\circ d_2'$ as described, we obtain a colored $(l,m)$-partition diagram $d_1\circ d_2$. Let $m_i$ be the number of  parts that lie entirely in the middle  and colored with $\zeta^i$ for $0\leq i\leq r-1$. Then the composition $d_1d_2$ of $d_1$ with $d_2$ is given by
		\begin{displaymath}
			d_1d_2=x_0^{m_0}x_1^{m_1}\cdots x_{r-1}^{m_{r-1}}d_1\circ d_2.
		\end{displaymath}
	\end{definition}
	
	\begin{example}
		Let $r=5$. Consider the following colored partition diagrams:
		
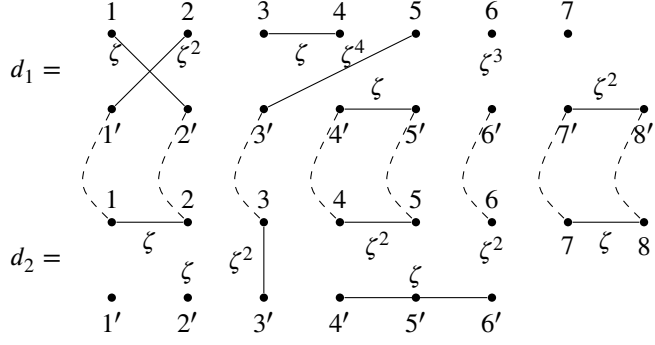
\begin{center} 
			\begin{tikzpicture}[scale=1,mycirc/.style={circle,fill=black, minimum size=0.1mm, inner sep = 1.1pt}]
				
				\node (1) at (-1,0.5) {$d_1 =$};
				\node[mycirc,label=above:{$1$}] (n1) at (0,1) {};
				\node[mycirc,label=above:{$2$}] (n2) at (1,1) {};
				\node[mycirc,label=above:{$3$}] (n3) at (2,1) {};
				\node[mycirc,label=above:{$4$}] (n4) at (3,1) {};
				\node[mycirc,label=above:{$5$}] (n5) at (4,1) {};
				\node[mycirc,label=above:{$6$}, label=below:{$\zeta^3$}] (n6) at (5,1) {};
				\node[mycirc,label=above:{$7$}] (n7) at (6,1) {};
				
				\node[mycirc,label=below:{$1'$}] (n1') at (0,0) {};
				\node[mycirc,label=below:{$2'$}] (n2') at (1,0) {};
				\node[mycirc,label=below:{$3'$}] (n3') at (2,0) {};
				\node[mycirc,label=below:{$4'$}] (n4') at (3,0) {};
				\node[mycirc,label=below:{$5'$}] (n5') at (4,0) {};
				\node[mycirc,label=below:{$6'$}] (n6') at (5,0) {};
				\node[mycirc,label=below:{$7'$}] (n7') at (6,0) {};
				\node[mycirc,label=below:{$8'$}] (n8') at (7,0) {};

				\path[-, draw](n1) edge node[ near start,left] {$\zeta$}  (n2');
				\path[-,draw](n2) edge node[near start, right] {$\zeta^2$} (n1');
				\path[-,draw](n3) edge node[below] {$\zeta$} (n4);
				\path[-,draw] (n5)	edge node[near start,left] {$\zeta^4$} (n3');
				\path[-,draw] (n4') edge node[above] {$\zeta$} (n5');
				\path[-,draw] (n7') edge node[above] {$\zeta^2$} (n8');

				\node (1) at (-1,-2) {$d_2 =$};
				\node[mycirc,label=above:{$1$}] (n7) at (0,-1.5) {};
				\node[mycirc,label=above:{$2$}] (n8) at (1,-1.5) {};
				\node[mycirc,label=above:{$3$}] (n9) at (2,-1.5) {};
				\node[mycirc,label=above:{$4$}] (n10) at (3,-1.5) {};
				\node[mycirc,label=above:{$5$}] (n11) at (4,-1.5) {};
				\node[mycirc,label=above:{$6$},label=below:{$\zeta^2$}] (n12) at (5,-1.5) {};
				\node[mycirc,label=below:{$7$}] (n13) at (6,-1.5) {};
				\node[mycirc,label=below:{$8$}] (n14) at (7,-1.5) {};

				\node[mycirc,label=below:{$1'$}] (n7'') at (0,-2.5) {};
				\node[mycirc,label=below:{$2'$},label=above:{$\zeta$}] (n8'') at (1,-2.5) {};
				\node[mycirc,label=below:{$3'$}] (n9') at (2,-2.5) {};
				\node[mycirc,label=below:{$4'$}] (n10') at (3,-2.5) {};
				\node[mycirc,label=below:{$5'$}] (n11') at (4,-2.5) {};
				\node[mycirc,label=below:{$6'$}] (n12') at (5,-2.5) {};
				
				\path[-,draw] (n7) edge node[below] {$\zeta$} (n8);
				\path[-,draw] (n9) edge node[left] {$\zeta^2$} (n9');
				\path[-,draw] (n10) edge node[below] {$\zeta^2$} (n11);
				\path[-,draw] (n13) edge node[below] {$\zeta$} (n14);
				\path[-,draw] (n10') edge node [above] {$\zeta$} (n12');
				
				\draw[dashed] (n1') .. controls(-0.5,-1)..(n7);
				\draw[dashed] (n2') .. controls(0.5,-1)..(n8);
				\draw[dashed] (n3') .. controls(1.5,-1)..(n9);
				\draw[dashed] (n4') .. controls(2.5,-1)..(n10);
				\draw[dashed] (n5') .. controls(3.5,-1)..(n11);
				\draw[dashed] (n6') .. controls(4.5,-1)..(n12);
				\draw[dashed] (n7') .. controls(5.5,-1)..(n13);
				\draw[dashed] (n8') .. controls(6.5,-1)..(n14);
			\end{tikzpicture}
		\end{center}
		The concatenation $d_1\circ d_2$ is given as follows:
		
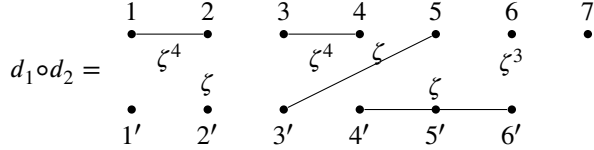
\begin{center}
			\begin{tikzpicture}[scale=1,mycirc/.style={circle,fill=black, minimum size=0.1mm, inner sep = 1.1pt}]
				\node (1) at (-1,0.5) {$d_1\circ d_2 =$};
				\node[mycirc,label=above:{$1$}] (n1) at (0,1) {};
				\node[mycirc,label=above:{$2$}] (n2) at (1,1) {};
				\node[mycirc,label=above:{$3$}] (n3) at (2,1) {};
				\node[mycirc,label=above:{$4$}] (n4) at (3,1) {};
				\node[mycirc,label=above:{$5$}] (n5) at (4,1) {};
				\node[mycirc,label=above:{$6$}, label=below:{$\zeta^3$}] (n6) at (5,1) {};
				\node[mycirc,label=above:{$7$}] (n7) at (6,1) {};
				
				\node[mycirc,label=below:{$1'$}] (n7'') at (0,0) {};
				\node[mycirc,label=below:{$2'$},label=above:{$\zeta$}] (n8'') at (1,0) {};
				\node[mycirc,label=below:{$3'$}] (n9') at (2,0) {};
				\node[mycirc,label=below:{$4'$}] (n10') at (3,0) {};
				\node[mycirc,label=below:{$5'$}] (n11') at (4,0) {};
				\node[mycirc,label=below:{$6'$}] (n12') at (5,0) {};
				
				\path[-,draw] (n1) edge node[below] {$\zeta^4$} (n2);
				\path[-,draw] (n3) edge node[below]{$\zeta^4$} (n4);
				\path[-,draw] (n5) edge node[near start, left]{$\zeta$} (n9');
				\path[-,draw] (n10') edge node[above] {$\zeta$} (n12');
			\end{tikzpicture}
		\end{center}
		There is no part of colors $\zeta^0$ and $\zeta^1$ lying entirely in the middle and so $m_0=0=m_1$. We have one part and two parts of colors $\zeta^2$ and $\zeta^3$, respectively. Thus $m_2=1$ and $m_3=2$. Therefore $d_1d_2=x_2x_3^2 d_1\circ d_2$.
	\end{example}
	
	\textbf{ The monoidal structure.} We are going to define a monoidal structure $\star$ on the category ${\CPar}(\textbf{x})$. For objects $k$ and $l$ in ${\CPar}(\textbf{x})$, let $k\star l:=k+l$. For diagram morphisms $d_1$ and $d_2$ in ${\CPar}(\textbf{x})$, let $d_1\star d_2$ be the colored partition diagram obtained by drawing $d_1$ to the left of $d_2$ and then re-indexing the top and the bottom vertices of $d_2$. The operation $\star$ gives a rigid symmetric strict monoidal structure on the multiparameter colored partition category ${\CPar}(\textbf{x})$.
	
	The partition category ${\Par}(x_0)$ defined by Deligne \cite{Deligne} is a wide subcategory of  ${\CPar}(\textbf{x})$. For $k\in\ZZ_{\geq 0}$, the endomorphism of $k$ in ${\CPar}(\textbf{x})$ is called \emph{the multiparameter colored partition algebra}, denoted $\cPar_k(\textbf{x})$.  
	\subsection{Presentation}
	A presentation of the partition category as a monoidal category is given in~\cite[Theorem 1]{Comes}. Next we define a rigid symmetric strict monoidal category by giving its presentation. Then we show that this category is equivalent to the multiparameter colored partition category. 
	\begin{definition}
		Let $\widetilde{\CPar}(\textbf{x})$ be a strict $\CC$-linear monoidal category defined in terms of generators and relations as follows. The monoidal structure is denoted by $\star$.
		\begin{enumerate}[($a$)]
			\item It has  one generating object $[1]$, and so  for a nonnegative integer $n$, $[n]$ denotes $$\underbrace{[1]\star[1]\star\cdots\star[1]}_{n \text{ times}}.$$
			\item It has the following generating morphisms: 
			
			
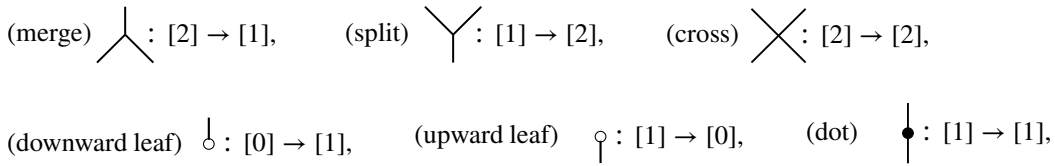
\begin{tikzpicture}
				\node at (-1,2) {\text{(merge)}};
				\node at (0,2) {\merge};
				\node at (1.1,2) {$\colon [2]\to [1],$};
			\end{tikzpicture} \hspace{0.5cm}
			\begin{tikzpicture}
				\node at (-1,2) {\text{(split)}};
				\node at (0,2) {\spt};
				\node at (1.1,2) {$\colon [1] \to [2],$};
			\end{tikzpicture}\hspace{0.5cm}
			\begin{tikzpicture}
				\node at (-1,2) {\text{(cross)}};
				\node at (0,2) {\cross};
				\node at (1.1,2) {$\colon [2] \to [2],$};
			\end{tikzpicture}
			
			\begin{tikzpicture}
				\node at (-1.5,0) {\text{(downward leaf)}};
				\opendot{0,0};
				\draw [black,line width=0.8pt] (0,0.055)--(0,0.35);
				\node at (1,0) {$\colon [0] \to [1],$};
			\end{tikzpicture}\hspace{0.5cm}
			\begin{tikzpicture}
				\node at (-1.5,0) {\text{(upward leaf)}};
				\opendot{0,0};
				\draw [black,line width=0.8pt] (0,-0.055)--(0,-0.35);
				\node at (1,0) {$\colon [1] \to [0],$};
			\end{tikzpicture}\hspace{0.5cm}
			\begin{tikzpicture}
				\node at (-1.0,1.6) {\text{(dot)}};
				\draw [black, line width=0.7pt] (0,1.2)--(0,2);
				\filldraw[black] (0,1.6) circle (2pt) ;
				\node at (1,1.6) {$\colon [1] \to [1],$};
			\end{tikzpicture}
		\end{enumerate}
		For a nonnegative integer $a$, the composition of the dot morphism with itself $a$ many times is depicted as follows:
		\begin{align*}
			\begin{tikzpicture}
				\node at (0,0) {$\diage$};
				\node at (0.6,0.0) {$:=$};
				\node at (1.3,0) {$\diagf$};
			\end{tikzpicture}
		\end{align*}
		Then the relations among the generating morphisms in the part ($b$) are the following:
		\begin{alignat*}{3}
			&\begin{tikzpicture}
				\node at (-1,2) {$(1)$};
				\node at (0,2) {\diaga};
				\node at (0.6,2) {$=$};
				\draw [black, line width=1pt] (1.1,1.5)--(1.1,2.5);
				\node at (1.6,2) {$=$};
				\node at (2.2,2) {\diagb};
			\end{tikzpicture} 
			&&
			\begin{tikzpicture}
				\node at (-1,2) {$(2)$};
				\node at (0,2) {\diagc};
				\node at (0.6,2) {$=$};
				\draw [black, line width=1pt] (1.1,1.5)--(1.1,2.5);
				\node at (1.6,2) {$=$};
				\node at (2.2,2) {\diagd};
			\end{tikzpicture} &&
			\begin{tikzpicture}
				\node at (-1,0) {$(3)$};
				\node at (0,0) {$\diagg$};
				\node at (0.6,0) {$=$};
				\node at (1.3,0) {$\diagh$};
			\end{tikzpicture}\\
			&
			\begin{tikzpicture}
				\node at (-1,0) {$(4)$};
				\node at (0,0) {$\diagi$};
				\node at (0.6,0) {$=$};
				\node at (1.3,0) {$\diagj$};
			\end{tikzpicture}
			&&
			\begin{tikzpicture}
				\node at (-1,0) {$(5)$};
				\node at (0,0) {$\diagk$};
				\node at (0.6,0) {$=$};
				\node at (1.3,0) {$\diagl$};
				\node at (1.9,0) {$=$};
				\node at (2.6,0) {$\diagm$};
			\end{tikzpicture}
			&&
			\begin{tikzpicture}
				\node at (-1,0) {$(6)$};
				\node at (0,0) {$\diagn$};
				\node at (0.6,0) {=};
				\node at (1.3,0) {$\merge$};
			\end{tikzpicture}\\
			&
			\begin{tikzpicture}
				\node at (-1,0) {$(7)$};
				\node at (0,0) {$\tiaga$};
				\node at (0.6,0) {=};
				\node at (1.3,0) {$\spt$};
			\end{tikzpicture}
			&& 
			\begin{tikzpicture}
				\node at (-1,0) {$(8)$};
				\node at (0,0) {$\tiagba$};
				\node at (0.8,0) {=};
				\node at (1.5,0) {$\tiagbb$};
			\end{tikzpicture}
			&& 
			\begin{tikzpicture}
				\node at (-1,0) {$(9)$};
				\node at (0,0) {$\tiagca$};
				\node at (0.8,0) {=};
				\node at (1.5,0) {$\tiagcb$};
			\end{tikzpicture}
			\\
			&
			\begin{tikzpicture}
				\node at (-1,0) {$(10)$};
				\node at (0,0) {$\tiagda$};
				\node at (0.8,0) {=};
				\node at (1.5,0) {$\tiagdb$};
			\end{tikzpicture}
			&& \begin{tikzpicture}
				\node at (-1,0) {$(11)$};
				\node at (0,0) {$\tiagea$};
				\node at (0.8,0) {=};
				\node at (1.5,0) {$\tiageb$};
			\end{tikzpicture}
			&& 	 \begin{tikzpicture}
				\node at (-1,0) {$(12)$};
				\node at (0,0) {$\tiagfa$};
				\node at (0.8,0) {=};
				\node at (1.5,0) {$\tiagfb$};
			\end{tikzpicture} 
			\\
			&
			\begin{tikzpicture}
				\node at (-1,0) {$(13)$};
				\node at (0,0) {$\tiagga$};
				\node at (0.8,0) {=};
				\node at (1.5,0) {$\tiaggb$};
			\end{tikzpicture} 
			&&
			\begin{tikzpicture}
				\node  at (-1,0) {$(14)$};
				\node at (0,0) {$\diagp$};		
				\node at (0.7,0) {=};
				\node at (1.1,0){$\diagq$};			
			\end{tikzpicture}
			&&
			\begin{tikzpicture}
				\node  at (-1,0) {$(15)$};
				\node at (0,0) {$\diagr$};
				\node  at (0.6,0) {$=$};
				\node at (1.3,0) {$\diags$};
			\end{tikzpicture}\\
			&
			\begin{tikzpicture}
				\node at (-1,0) {$(16)$};
				\node at (0,0) {$\diagt$};
				\node at (0.6,0) {$=$};
				\node at (1.3,0) {$\diagu$};
			\end{tikzpicture}
			&&\begin{tikzpicture}
				\node at (-1,0) {$(17)$};
				\node at (0,0) {$\diagv$};
				\node at (0.6,0) {$=$};
				\node at (1,0) {$\diagq$};
			\end{tikzpicture}
			&&
			\begin{tikzpicture}
				\node at (-1,0) {$(18)$};
				\node at (0,0) {$\diagw$};
				\node at (0.6,0) {$=$};
				\node at (1,0){$\diagx$};
			\end{tikzpicture}\\
			&
			\begin{tikzpicture}
				\node at (-1,0) {$(19)$};
				\node at (0,0) {$\diagy$};   
				\node at (0.6,0) {$=$};
				\node at (1,0){$\diagz$};
			\end{tikzpicture}
			&&
			\begin{tikzpicture}
				\node at (-1,0) {$(20)$};
				\node at (0,0) {$\diagaa$};
				\node at (0.6,0) {$=$};
				\node at (1,0) {$\diagab$};
				\node at (1.6,0) {$=$};
				\node at (2,0) {$\diagac$};
			\end{tikzpicture}
			&&
			\begin{tikzpicture}
				\node  at (-1,0) {$(21)$};
				\node at (0,0) {$\diagba$};
				\node at (0.6,0) {$=$};
				\node at (1,0) {$\diagbb$};
				\node at (1.6,0) {$=$};
				\node at (2,0) {$\diagbc$};	
			\end{tikzpicture}\\
			&
			\begin{tikzpicture}
				\node at (-1,0) {$(22)$};
				\node at (0,0) {$\diagcc$};
				\node at (0.6,0) {$=x_i$,};
			\end{tikzpicture}
			\hspace{0.5cm}
			\begin{tikzpicture}
				\node at (0,-1) {};
				\node at (0,1) {};
				\node at (0.9,-0.37) {$\text{ for }0\leq i\leq r-1$.};
			\end{tikzpicture} && && 
		\end{alignat*}
	\end{definition}
	\begin{remark}\label{rem:extrarel}
		\begin{enumerate}[$(i)$]
			\item The relations $(20)$ and $(21)$ imply the following relations:
			\begin{align*}
				\begin{tikzpicture}
					\node at (0,0) {$\diagda$};
					\node at (0.6,0) {$=$}; 
					\node at (1.2,0) {$\diagdb$}; 
				\end{tikzpicture}& &
				\begin{tikzpicture}
					\node at (0,0) {$\diagea$};
					\node at (0.6,0) {$=$}; 
					\node at (1.2,0) {$\diageb$}; 
				\end{tikzpicture}
			\end{align*}
			We call the relations $(18)-(21)$ including the above two relations as {\em sliding relations}.
			\item  The morphisms $S^{2}_0:[0]\to [2]$ and $S^0_2:[2]\to [0]$ defined below are called the {\em cap} and {\em cup} morphisms, respectively:
			\begin{align*}
				\begin{tikzpicture}
					\node at (-1,0) {$S^2_0=$};
					\node at (0,0) {$\diagga$};	
					\node at (0.6,0) {$:=$};
					\node at (1.5,0.3) {$\diaggb$};
				\end{tikzpicture}&&
				\begin{tikzpicture}	
					\node at (-1,0) {$S^0_2=$};
					\node at (0,0) {$\diagfa$};
					\node at (0.6,0) {$:=$};
					\node at (1.5,-0.3) {$\diagfb$};
				\end{tikzpicture}
			\end{align*}

			\item Using the morphisms $S_0^{2}$ and $S_2^{0}$ together with the relations $(1)$, $(2)$ and $(5)$ we see that the generating object $[1]$ is rigid, and hence the category $\widetilde{\CPar}(\textbf{x})$ is rigid. From the relation $(22)$, for $i=0$, we note that the categorical dimension of $[1]$ is $x_0$. Also, the object $[1]$ is Frobenius, due the relations $(1),(2), (5), (8)$ and $(9)$. In summary, $\widetilde{\CPar}(\textbf{x})$ is a rigid symmetric strict monoidal category generated by a Frobenius object of the categorical dimension $x_0$ and an order $r$ automorphism on this object.
		\end{enumerate}
	\end{remark}
	
	The following proposition gives a presentation of $ {\CPar}(\textbf{x})$ and the idea of the proof is more or less the same as in the setting of a group partition category in \cite{Savage}. 
	\begin{proposition}\label{prop:present}
		Define a functor $F:\widetilde{\CPar}(\mathrm{\mathbf{x}}) \to {\CPar}(\mathrm{\mathbf{x}})$ which is given on objects as follows
		\begin{align}
			F([k])=k, \text{ for } k\in\mathbb{Z}_{\geq 0}, 
		\end{align}
		\text{ and on generating morphisms as follows }
		\begin{align*}
			&\resizebox{2.5cm}{1.5cm}
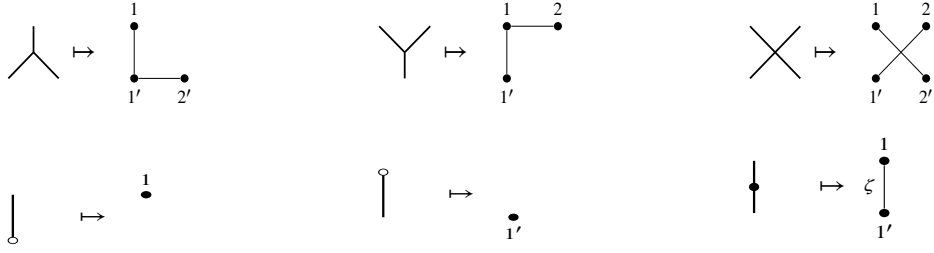
{\begin{tikzpicture} [scale=1,mycirc/.style={circle,fill=black, minimum size=0.1mm, inner sep = 1.6pt}]
					\draw [black, line width=1 pt] (0,0)--(0.5,0.5);
					\draw [black, line width=1 pt] (1,0)--(0.5,0.5);
					\draw [black,line width=1 pt] (0.5,0.5)--(0.5,1);
					\node [scale=1.5] at (1.5,0.5) {$\mapsto$};
					\node[mycirc,label=above:{$1$}] (n1) at (2.5,1) {};
					\node[mycirc,label=below:{$1'$}] (n2) at (2.5,0) {};
					\node[mycirc,label=below:{$2'$}] (n3) at (3.5,0) {};
					\draw (n1)--(n2);
					\draw (n2)--(n3);
			\end{tikzpicture}}
			&&\resizebox{2.5cm}{1.5cm}{\begin{tikzpicture} [scale=1,mycirc/.style={circle,fill=black, minimum size=0.1mm, inner sep = 1.6pt}]
					\draw [black, line width=1 pt] (0,1)--(0.5,0.5);
					\draw [black, line width=1 pt] (1,1)--(0.5,0.5);
					\draw [black,line width=1 pt] (0.5,0.5)--(0.5,0);
					\node [scale=1.5] at (1.5,0.5) {$\mapsto$};
					\node[mycirc,label=above:{$1$}] (n1) at (2.5,1) {};
					\node[mycirc,label=above:{$2$}] (n2) at (3.5,1) {};
					\node[mycirc,label=below:{$1'$}] (n3) at (2.5,0) {};
					\draw (n1)--(n2);
					\draw (n1)--(n3);
			\end{tikzpicture}}
			&&\resizebox{2.5cm}{1.5cm}{\begin{tikzpicture} [scale=1,mycirc/.style={circle,fill=black, minimum size=0.1mm, inner sep = 1.6pt}]
					\draw [black, line width=1 pt] (0,0)--(1,1);
					\draw [black, line width=1 pt] (0,1)--(1,0);
					\node [scale=1.5] at (1.5,0.5) {$\mapsto$};
					\node[mycirc,label=above:{$1$}] (n1) at (2.5,1) {};
					\node[mycirc,label=above:{$2$}] (n2) at (3.5,1) {};
					\node[mycirc,label=below:{$1'$}] (n3) at (2.5,0) {};
					\node[mycirc,label=below:{$2'$}] (n4) at (3.5,0) {};
					\draw (n1)--(n4);
					\draw (n2)--(n3);
			\end{tikzpicture}}\\
			&\resizebox{2cm}{1cm}{\begin{tikzpicture} [scale=1,mycirc/.style={circle,fill=black, minimum size=0.1mm, inner sep = 1.6pt}]
					\draw [black, line width=1 pt] (0,0)--(0,1);
					\opendot{0,0};
					\node [scale=1.3] at (1.2,0.5) {$\mapsto$};
					\node[mycirc,label=above:{$1$}] (n1) at (2,1) {};
			\end{tikzpicture}}&& 
			\resizebox{2cm}{1cm}{\begin{tikzpicture} [scale=1,mycirc/.style={circle,fill=black, minimum size=0.1mm, inner sep = 1.6pt}]
					\draw [black, line width=1 pt] (0,0)--(0,1);
					\opendot{0,1};
					\node [scale=1.3] at (1.2,0.5) {$\mapsto$};
					\node[mycirc,label=below:{$1'$}] (n1) at (2,0) {};
			\end{tikzpicture}}
			&&
			\resizebox{2cm}{1.5cm}{\begin{tikzpicture} [scale=1,mycirc/.style={circle,fill=black, minimum size=0.1mm, inner sep = 1.6pt}]
					\draw [black, line width=1 pt] (0,0)--(0,1);
					\blackdot{0,0.5};
					\node [scale=1.3] at (1.2,0.5) {$\mapsto$};
					\node[mycirc,label=above:{$1$}] (n1) at (2,1) {};
					\node[mycirc,label=below:{$1'$}] (n1') at (2,0) {};
					\path[-,draw] (n1) edge node[left] {$\zeta$} (n1');
			\end{tikzpicture}}
		\end{align*}
		Then $F$ is an equivalence of categories. 
	\end{proposition}
	\begin{proof}
		By the definition of $F$, it is bijective on objects. So we only need to check that
		\begin{align}
			F\colon \Hom_{\widetilde{\CPar}(\textbf{x})}([k],[l])\to \Hom_{{ \CPar}(\textbf{x})}(k,l)
		\end{align}
		is an isomorphism. 
		
		Let $d$ be a colored $(k,l)$-partition diagram. It is easy to observe that $d$ admits a decomposition
		\begin{align}\label{al:decom}
			\begin{tikzpicture}[scale=1,mycirc/.style={circle,fill=black, minimum size=0.1mm, inner sep = 1.1pt}]
				\node (1) at (-1,0.5) {$d=$};
				\node[mycirc,label=above:{$1$}] (n1) at (0,1) {};
				\node[mycirc,label=above:{$2$}] (n2) at (1,1) {};
				\node[mycirc,label=above:{$3$}] (n3) at (2,1) {};
				\node (n4) at (2.5,1) {$\ldots$};
				\node[mycirc,label=above:{$k$}] (n5) at (3,1) {};
				\node[mycirc,label=below:{$1'$}] (n1') at (0,0) {};
				\node[mycirc,label=below:{$2'$}] (n2') at (1,0) {};
				\node[mycirc,label=below:{$3'$}] (n3') at (2,0) {};
				\node (n4') at (2.5,0) {$\ldots$};
				\node[mycirc,label=below:{$k'$}] (n5') at (3,0) {};
				\path[-,draw] (n1) edge node[left] {$\zeta^{i_1}$} (n1');
				\path[-,draw] (n2) edge node[left] {$\zeta^{i_2}$} (n2');
				\path[-,draw] (n3) edge node[left] {$\zeta^{i_3}$} (n3');
				\path[-,draw] (n5) edge node[left] {$\zeta^{i_k}$} (n5');
			\end{tikzpicture}
			\begin{tikzpicture}[scale=1,mycirc/.style={circle,fill=black, minimum size=0.1mm, inner sep = 1.1pt}]
				\node (1) at (-1,0.5) {$\, \circ \, d'\, \circ$};
				\node[mycirc,label=above:{$1$}] (n1) at (0.5,1) {};
				\node[mycirc,label=above:{$2$}] (n2) at (1.5,1) {};
				\node[mycirc,label=above:{$3$}] (n3) at (2.5,1) {};
				\node (n4) at (3.0,1) {$\ldots$};
				\node[mycirc,label=above:{$l$}] (n5) at (3.5,1) {};
				\node[mycirc,label=below:{$1'$}] (n1') at (0.5,0) {};
				\node[mycirc,label=below:{$2'$}] (n2') at (1.5,0) {};
				\node[mycirc,label=below:{$3'$}] (n3') at (2.5,0) {};
				\node (n4') at (3.0,0) {$\ldots$};
				\node[mycirc,label=below:{$l'$}] (n5') at (3.5,0) {};
				\path[-,draw] (n1) edge node[left] {$\zeta^{j_1}$} (n1');
				\path[-,draw] (n2) edge node[left] {$\zeta^{j_2}$} (n2');
				\path[-,draw] (n3) edge node[left] {$\zeta^{j_3}$} (n3');
				\path[-,draw] (n5) edge node[left] {$\zeta^{j_l}$} (n5');
			\end{tikzpicture}
		\end{align}
		where $d'$ is a $(k,l)$-partition diagram (i.e. a diagram in which all colors are equal to the identity color). %and the partition diagrams appearing on the left and right are given by taking tensor products of powers of the dot morphism.	
		
		From \cite[Theorem~1]{Comes}, we know that $d'$ is in the image of $F$. Also, by the definition of $F$,  partition diagrams appearing on the left and right of the decomposition~\eqref{al:decom} are the images of some tensor products of powers of the dot morphism, so every $d$ is in the image of $F$. Therefore $F$ is full.
		
		We show that $F$ is faithful by proving
		\begin{align}\label{al:map}
			\dim	\Hom_{\widetilde{\CPar}(\textbf{x})}([k],[l]) \leq \dim \Hom_{{ \CPar}(\textbf{x})}(k,l).
		\end{align}
		We first recall certain definitions from \cite{Savage}. Let $S^0_2$ be the cup and $S^2_0$ be cap morphisms (see the part ($ii$) of Remark \ref{rem:extrarel}). For $a,b\in\ZZ_{>0}$, $S^b_{a}$ is the following diagram obtained from compositions and tensor products of merge and split morphisms:
		\begin{align*}
			\resizebox{2cm}{3.5cm}
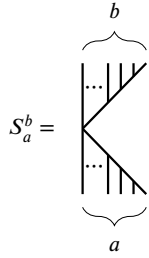
{
				\begin{tikzpicture}
					\node at (-0.8,0) {$S^b_a=$};
					\draw [black, line width= 1pt] (0,-1)--(0,1);
					\draw [black, line width= 1pt] (0,0)--(1,1);
					\draw [black, line width= 1pt] (0,0)--(1,-1);
					\draw [black, line width= 1pt] (0.8,0.8)--(0.8,1);
					\node [scale=1]at (0.18,0.6) {$\cdots$};
					\node [scale=1]at (0.18,-0.6) {$\cdots$};
					\draw [black, line width= 1pt] (0.6,0.6)--(0.6,1);
					\draw [black, line width= 1pt] (0.4,0.4)--(0.4,1);
					\draw [black, line width= 1pt] (0.8,-0.8)--(0.8,-1);
					\draw [black, line width= 1pt] (0.6,-0.6)--(0.6,-1);
					\draw [black, line width= 1pt] (0.4,-0.4)--(0.4,-1);
					\draw[decorate,decoration={brace, amplitude=8 pt}] (0,1.2) to (1,1.2);
					\node [scale=1] at (0.5, 1.8) {$b$};
					\draw[decorate,decoration={brace, mirror, amplitude=8pt}] (0,-1.2) to (1,-1.2);
					\node [scale=1] at (0.5, -1.8) {$a$};
			\end{tikzpicture}}
		\end{align*}
		Let $f$ be a morphism in $\CPar(\textbf{x})$  obtained from generators by taking tensor products and compositions. If we don't involve the dot morphism (this means that we are  working in the partition category), then, from \cite[Theorem~1]{Comes}, the morphism $f$ is of the form $D_{\pi}\circ S\circ D_{\sigma}$ times  a polynomial in $x_0$, where $S$ is a tensor product of the morphisms $S^l_m$ ($l,m\in\ZZ_{\geq 0}$),  and $D_\pi,D_\sigma$ are tensor products of compositions of crosses (i.e. elements of the corresponding symmetric groups). By involving the dot morphism, we get that the morphism $f$ is a dotted version of $D_{\pi}\circ S\circ D_{\sigma}$ times a polynomial in $\textbf{x}$ (this polynomial appears due to the relations involving $(22)$). Now the key observation is that $D_{\pi}\circ S\circ D_{\sigma}$ is a disjoint union of acyclic graphs and hence we can push all the colors into the $D_{\pi}$ and $D_{\sigma}$ components. Consequently, every morphism obtained from generators and relations can be written in the form
		\begin{align}\label{al:deco}
			\begin{tikzpicture}[scale=1,mycirc/.style={circle,fill=black, minimum size=0.1mm, inner sep = 1.1pt}]
				%\node (1) at (-1,0.5) {$d=$};
				\node (n1) at (0,1) {};
				\node(n2) at (1,1) {};
				\node(n3) at (2,1) {};
				\node (n4) at (2.5,0.8) {$\ldots$};
				\node(n5) at (3,1) {};
				\node(n1') at (0,0) {};
				\node (n2') at (1,0) {};
				\node (n3') at (2,0) {};
				\node (n4') at (2.5,0.2) {$\ldots$};
				\node (n5') at (3,0) {};
				\filldraw[black] (0,0.5) circle (1.6pt);
				\filldraw[black] (1,0.5) circle (1.6pt);
				\filldraw[black](2,0.5)circle (1.6pt);
				\filldraw[black](3,0.5) circle (1.6pt);
				\path[-,draw] (n1) edge node[left] {$i_1$} (n1');
				\path[-,draw] (n2) edge node[left] {$i_2$} (n2');
				\path[-,draw] (n3) edge node[left] {$i_3$} (n3');
				\path[-,draw] (n5) edge node[left] {$i_k$} (n5');
			\end{tikzpicture}
			\begin{tikzpicture}[scale=1,mycirc/.style={circle,fill=black, minimum size=0.1mm, inner sep = 1.1pt}]
				\node (1) at (-1,0.5) {$\, \circ \, D_{\pi}\circ S\circ D_{\sigma}\, \circ$};
				\node (n1) at (1,1) {};
				\node(n2) at (2,1) {};
				\node (n3) at (3,1) {};
				\node (n4) at (3.5,0.8) {$\ldots$};
				\node (n5) at (4,1) {};
				\node (n1') at (1,0) {};
				\node (n2') at (2,0) {};
				\node (n3') at (3,0) {};
				\node (n4') at (3.5,0.2) {$\ldots$};
				\node (n5') at (4,0) {};
				\filldraw[black] (1,0.5) circle (1.6 pt);
				\filldraw[black](2,0.5) circle (1.6pt);
				\filldraw[black](3,0.5) circle (1.6pt);
				\filldraw[black](4,0.5)circle (1.6pt);
				\path[-,draw] (n1) edge node[left] {$j_1$} (n1');
				\path[-,draw] (n2) edge node[left] {$j_2$} (n2');
				\path[-,draw] (n3) edge node[left] {$j_3$} (n3');
				\path[-,draw] (n5) edge node[left] {$j_l$} (n5');
			\end{tikzpicture}
		\end{align}
		with uncolored $D_\pi$, $S$ and $D_\sigma$,
		times a polynomial in $\textbf{x}$. Therefore there is a set $\mathcal{S}$  consisting  of (some of) the morphisms of the form \eqref{al:deco}  which spans $\Hom_{\widetilde{\CPar}(\textbf{x})}([k],[l])$. 
		
		If the image of a morphism of the form \eqref{al:deco} under $F$ is a partition diagram (i.e. involved no colors), then, using sliding relations and the relation $(17)$, it is easy to see the morphism is equal to $D_{\pi}\circ S\circ D_{\sigma}$. By \cite[Theorem~1]{Comes}, $F$ is also bijective on morphisms of the form $D_{\pi}\circ S\circ D_{\sigma}$. The tensor products and compositions of the dot morphism is invertible and $F$ is clearly bijective on such morphism.  So  $F$ maps $\mathcal{S}$ to a linearly independent subset of $\Hom_{{ \CPar}({\bf x})}(k,l)$. Thus the map \eqref{al:map} is injective. 
	\end{proof}
	As in the case of partition category, due to Proposition \ref{prop:present}, we can think of every morphism in $\widetilde{\CPar}(\textbf{x})$ is a linear combination of colored partition diagrams. For the sake of convenience, a diagram morphism in the category $\widetilde{\CPar}(\textbf{x})$ will also be called a colored partition diagram.
	
	For $x_0=x_1=\cdots=x_{r-1}=1$, the multiplication of two colored $(k,k)$-partition diagrams is again a colored $(k,k)$-partition diagram. Thus the set $\cPar_k$ of all colored $(k,k)$-partition diagrams is a monoid. We call $\cPar_k$ a \emph{colored partition monoid}. Note that when $r=1$ the monoid $\cPar_k$ is the partition monoid $P_k$. Let $1_k$ denote the identity morphism on $k$. 
	
	\begin{proposition}\label{prop:presentalg}
		The  colored partition monoid $\cPar_k$ has the following presentation.
		The monoid $\cPar_k$ is generated by

		\begin{tikzpicture}[scale=1,mycirc/.style={circle,fill=black, minimum size=0.1mm, inner sep = 1.1pt}]
			\node (1) at (-1,0.5) {$s_0=$};
			\node[mycirc,label=above:{$1$}] (n1) at (0,1) {};
			\node[mycirc,label=above:{$2$}] (n2) at (1,1) {};
			\node[mycirc,label=above:{$3$}] (n3) at (2,1) {};
			\node (n4) at (2.5,1) {$\ldots$};
			\node[mycirc,label=above:{$k$}] (n5) at (3,1) {};
			\node[mycirc,label=below:{$1'$}] (n1') at (0,0) {};
			\node[mycirc,label=below:{$2'$}] (n2') at (1,0) {};
			\node[mycirc,label=below:{$3'$}] (n3') at (2,0) {};
			\node (n4') at (2.5,0) {$\ldots$};
			\node[mycirc,label=below:{$k'$}] (n5') at (3,0) {};
			\path[-,draw] (n1) edge node[left] {$\zeta$} (n1');
			\path[-,draw] (n2) to (n2');
			\path[-,draw] (n3) to (n3');
			\path[-,draw] (n5) to (n5');
			\node at (3.9,0.25) {$,$};
		\end{tikzpicture} \hspace{0.2cm}				
		\begin{tikzpicture}[scale=1,mycirc/.style={circle,fill=black, minimum size=0.1mm, inner sep = 1.1pt}]
			\node (1) at (-1,0.5) {$s_i=$};
			\node[mycirc,label=above:{$1$}] (n1) at (0,1) {};
			\node at (0.5,1) {$\ldots$};
			\node[mycirc,label=above:{$i-1$}] (n4) at (1,1) {};
			\node[mycirc,label=above:{$i$}] (n2) at (2,1) {};
			\node[mycirc,label=above:{$i+1$}] (n3) at (3,1) {};
			\node[mycirc,label=above:{$i+2$}] (n6) at (4.2,1) {};
			\node  at (4.7,1) {$\ldots$};
			\node[mycirc,label=above:{$k$}] (n5) at (5.2,1) {};
			\node[mycirc,label=below:{$1'$}] (n1') at (0,0) {};
			\node[mycirc,label=below:{$(i-1)'$}] (n4') at (1,0) {};
			\node[mycirc,label=below:{$i'$}] (n2') at (2,0) {};
			\node[mycirc,label=below:{$(i+1)'$}] (n3') at (3,0) {};
			\node[mycirc,label=below:{$(i+2)'$}] (n6') at (4.2,0) {};
			\node at (4.7,0) {$\ldots$};
			\node at (0.5,0) {$\ldots$};
			\node[mycirc,label=below:{$k'$}] (n5') at (5.2,0) {};
			\path[-,draw] (n1) to (n1');
			\path[-,draw] (n2) to (n3');
			\path[-,draw] (n3) to (n2');
			\path[-,draw] (n5) to (n5');
			\path[-,draw] (n4) to (n4');
			\path[-,draw] (n6) to (n6');
		\end{tikzpicture} \hspace{0.2cm}	\begin{tikzpicture}
			\node at (0,-1) {};
			\node at (0,1) {};
			\node at (0.9,0.25) {$\text{ for }1\leq i\leq k-1,$};
		\end{tikzpicture}\hspace{0.2cm}

		\begin{tikzpicture}[scale=1,mycirc/.style={circle,fill=black, minimum size=0.1mm, inner sep = 1.1pt}]
			\node (1) at (-1,0.5) {$p_i=$};
			\node[mycirc,label=above:{$1$}] (n1) at (0,1) {};
			\node at (0.5,1) {$\ldots$};
			\node[mycirc,label=above:{$i-1$}] (n4) at (1,1) {};
			\node[mycirc,label=above:{$i$}] (n2) at (2,1) {};
			\node[mycirc,label=above:{$i+1$}] (n3) at (3,1) {};
			\node  at (3.5,1) {$\ldots$};
			\node[mycirc,label=above:{$k$}] (n5) at (4,1) {};
			\node[mycirc,label=below:{$1'$}] (n1') at (0,0) {};
			\node[mycirc,label=below:{$(i-1)'$}] (n4') at (1,0) {};
			\node[mycirc,label=below:{$i'$}] (n2') at (2,0) {};
			\node[mycirc,label=below:{$(i+1)'$}] (n3') at (3,0) {};
			
			\node at (3.5,0) {$\ldots$};
			\node at (0.5,0) {$\ldots$};
			\node[mycirc,label=below:{$k'$}] (n5') at (4,0) {};
			\path[-,draw] (n1) to (n1');
			\path[-,draw] (n3) to (n3');
			\path[-,draw] (n5) to (n5');
			\path[-,draw] (n4) to (n4');
		\end{tikzpicture} \hspace{0.2cm}
		\begin{tikzpicture}
			\node at (0,-1) {};
			\node at (0,1) {};
			\node at (0.9,0.25) {$\text{ for }1\leq i\leq k,$};
		\end{tikzpicture}

		\begin{tikzpicture}[scale=1,mycirc/.style={circle,fill=black, minimum size=0.1mm, inner sep = 1.1pt}]
			\node (1) at (-1,0.5) {$q_i=$};
			\node[mycirc,label=above:{$1$}] (n1) at (0,1) {};
			\node at (0.5,1) {$\ldots$};
			\node[mycirc,label=above:{$i-1$}] (n4) at (1,1) {};
			\node[mycirc,label=above:{$i$}] (n2) at (2,1) {};
			\node[mycirc,label=above:{$i+1$}] (n3) at (3,1) {};
			\node[mycirc,label=above:{$i+2$}] (n6) at (4.2,1) {};
			\node  at (4.7,1) {$\ldots$};
			\node[mycirc,label=above:{$k$}] (n5) at (5.2,1) {};
			\node[mycirc,label=below:{$1'$}] (n1') at (0,0) {};
			\node[mycirc,label=below:{$(i-1)'$}] (n4') at (1,0) {};
			\node[mycirc,label=below:{$i'$}] (n2') at (2,0) {};
			\node[mycirc,label=below:{$(i+1)'$}] (n3') at (3,0) {};
			\node[mycirc,label=below:{$(i+2)'$}] (n6') at (4.2,0) {};
			\node at (4.7,0) {$\ldots$};
			\node at (0.5,0) {$\ldots$};
			\node[mycirc,label=below:{$k'$}] (n5') at (5.2,0) {};
			\path[-,draw] (n1) to (n1');
			\path[-,draw] (n2) to (n2');
			\path[-,draw] (n3) to (n3');
			\path[-,draw] (n2) to (n3);
			\path[-,draw] (n2') to (n3');
			\path[-,draw] (n5) to (n5');
			\path[-,draw] (n4) to (n4');
			\path[-,draw] (n6) to (n6');
		\end{tikzpicture}	\hspace{0.2cm}	\begin{tikzpicture}
			\node at (0,-1) {};
			\node at (0,1) {};
			\node at (0.9,0.25) {$\text{ for }1\leq i\leq k-1$};
		\end{tikzpicture}
		
		and these are subject to the following relations:
		\begin{enumerate}
			\item $s_0^{r}=1_k$, 
			\item $s_0s_1s_0s_1=s_1s_0s_1s_0$,
			\item $s_0s_i=s_is_0$,\,\, for $i\neq 1$,
			\item $s_i^2=1_k$,\,\, for $1\leq i\leq k$, 
			\item  $s_is_{i+1}s_i=s_{i+1}s_is_{i+1}$,\,\, for $1\leq i\leq k-1$, 
			\item $s_is_j=s_js_i$,\, for $1\leq i,j\leq k-1$ and $\lvert i-j\rvert >1$,
			\item $p_i^2=p_i$,\,\, for $1\leq i\leq k$, 
			\item $p_ip_{j}=p_{j}p_i$,\,\, for $1\leq i,j\leq k$,
			\item $s_0p_i=p_is_0$,\,\, for $i\neq 1$,
			\item $s_ip_j=p_js_i$,\,\,for $1\leq i\leq k-1$, $1\leq j\leq k$ and 
			$\lvert i-j\rvert >1$,
			\item $s_ip_i=p_{i+1}s_{i}$,\,\, for $1\leq i\leq k-1$,
			\item $p_is_{i-1}\cdots s_{1}s_0 s_1\cdots s_{i-1}p_i=p_i$,\,\, for $1\leq i\leq k$,
			\item  $ p_ip_{i+1}=p_{i}p_{i+1}s_i$,\,\, for $1\leq i\leq k-1$,
			\item $q_i^2=q_i$,\,\, for $1\leq i\leq k-1$, 
			\item $q_iq_j=q_jq_i$, for $1\leq i,j\leq k-1$,
			\item $s_0q_i=q_is_0$,\,\, for $1\leq i\leq k-1$, 
			\item $s_iq_j=q_js_i$,\,\, for $1\leq i,j\leq k-1$ and $\lvert i-j\rvert >1 $,
			\item $s_is_jq_i=q_js_is_j$,\,\, for $1\leq i,j\leq k-1$ and $\lvert i-j\rvert =1$,
			\item $s_iq_i=q_is_i=q_i$,\,\, for $1\leq i\leq k-1$,
			\item $q_ip_j=p_jq_i$,\,\, for $1\leq i\leq k-1$, $1\leq j\leq k$ and 
			$\lvert i-j\rvert >1$, 
			\item $q_ip_jq_i=q_i$,\,\, for $j=i,i+1$,
			\item $p_jq_ip_j=p_j$,\,\, for $j=i,i+1$.
		\end{enumerate}
	\end{proposition}
	
	\begin{proof}
		Let $Q_k$ be the monoid generated by $a_l$ (in place of $s_l$), for $0\leq l\leq k-1$, then also $b_m$ (in place of $p_m$), for $1\leq m\leq k$, and, finally, $c_n$ (in place of $q_n$), for $1\leq n\leq k-1$, which satisfy the analogues of the relations $(1)-(22)$. Then we have a surjective monoid homomorphism $\pi:Q_k\to \cPar_{k}$. Let $P'_k$ be the submonoid of $Q_k$ generated by $a_{l}$ and $c_{l}$, for $1\leq l\leq k-1$ and $b_n$, for $1\leq n\leq k$. 
		From~\cite[Theorem~1.11(d)]{HR05} and~\cite[Thereom~36]{East11}, it follows that the map $\pi$, restricted to $P'_k$, gives an isomorphism between $P_k'$ and the partition monoid $P_k$.
		
		For $0\leq l\leq k-1$, let $w_{l}=a_{l}\cdots a_{1}a_{0}a_{1}\cdots a_{l}$. In particular, $w_{l+1}=a_{l+1}w_{l}a_{l+1}$. From a combination of  the relations $(2)$, $(3)$ and $(6)$, it follows that the monoid $T_k$ generated by $w_{l}$, for $0\leq l\leq k-1$, is commutative. The map $\pi$ restricted to $T_k$ gives an isomorphism between $T_k$ and $C_r^{k}$, that is the direct product of $C_r$ with itself $k$ times.
		
		We note that $Q_k$ is also generated by $w_l$, for $0\leq l\leq k-1$, then also  $a_i$, for $1\leq i\leq k-1$, then $b_m$, for $1\leq m\leq k-1$ and, finally, $c_{n}$, for $1\leq n\leq k-1$. From the relations $(1)-(22)$, we can derive the following relations:
		\begin{enumerate}
			\item[($23$)] $w_l^{r}=1_k$, for $0\leq l\leq k-1$,
			\item[($24$)] $a_iw_l=w_la_i$, for $1\leq i\leq k-1$, $0\leq l\leq k-1$ and $l\notin \{i-1,i\}$,
			\item[($25$)] $w_{i-1}a_i=a_iw_{i}$, for $1\leq i\leq k-1$,
			\item[($26$)] $b_mw_l=w_lb_m$, for $ 1\leq m\leq k$, $0\leq l \leq k-1$ and $l\neq m-1$,
			\item[($27$)] $c_nw_l=w_lc_n$, for $0\leq l\leq k-1$,  $1\leq n\leq k-1$,
			\item[($28$)] $ b_{i}w_{i-1}c_ib_i=b_iw_i $, for $1\leq i\leq k-1$.
		\end{enumerate}
		
		From the relations $(1)-(28)$, it follows that the monoid $Q_k$ admits the decomposition $T_kP_k T_k$. 
		
		Since the elements of $T_k$ are invertible, to prove injectivity of $\pi$, it is enough to show that, for $w,w'\in T_k$ and $y\in P'_k$, the containment $\pi(wyw')\in P_k$ implies $wyw'=y$.  If $\pi(wyw')\in P_k$, then $\pi(w)\pi(y)\pi(w')=\pi(y)$ (because, if two colored partition diagrams are equal, then the underlying partition diagrams must be the same). This implies that,
		given a part of $\pi(w)\pi(y)\pi(w')$, the multiplication of all the colors in that part is the identity of $C_r$. Now, the following arguments imply that the element $wyw'$ is equal to $y$.
		\begin{itemize}
			\item For nonpropagating parts, one can move colors from left to right and vice-versa. On the side of the monoid $Q_k$, this corresponds locally to applying the relations  $(i)\, b_ib_{i+1}c_{i}w_{i-1}=b_{i}b_{i+1}c_{i}w_{i}$ and  $(ii)\,w_{i-1}c_ib_{i}b_{i+1}=w_{i}c_ib_{i}b_{i+1}$. These relations can be easily obtained  using the relations ($19$) and $(26)$. 
			
			\item For propagating parts, one can move colors from top to bottom and vice-versa. On the side of the monoid $Q_k$, this corresponds locally to the relations that give the effect of moving an element of $T_k$ in an element of $Q_k$ from left to right and vice-versa. A typical generator $w_i$ of $T_k$ commutes with all $c_n$, all $a_l$ for $l\notin\{i, i+1\}$, and all $b_{m}$ for $m\neq i+1$. For the exceptional cases where we don't have commutation relations, we can use the relations ($12$) and ($28$) and the relations $a_iw_{i}=w_{i-1}a_i$ and $a_{i+1}w_{i}=w_{i+1}a_{i+1}$ that follow from the definition of $w_i$.
		\end{itemize}
	\end{proof}
	
	\textbf{Twisted semigroup algebra structure.} For $d_1,d_2\in\cPar_k$, let $\tau(d_1,d_2)\in\CC$ be defined by $$d_1d_2=\tau(d_1,d_2)d_1\circ d_2.$$  For $d_1,d_2,d_3\in\cPar_k$, the associativity of the multiplication of colored partition diagrams gives 
	$$\tau(d_1,d_2)\tau(d_1d_2,d_3)=\tau(d_1,d_2d_3)\tau(d_2,d_3).$$
	So the map $\tau$ is a twisting map on the monoid $\cPar_k$ in the sense of \cite[Definition 3]{Wilcox}. The twisted semigroup algebra of $\cPar_k$ is precisely the multiparameter colored partition algebra $\cPar_k({\bf x})$.
	
	Suppose that $x_i\neq0$, for all $0\leq i\leq r-1$. Using~\cite[Theorem 44]{East11} together with Proposition~\ref{prop:presentalg}, we see that the multiparameter colored partition algebra $\cPar_k(\textbf{x})$ has a presentation as in Proposition~\ref{prop:presentalg} except that the relation $(7)$ is replaced by $p_i^{2}=x_{0}p_i$ and the relation $(22)$ is replaced by $p_is_{i-1}\cdots s_1s_0s_1\cdots s_{i-1}p_i=x_1p_i$, for all $1\leq i\leq k$.
	\begin{remark}
		East~\cite{East20} gave a generic formulation to derive a presentation of diagram categories from the given presentations of monoids of endomorphisms. The natural examples of such setups are diagram monoids and diagram categories. In~\cite{East20},  rigorous proofs for presentations  of several diagram categories were given. Now one can also try to use the presentation of colored partition monoids together with the results in~\cite{East20}, to give a presentation of the multiparameter colored partition category. Recently, Clark and East \cite{Clark} studied wreath products for symmetric inverse monoids and dual symmetric inverse monoids. From the  diagrammatic  viewpoint, a colored partition monoid can be thought of as some kind of  wreath product for a partition monoid.
	\end{remark}
	
	\subsection{Various subcategories of \texorpdfstring{$\CPar(\bf{x})$}{}}\label{sec:variouscat}
	In what follows, we generalize the downward, the upward, the normally downward, and the normally upward partition categories defined in \cite[Section~3.2]{Brv}.  
	
	The {\em colored downward partition category $\CPar^{\sharp}$} is the wide subcategory of $\CPar{({\bf x})}$ whose morphisms are generated by the merge, the cross, the downward leaf and the dot morphisms. In particular, it follows from Proposition~\ref{prop:present} that the morphism space from $[l]$ to $[k]$ in $\CPar^{\sharp}$ is nonzero only if $k\leq l$, in which case it has a basis consisting of those colored $(k,l)$-partition diagrams which have exactly $k$ propagating parts. Analogously, we can define the {\em colored upward partition category $\CPar^{\flat}$} by replacing the merge morphism and the downward leaf morphism in the definition of $\CPar^{\sharp}$ by the split morphism and the upward leaf morphism, respectively. The morphism space from $[l]$ to $[k]$ in $\CPar^{\flat}$ is nonzero only if $k\geq l$, in which case  it has a basis consisting of $(k,l)$-partition diagrams with exactly $l$ propagating parts.
	
	The set $\ZZ_{\geq0}$ of objects of  $\CPar{({\bf x})}$ is partial ordered with respect to the natural order $\leq $ on $\ZZ_{\geq 0}$. Taking this into account, the categories $\CPar^{\sharp}$ and $\CPar^{\flat}$ are downward and upward in the sense of \cite[Section~3.8]{SamSnow}.
	
	For a part $B$ in a $(k,l)$-partition diagram, let $B^u=B\cap \{1,2,\ldots,k\}$ and $B^l=B\cap\{1',2',\ldots,l'\}$.
	
	Let $d$ be an downward $(k,l)$-partition diagram, for $k\leq l$. Then $d$ has $k$ propagating parts. Let $B_1,\ldots, B_k$ be the propagating parts of $d$ ordered so that 
	\begin{displaymath}
		\min{B_1^l} < \cdots < \min{B_k^l}.
	\end{displaymath}
	Then $d$ is called a \emph{normally ordered downward partition diagram} if $B_1=B_1^l\cup\{1\},\ldots, B_k=B_k^l\cup\{k\}$. 
	Similarly, assume that $d$ is a upward $(k,l)$-partition diagram, for $k\geq l$. Then $d$ has $l$ propagating parts. Let $B_1,\ldots, B_l$ be the propgating parts of $d$ so that
	\begin{displaymath}
		\min{B_1^u}<\cdots <\min{B_l^u}.
	\end{displaymath}
	Then $d$ is called a \emph{normally ordered upward partition diagram} if $B_1=B_1^u\cup\{1'\},\ldots,B_l=B_l^u\cup\{l'\}$. This is an adaptation of \cite[p.~15]{Brv}.
	
	We say that an upward (respectively, a downward)  colored partition diagram is normally ordered upward (respectively, downward) if all of its propagating parts are colored with the identity color and its underlying partition diagram is normally ordered upward (respectively, downward). Let $\CPar^{+}$ (respectively, $\CPar^{-}$) denote the wide subcategory of $\CPar^{\sharp}$ (respectively, $\CPar^{\flat}$) whose morphism space consists of the linear combinations of all normally ordered colored downward partition diagrams (respectively, normally ordered upward colored partition diagrams).
	
	It is easy to observe that both categories $\CPar^{+}$ and $\CPar^{-}$ are strictly downward and strictly upward in the sense of \cite[Section 3.8]{SamSnow}.
	\begin{example}
		Consider the following upward partition diagrams:
		
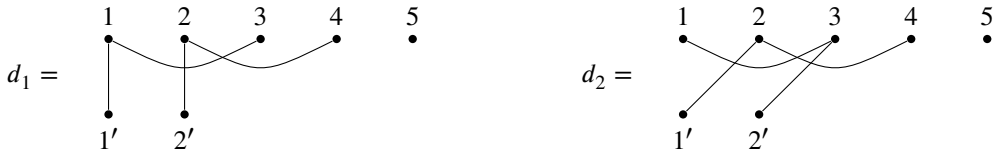
\begin{align*} 
			\begin{tikzpicture}[scale=1,mycirc/.style={circle,fill=black, minimum size=0.1mm, inner sep = 1.1pt}]
				\node (1) at (-1,0.5) {$d_1 =$};
				\node[mycirc,label=above:{$1$}] (n1) at (0,1) {};
				\node[mycirc,label=above:{$2$}] (n2) at (1,1) {};
				\node[mycirc,label=above:{$3$}] (n3) at (2,1) {};
				\node[mycirc,label=above:{$4$}] (n4) at (3,1) {};
				\node[mycirc,label=above:{$5$}] (n5) at (4,1) {};
				\node[mycirc,label=below:{$1'$}] (n1') at (0,0) {};
				\node[mycirc,label=below:{$2'$}] (n2') at (1,0) {};
				\path[-, draw](n1) to (n1');
				\path[-,draw](n2) to (n2');
				\path[-,draw](n1)..controls(1,0.5)..(n3);
				\path[-,draw](n2)..controls(2,0.5).. (n4);
			\end{tikzpicture}
			&&
			\begin{tikzpicture}[scale=1,mycirc/.style={circle,fill=black, minimum size=0.1mm, inner sep = 1.1pt}]
				\node (1) at (-1,0.5) {$d_2 =$};
				\node[mycirc,label=above:{$1$}] (n1) at (0,1) {};
				\node[mycirc,label=above:{$2$}] (n2) at (1,1) {};
				\node[mycirc,label=above:{$3$}] (n3) at (2,1) {};
				\node[mycirc,label=above:{$4$}] (n4) at (3,1) {};
				\node[mycirc,label=above:{$5$}] (n5) at (4,1) {};
				\node[mycirc,label=below:{$1'$}] (n1') at (0,0) {};
				\node[mycirc,label=below:{$2'$}] (n2') at (1,0) {};
				\path[-, draw](n2) to (n1');
				\path[-,draw](n3) to (n2');
				\path[-,draw](n1)..controls(1,0.5)..(n3);
				\path[-,draw](n2)..controls(2,0.5).. (n4);
			\end{tikzpicture}
		\end{align*}
		Then $d_1$ is normally ordered while $d_2$ is not.
	\end{example}
	Let $\CSym$ be the subcategory of $\CPar{({\bf x})}$ whose morphism spaces are generated by the cross and the dot morphisms. There is no nonzero morphism in $\CSym$ from $n$ to $m$, for $n\neq m$. Also, the endomorphism space of an object $n$ in $\CSym$ is isomorphic to the group algebra of the complex reflection group $G(r,n)$ (see Section~\ref{sec:wreath}). Clearly, $\CSym$ is a wide subcategory of both $\CPar^{\flat}$ and $\CPar^{\sharp}$. 
	\subsection{The path algebra and its triangular decomposition}
	The path algebra $\cPar{({\bf x})}$ of $\CPar{({\bf x})}$ is given by
	\begin{displaymath}
		\cPar{({\bf x})}:= \bigoplus_{n,m\in\ZZ_{\geq 0}}\Hom_{\CPar{({\bf x})}}(n,m).
	\end{displaymath}
	For $n\in \ZZ_{\geq 0}$, recall that $1_n$ denotes the identity morphism on the object $n$ in $\CPar{({\bf x})}$.  Then the algebra $\cPar{({\bf x})}$ admits a  decomposition
	\begin{displaymath}
		\cPar{({\bf x})}=\bigoplus_{n,m\in\ZZ_{\geq 0}}1_n\cPar{({\bf x})} 1_m.
	\end{displaymath}
	Analogously, let $\cPar^{\sharp}$, $\cPar^{\flat}$, $\cPar^{+}$, $\cPar^{-}$, and $\cSym$ be the path algebras of $\CPar^{\sharp}$, $\CPar^{\flat}$, $\CPar^{+}$, $\CPar^{-}$, and $\CSym$, respectively. Note that 
	\begin{displaymath}
		\cSym=\bigoplus_{n\in\ZZ_{\geq 0}}\CC[G(r,n)].
	\end{displaymath}
	All the path algebras defined here are locally finite-dimensional and locally unital algebras.
	
	Let $\KK:={\displaystyle{\bigoplus_{n\in\ZZ_{\geq 0}}}}\CC 1_n$, which we regard as a locally unital subalgebra of $\cPar{({\bf x})}$. Then all of the path algebras considered above are $\KK$-$\KK$-bimodules.
	\begin{theorem}[Triangular decomposition]\label{thm:td}
		Each of the following maps, given by multiplication, is an isomorphism
		\begin{align}
			\cPar^{-}\otimes_{\KK} \cSym \otimes_{\KK} \cPar^{+} &\cong \cPar{({\bf x})} \label{al:iso1},\\
			\cPar^{-}\otimes_{\KK}\cSym&\cong \cPar^{\flat}, \label{al:iso2}\\
			\cSym\otimes_{\KK}\cPar^{+} & \cong \cPar^{\sharp},\label{al:iso3}\\
			\cPar^{\flat}\otimes_{\cSym}\cPar^{\mathrm{\sharp}}&\cong \cPar{({\bf x})}. \label{al:T3}
		\end{align}
	\end{theorem}
	\begin{proof}
		Let $d$ be a $(k,l)$-colored partition diagram with $m$ propagating parts. We know that $d$ admits a decomposition \eqref{al:decom}. The partition diagram $d'$ which underlines $d$ decomposes uniquely as $d_1'd_0'd_2'$, where $d_1'$ is a normally ordered upward partition diagram, $d_0'$ is a permutation in $S_m$, and $d_2'$ is a normally ordered downward partition diagram.
		
		Now we want to distribute colors to get a decomposition 
		$d=d_1d_0d_2$, where
		\begin{itemize}
			\item $d'_i$ is the underlying partition diagram for
			$d_i$, for $i=0,1,2$;
			\item $d_1$ is a normally ordered colored upward diagram;
			\item $d_2$ is a normally ordered colored downward diagram.
		\end{itemize}
		We color the parts of 
		$d'_0$ using the colors of the corresponding 
		propagating parts of $d$. This defines $d_0$.
		
		All propagating parts in $d'_1$ and $d'_2$ are, by definition,  colorless. All other parts of $d'_1$
		and $d'_2$ are colored using the same colors as 
		the corresponding parts of $d$. This gives us normally
		ordered colored upward and downward diagrams, respectively.
		
		The decomposition $d=d_1d_0d_2$ follows from the sliding relations. This implies that the map in \eqref{al:iso1}
		is surjective. It is easy to see that the decomposition $d=d_1d_0d_2$ is unique, which implies that the map in
		\eqref{al:iso1} is injective and hence bijective.
		
		The isomorphisms in \eqref{al:iso2} and 
		\eqref{al:iso3} are proved similarly and the
		isomorphism in \eqref{al:T3} is a direct consequence 
		of \eqref{al:iso1}, \eqref{al:iso2}, and \eqref{al:iso3}.
	\end{proof}
	\begin{example} In the following, we illustrate an example of the decomposition of a colored partition diagram mentioned in the course of the above proof:
		
		\resizebox{12cm}{7cm}
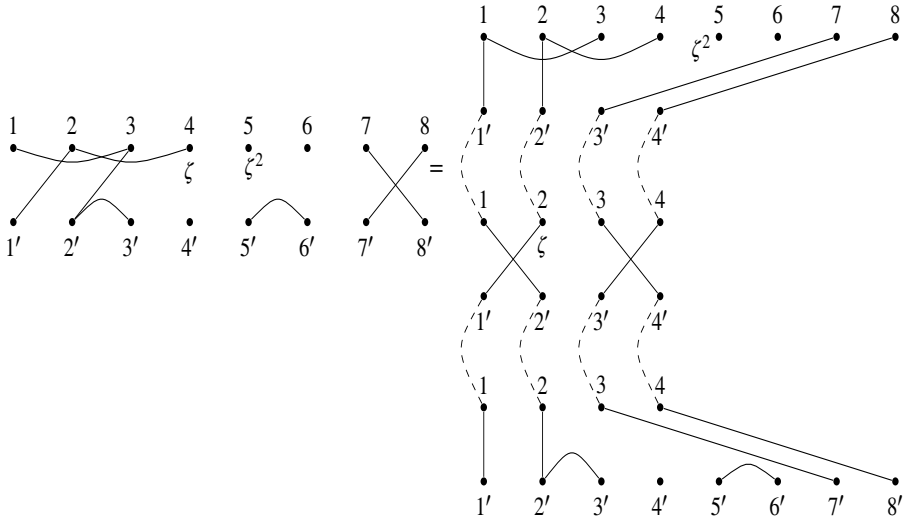
{\begin{tikzpicture}[scale=1,mycirc/.style={circle,fill=black, minimum size=0.1mm, inner sep = 1.1pt}]
				\node[mycirc,label=above:{$1$}] (n1) at (0,0.5) {};
				\node[mycirc,label=above:{$2$}] (n2) at (1,0.5) {}; 
				\node[mycirc,label=above:{$3$}] (n3) at (2,0.5) {};
				\node[mycirc,label=above:{$4$},label=below:{$\zeta$}] (n4) at (3,0.5) {};
				\node[mycirc,label=above:{$5$}] (n5) at (4,0.5) {};
				\node at (4.1,0.25) {$\zeta^2$};
				\node[mycirc,label=above:{$6$}] (n6) at (5,0.5) {};
				\node[mycirc,label=above:{$7$}] (n7) at (6,0.5) {};
				\node[mycirc,label=above:{$8$}] (n8) at (7,0.5) {};
				\node[mycirc,label=below:{$1'$}] (n1') at (0,-0.5) {};
				\node[mycirc,label=below:{$2'$}] (n2') at (1,-0.5) {};
				\node[mycirc,label=below:{$3'$}] (n3') at (2,-0.5) {};
				\node[mycirc,label=below:{$4'$}] (n4') at (3,-0.5) {};
				\node[mycirc,label=below:{$5'$}] (n5') at (4,-0.5) {};
				\node[mycirc,label=below:{$6'$}] (n6') at (5,-0.5) {};
				\node[mycirc,label=below:{$7'$}] (n7') at (6,-0.5) {};
				\node[mycirc,label=below:{$8'$}] (n8') at (7,-0.5) {};
				\draw[-] (n1)..controls(1,0.25)..(n3);
				\draw[-] (n2) to (n1');
				\draw[-] (n2)..controls(2,0.25)..(n4);
				\draw[-] (n3) to (n2');
				\draw[-] (n2')..controls(1.5,-0.1)..(n3');
				\draw[-] (n5')..controls(4.5,-0.1)..(n6');
				\draw[-] (n7) to (n8');
				\draw[-] (n7') to (n8); 
				\node at (7.2,0.2) {=};
				\node[mycirc,label=above:{$1$}] (m1) at (8,2) {};
				\node[mycirc,label=above:{$2$}] (m2) at (9,2) {}; 
				\node[mycirc,label=above:{$3$}] (m3) at (10,2) {}; 
				\node[mycirc,label=above:{$4$}] (m4) at (11,2) {}; 
				\node[mycirc,label=above:{$5$}] (m5) at (12,2) {}; 
				\node at (11.7,1.85) {$\zeta^2$};
				\node[mycirc,label=above:{$6$}] (m6) at (13,2) {}; 
				\node[mycirc,label=above:{$7$}] (m7) at (14,2) {}; 
				\node[mycirc,label=above:{$8$}] (m8) at (15,2) {}; 
				\node[mycirc,label=below:{$1'$}] (m1') at (8,1) {};
				\node[mycirc,label=below:{$2'$}] (m2') at (9,1) {};
				\node[mycirc,label=below:{$3'$}] (m3') at (10,1) {};
				\node[mycirc,label=below:{$4'$}] (m4') at (11,1) {};
				\draw[-] (m1)..controls(9,1.6)..(m3);
				\draw[-] (m1) to (m1');
				\draw[-] (m2)..controls(10,1.6)..(m4);
				\draw[-] (m2) to (m2');
				\draw[-] (m7) to (m3');
				\draw[-] (m8) to (m4');
				
				\node[mycirc,label=above:{$1$}] (p1) at (8,-0.5) {};
				\node[mycirc,label=above:{$2$},label=below:{$\zeta$}] (p2) at (9,-0.5) {};
				\node[mycirc,label=above:{$3$}] (p3) at (10,-0.5) {};
				\node[mycirc,label=above:{$4$}] (p4) at (11,-0.5) {};
				\node[mycirc,label=below:{$1'$}] (p1') at (8,-1.5) {};
				\node[mycirc,label=below:{$2'$}] (p2') at (9,-1.5) {};
				\node[mycirc,label=below:{$3'$}] (p3') at (10,-1.5) {};
				\node[mycirc,label=below:{$4'$}] (p4') at (11,-1.5) {};
				\draw[-] (p1) to (p2');
				\draw[-] (p2) to (p1');
				\draw[-] (p3) to (p4');
				\draw[-] (p4) to (p3');

				\node[mycirc,label=above:{$1$}] (l1) at (8,-3) {};
				\node[mycirc,label=above:{$2$}] (l2) at (9,-3) {};
				\node[mycirc,label=above:{$3$}] (l3) at (10,-3) {};
				\node[mycirc,label=above:{$4$}] (l4) at (11,-3) {};
				\node[mycirc,label=below:{$1'$}] (l1') at (8,-4) {};
				\node[mycirc,label=below:{$2'$}] (l2') at (9,-4) {};
				\node[mycirc,label=below:{$3'$}] (l3') at (10,-4) {};
				\node[mycirc,label=below:{$4'$}] (l4') at (11,-4) {};
				\node[mycirc,label=below:{$5'$}] (l5') at (12,-4) {};
				\node[mycirc,label=below:{$6'$}] (l6') at (13,-4) {};
				\node[mycirc,label=below:{$7'$}] (l7') at (14,-4) {};
				\node[mycirc,label=below:{$8'$}] (l8') at (15,-4) {};
				\draw[-] (l1) to (l1');
				\draw[-] (l2) to (l2');
				\draw[-] (l7') to (l3);
				\draw[-] (l8') to (l4);
				\draw[-] (l2')..controls(9.5,-3.5)..(l3');
				\draw[-] (l5')..controls(12.5,-3.7)..(l6');
				
				\draw[dashed] (p1) .. controls(7.5,0.25)..(m1');
				\draw[dashed] (p2) .. controls(8.5,0.25)..(m2');
				\draw[dashed] (p3) .. controls(9.5,0.25)..(m3');
				\draw[dashed] (p4) .. controls(10.5,0.25)..(m4');
				\draw[dashed] (l1) .. controls(7.5,-2.2)..(p1');
				\draw[dashed] (l2) .. controls(8.5,-2.2)..(p2');
				\draw[dashed] (l3) .. controls(9.5,-2.2)..(p3');
				\draw[dashed] (l4) .. controls(10.5,-2.2)..(p4');
		\end{tikzpicture}}
		
	\end{example}
	The decomposition in Theorem~~\ref{thm:td} is a triangular decomposition but not a split triangular decomposition,  specifically, it does not satisfy (STD6) in~\cite[Remark~5.32]{Strop}: $1_n \cPar^{\flat}1_n$ (or $1_n \cPar^{\sharp}1_n$) is much larger than $\CC 1_n$. From Theorem \ref{thm:td}, we can  immediately see that $\CPar{(\textbf{x})}$ is a triangular category.
	
	\begin{corollary}\label{coro:tringularcat}
		The category $\CPar{(\textbf{x})}$ is a triangular category with the triangular structure given by the pair of wide subcategories $(\CPar^{\flat},\CPar^{\sharp})$. 
	\end{corollary}
	\begin{proof}
		Recall the definition of a triangular category from Section \ref{sec:triancat}. The category $\CPar{(\textbf{x})}$ is a small category with finite-dimensional homomorphism spaces and therefore satisfies ($T_0$). For $n\in\ZZ_{\geq 0}$, the endomorphism algebra  $\End_{\CPar^{\flat}}(n)=\End_{\CPar^{\sharp}}(n)$ is isomorphic to the group algebra of $G(r,n)$, in particular, it is semisimple. Hence the pair $(\CPar^{\flat},\CPar^{\sharp})$ satisfies ($\T_1$). From the discussion in Section \ref{sec:variouscat}, we see that this pair also satisfies ($\T_2$). Finally, the isomorphism \eqref{al:T3} implies  ($\T_3$). Thus $\CPar{(\textbf{x})}$ is a triangular category with the triangular structure given by the pair $(\CPar^{\flat},\CPar^{\sharp})$.
	\end{proof}
	
	\subsection{Standard, co-standard and simple modules}
	Since $\cPar({\bf x})$ admits a triangular decomposition, we can apply \cite[Thereom 5.38]{Strop} to describe its (co)standard modules and classify its simple modules. Note that the Cartan subalgebra $\cSym$ of $\cPar({\bf x})$ is semisimple. The simple modules of $\cSym$ are indexed by the elements of $\mathcal{P}_{r}=\displaystyle{\bigcup_{n\in\ZZ_{\geq 0}}}\mathcal{P}_{r,n}$. For $\OV{\lambda}\in\mathcal{P}_{r,n}$, recall that $S(\OV{\lambda})$ is the corresponding simple module of $\CC[G(r,n)]$ and hence also a simple $\cSym$-module. 
	
	Recall the definitions of the functors $\infl^{\sharp}$ and $\infl^{\flat}$ from Sections~\ref{sub:stand} and~\ref{sub:costand}. For $\OV{\lambda}\in\mathcal{P}_r$, let 
	\begin{displaymath}
		S^{\sharp}(\OV{\lambda})=\infl^{\sharp}(S(\OV{\lambda}))\quad \text{and} \quad S^{\flat}(\OV{\lambda})=\infl^\flat(S(\OV{\lambda})). 
	\end{displaymath}
	Then the corresponding standard and co-standard modules are given by
	\begin{displaymath}
		\Delta(\OV{\lambda})=j_{!}(S^{\sharp}(\OV{\lambda}))=\cPar({\bf x})\otimes_{\cPar^{\sharp}}S(\OV{\lambda})\quad \text{and} \quad  \nabla(\OV{\lambda})=j_{*}(S^{\flat}(\OV{\lambda}))=\Hom_{\cPar^{\flat}}(\cPar(x), S^{\flat}(\OV{\lambda})).
	\end{displaymath}
	For $\OV{\lambda}\in\mathcal{P}_r$, the module $L(\OV{\lambda}):=\hd \Delta(\OV{\lambda})\cong \soc\nabla(\OV{\lambda})$ is the corresponding simple module over $\cPar(\textbf{x})$.
	
	For $\OV{\lambda}$ and $\OV{\mu}$ in $\mathcal{P}_r$, define $\OV{\lambda}\preceq \OV{\mu}$ if either $\OV{\lambda}=\OV{\mu}$ or $\lvert \OV{\lambda}\rvert > \lvert \OV{\mu}\rvert$. Since $\cPar(\textbf{x})$ admits a triangular decomposition and the Cartan subalgebra $\cSym$ is semisimple, the following proposition is a direct consequence~\cite[Corollary~5.39]{Strop} (see also~\cite[Theorem ~3.3]{Brv}). For the definition of an upper finite highest weight category we refer to~\cite[Definition (HW)]{Strop}. 
	
	\begin{proposition}\label{prop:upperfinite}
		The set $\{L(\OV{\lambda})\mid \OV{\lambda}\in\mathcal{P}_r\}$ gives a complete set of pairwise non-isomorphic simple modules of $\cPar(\mathrm{\mathbf{x}})$. The category $\cPar(\mathrm{\mathbf{x}})\Mod$ of locally finite-dimensional left $\cPar(\mathrm{\mathbf{x}})$-modules is an upper finite highest weight category with weight poset given by $(\mathcal{P}_r,\preceq)$. 
	\end{proposition}

	\subsubsection{Anti-involutions on $\cPar({\bf x})$} \label{sec:anti} 
	Like in the case of $G(r,n)$, the algebra $\cPar(\textbf{x})$ also admits two natural anti-involutions. The anti-involution $\Inv_n$ given in Section~\ref{sec:twist}  extends to an anti-involution $\Inv$ on $\cPar({\bf x})$ by sending a colored partition diagram to its horizontal flip and inverting the colors of its parts.

	Next we discuss a simple-preserving duality on $\cPar(\textbf{x})\Mod$ and the key point is to use the duality for finite-dimensional $G(r,n)-$modules with respect to which simple modules of $G(r,n)$ are self-dual. The anti-involution $\bar{\Inv}_n$ given in Section~\ref{sec:twist} is one such anti-involution, which also extends to an anti-involution $\bar{\Inv}$ on $\cPar({\bf x})$ by sending a colored partition diagram to its horizontal flip (but not inverting the colors in contrast to the case of $\Inv$). 
	
	For $M$ a locally finite-dimensional left $\cPar(\textbf{x})$-module, we have $M=\displaystyle{\bigoplus_{n\in\ZZ_{\geq 0}}}1_nM$. Let $$\Theta(M):=\displaystyle{\bigoplus_{n\in\ZZ_{\geq 0}}}(1_nM)^{*},$$ where $(1_nM)^{*}$ denotes the $\CC$-linear dual of $1_nM$. We define an action of $\cPar(\textbf{x})$ on $\Theta(M)$ using the involution $\bar{\Inv}$. For $f\in (1_nM)^{*}$, then $m\in 1_nM$ and $d$ a colored partition diagram, define
	$(d\cdot f)(m)=f(\bar{\Inv}(d)m)$. So $\Theta(M)$ is also a locally finite-dimensional left $\cPar(\textbf{x})$-module. Consequently, we have a contravariant functor
	$$\Theta:\cPar(\textbf{x})\Mod\to \cPar(\textbf{x})\Mod$$
	such that $\Theta^2$ is naturally isomorphic to the identity functor on $\cPar(\textbf{x})\Mod$. Recall the definition $\tilde{\Theta}_n$ from Section~\ref{sec:twist} and let $$\widetilde{\Theta}=\bigoplus_{n\in\ZZ_{\geq 0}}\widetilde{\Theta}_n:\cSym\fdMod\to \cSym\fdMod.$$ We have the following isomorphism of functors:
	\begin{displaymath}
		j_{!}\circ \widetilde{\Theta}\cong\Theta\circ j_{*}\quad \text{and} \quad\Theta\circ j_{!}\cong j_{*}\circ \tilde{\Theta}.
	\end{displaymath}
	Combining with above isomorphisms and $\widetilde{\Theta}(S(\OV{\lambda}))\cong S(\OV{\lambda})$, we conclude that $$\Theta(\Delta(\OV{\lambda}))\cong \nabla(\OV{\lambda}),\quad \Theta(\nabla(\OV{\lambda}))\cong \Delta(\OV{\lambda}),\quad \text{and } \Theta(L(\OV{\lambda}))\cong L(\OV{\lambda}).$$
	\section{Structure constants for bases in various Grothendieck rings}
	Since $\CPar(\textbf{x})$, $\CPar^{\sharp}$ and $\CSym$ admit monoidal structures, it follows from the discussion in Section~\ref{sec:indupro} that the categories $\cPar(\textbf{x})\BMod$, $\cPar^{\sharp}\BMod$ and $\cSym\BMod$ also admit respective monoidal structures. Then it follows from Section \ref{sec:Groth}, that the split Grothendieck groups $K_0(\cPar(\textbf{x}))$ and $K_{0}(\cPar^{\sharp})$ are rings. We will show that, as rings, $K_0(\cPar(\textbf{x}))\cong K_0(\cPar^{\sharp})\cong K_0(\cSym)$ and describe the structure constants for various bases.
	
	The study of colored downward partition category and its path algebra will play a key role in deriving the main results of this section. 
	\subsection{The colored downward partition category}
	This subsection is strongly inspired by~\cite[Section~3.4]{Brv}. We begin with a few definitions and notation which are relevant throughout this section. 
	\begin{definition}\label{def:add}
		Let $l,m,n\in \ZZ_{\geq 0}$. A nonnegative integer $t$ is called admissible if 
		$$0\leq t\leq \min(l+m-n,l+n-m,m+n-l) \quad \text{and} \quad
		t\equiv l+m+n\, (\text{mod }2).$$ Denote by $\mathcal{D}_{l,m,n}$ the set of all admissible $t$. Also, for $t\in \mathcal{D}_{l,m,n}$ define
		\begin{align*}
			a:=\frac{m+n-l-t}{2},  \quad  b:= \frac{l+n-m-t}{2}, \quad c:=\frac{l+m-n-t}{2}. 
		\end{align*}Note that such $a,b,c$ are nonnegative integers that satisfy $b+t+c=l$, $c+t+a=m$ and $a+t+b=n$.        Consider the following partition diagram:
		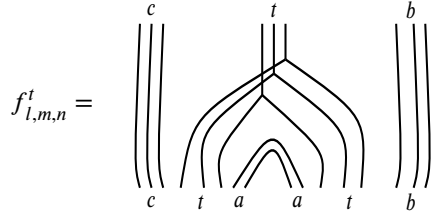
\begin{figure}[h!]
			\centering
			\begin{tikzpicture}
				\node at (-3,0) {$f^{t}_{l,m,n}=$};
				\node at (0,0) {$\diagha$};
			\end{tikzpicture}
			\caption{The element $f^{t}_{l,m,n} $}
			\label{ft}
		\end{figure}
	\end{definition}
	Given $w\in S_a$, we can define a permutation $\sigma_{w}$ on $1,2,\ldots, m,m+1,\ldots,m+n$ such that, for $1\leq i\leq a$, we have $\sigma_{w} (1+m-i)= 1+m -w(i)$ and  $\sigma_{w}(i+m)=w(i)+m$, moreover,  $\sigma_{w}$ fixes all other elements. Intuitively, on the elements $1+m-a, 1+m-(a-1),\ldots, 1+m-1=m$, the permutation $\sigma_w$ is given by $w$ and, on the elements $m+1,\ldots, m+a$, the permutation $\sigma_w$ is also given by $w$. This is, clearly, an embedding of groups from $S_t$ to $S_m\times S_n$. Note that $\CC[S_m\times S_n]$ canonically embeds inside $1_m\cPar^{\sharp}\otimes 1_n\cPar^\sharp$. From all of these, we have an embedding of algebras 
	\begin{align}
		\CC [S_a] &\to 1_m\cPar^{\sharp}\otimes 1_n\cPar^{\sharp}  \text { given by } \\
		w & \mapsto \sigma_{w}, \text{ where } w\in S_a.
	\end{align}
	Let $1^{t}_{l,m,n}$ be the image of the idempotent $\frac{1}{a!}\displaystyle{\sum_{w\in s_a}} w\in \CC [S_a]$ under the above embedding.
	
	The following lemma was first proved for the setup of the downward Brauer category, the downward walled Brauer category and the downward partition category in \cite{SamSnow}. Recently, a more detailed proof for the downward partition category is given in \cite[Lemma 3.5]{Brv} and we observe that their proof extends to our setup as well. 
	\begin{lemma}\label{lm:downproj}
		For $l,m,n\in\ZZ_{\geq 0}$, let $\mathcal{D}_{l,m,n}$ be the set as in Definition \ref{def:add}. For $t\in \mathcal{D}_{l,m,n}$, let $a,b,c$  be the nonnegative integers as in Definition \ref{def:add}, $1^{t}_{l,m,n}$ be as above and $f^{t}_{l,m,n}$ be as defined in Figure~\ref{ft}. Let $S_l/(S_c\times S_b\times S_t)$ be a set of coset representatives of the subgroup $S_c\times S_b\times S_d$ inside the group $S_l$.
		Then there is a right $\cPar^{\sharp}\otimes \cPar^{\sharp}$-module isomorphism
		\begin{align}\label{al:iso}
			\bigoplus_{l,m,n\in \ZZ_{\geq 0}}\, \bigoplus_{t\in\mathcal{D}_{l,m,n}}\, \bigoplus_{g\in S_{l}/S_c\times S_b\times S_t} 1^{t}_{l,m,n}(\cPar^\sharp\otimes\cPar^\sharp) \to \cPar^\sharp 1_{\star},
		\end{align}
		given by sending $1^{t}_{l,m,n}$ in the $g$-th summand on the left hand side to $g\circ f^t_{l,m,n}$. Consequently, the right $\cPar^\sharp\otimes\cPar^\sharp$-module $\cPar^\sharp1_{\star}$ is projective.
	\end{lemma}
	\begin{proof}
		For $r=1$, the fact that the map \eqref{al:iso} is an isomorphism follows from \cite[Lemma 3.5]{Brv}. Then, for an arbitrary $r$, the proof follows from the following elementary observations. We have the decomposition
		\begin{displaymath}
			\cPar^{\sharp}1_{\star}=\bigoplus_{l,m,n\in\ZZ_{\geq 0}} 1_{l}\cPar^\sharp 1_{m\star n}.
		\end{displaymath}
		Let $d\in 1_l\cPar^\sharp1_{m\star n}$ be a colored partition diagram. Then $d$ has $l$ propagating parts, in particular, every elements in the top row of $d$ belongs to a propagating part. This allows us to write $d=d_1\circ d_2 $, where $d_1\in 1_l\parti^\sharp 1_{m\star n}$ and $d_2$ is just $m+n$ vertical lines possibly with some colors (equivalently, $d_2\in C_r^{m+n}$, the direct product of $C_r$ with itself $m+n$ many times). So $1_{l}\cPar^{\sharp}1_{m\star n}$ has a diagram basis whose elements are of the form $d_1\circ d_2$ with $d_1$ and $d_2$ as above. Now, the surjectivity of the map \eqref{al:iso} follows from the $r=1$ case. 
		
		Furthermore, for $d_1, \tilde{d}_1 \in 1_l\parti^\sharp 1_{m+n}$ with $d_1\neq \tilde{d}_1$ and $d_2, \tilde{d}_2\in C_r^{m+n}$, we note that $d_1\circ \tilde{d}_1\neq d_2\circ \tilde{d}_2$ (in other words, if the underlying partition diagrams are different then the corresponding colored partition diagrams must be different regardless of their colors). For $d_3,\tilde{d}_3\in C_r^{m}$ and $d_4,\tilde{d}_4\in C_r^n$, if $d_1\circ (d_3\otimes \tilde{d}_{3})\neq d_1\circ(d_4\otimes \tilde{d}_{4})$ on the left hand side \eqref{al:iso}, then it follows from the compositions of diagrams and the definition of the monoidal structure $\star$ that $d_1\circ (d_3\star \tilde{d}_3)\neq d_1\circ (d_4\star \tilde{d}_{4})$ on the right hand side \eqref{al:iso}. This gives the injectivity of the map \eqref{al:iso}.
	\end{proof}
	
	As a consequence, by Lemma \ref{lm:downproj}, the induction product \eqref{al:cstar} for $\cPar^\sharp$ is biexact.
	\subsubsection{A permutation module}\label{sec:permu}
	Recall that $a$, $b$ and $c$ are functions of $l$, $m$, $n$ and $t$.
	
	Let $X_{l,m,n}^{t}$ be the set of all colored partitions of 
	$\{1,\ldots, l\}\sqcup\{1',\ldots,m'\}\sqcup\{1'',\ldots,n''\}$ such that exactly $a$ parts are of the form $\big(\{j',k''\},\zeta^{s_1}\big)$, exactly $b$ parts are of the form $\big(\{i,k''\},\zeta^{s_2}\big)$, exactly $c$ parts are of the form $\big(\{i,j'\},\zeta^{s_3}\big)$ and the remaining $t$ parts are of the form $\big(\{i,j',k''\},\zeta^{s_3}\big)$, for $i\in\{1,\ldots,l\}$, $j'\in\{1',\ldots,m'\}$, $k''\in\{1'',\ldots,n''\}$ and $0\leq s_1,s_2,s_3,s_4\leq r-1$. We illustrate disjoint parts of $X^{t}_{l,m,n}$ pictorially in Figure~\ref{fig:X}:
	
	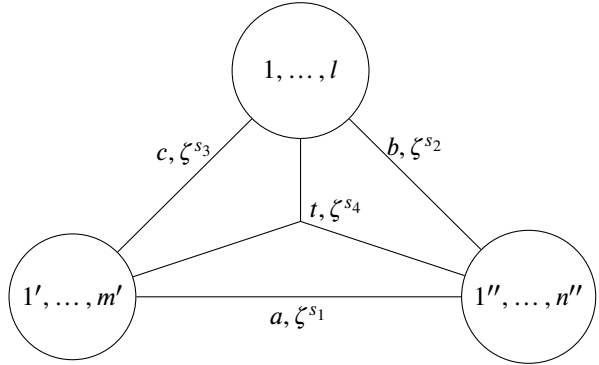
\begin{figure}[htb]
		\centering
		\begin{tikzpicture}
			\node[circle,draw, minimum size=1cm] (A) at  (0,0) {$1',\ldots,m'$};
			\node[circle,draw, minimum size=1.8cm] (B) at  (3,3)  {$1,\ldots, l$};
			\node[circle,draw, minimum size=1cm] (C) at  (6,0)  {$1'',\ldots, n''$};
			\draw (A) -- (B) node[midway, above=0.15cm]{$c,\zeta^{s_3}$};
			\draw (B) -- (C) node[midway, above=0.2cm]{$b,\zeta^{s_2}$};
			\draw (A) -- (C) node[midway, below]{$a,\zeta^{s_1}$};
			\draw (A)--(3,1);
			\draw (B)--(3,1);
			\draw (C)--(3,1) node[above=0.14cm, right]{$t,\zeta^{s_4}$};
		\end{tikzpicture}
		\caption{A part of $X^{t}_{l,m,n}.$}
		\label{fig:X}
	\end{figure}

	Let $G:=G(r,l)\times \big(G(r,m)\times G(r,n)\big)^\mathrm{op}$. Next, we define a left action of the group $G$ on the set $X^t_{l,m,n}$. Let $(h_1,\sigma_1)\in G(r,l)$, $(h_2,\sigma_2)\in G(r,m)$ and $(h_3,\sigma_3)\in G(r,n)$. Note that here we regard $\sigma_1, \sigma_2$ and $\sigma_3$ as permutations on $\{1,\ldots, l\}, \{1',\ldots, m'\} $ and $\{1'',\ldots, m''\}$, respectively. Likewise, $h_1, h_2$ and $h_3$ are now functions from $\{1,\ldots, l\}, \{1',\ldots, m'\}$ and $\{1'',\ldots,m''\}$, respectively, to $C_r$.

	Let $g:=\big((h_1,\sigma_1),(h_2,\sigma_2),(h_3,\sigma_3)\big)\in G$. Define an action of $g$ on typical parts of  elements of $X^t_{l,m,n}$ as follows: 
	\begin{align}
		g\big(\{j',k''\},\zeta^{s_1}\big) &:=\big(\{\sigma_2^{-1}(j'),\sigma_3^{-1}(k'')\}, h_2(j')h_3(k'')\zeta^{s_1}\big),\\[1ex]
		g\big(\{i,k''\},\zeta^{s_2}\big) &:= \big(\{\sigma_1(i),\sigma_3^{-1}(k'')\}, h_1(\sigma_1(i)) h_3(k'') \zeta^{s_2}\big),\\[1ex]
		g\big(\{i,j'\},\zeta^{s_3}\big) &:= \big(\{\sigma_1(i),\sigma_2^{-1}(j')\},h_1(\sigma_1(i))h_{2}(j')\zeta^{s_3}\big),\\[1ex]
		g\big(\{i,j',k''\},\zeta^{s_4}\big) &:= \big(\{\sigma_1(i),\sigma_2^{-1}(j'),\sigma_3^{-1}(k'')\}, h_1(\sigma_1(i))h_2(j')h_3(k'') \zeta^{s_4}\big).
	\end{align}
	Then, by letting $g$ act on each part of an element $d$ in $X^t_{l,m,n}$, gives an action of $g$ on $d$. It is easy to see, from the definitions, that the action of $G$ on $X^{t}_{l,m,n}$ is transitive, in particular, it is uniquely determined by the stabilizer of an element. Below we find the stabilizer of an element in $X^{t}_{l,m,n}$ under this action.
	\subsubsection{The stabilizer}\label{sec:stab}  Inside $G(r,l)$,  we have the parabolic subgroup $G(r,b)\times G(r,t)\times G(r,c)$ such that the permutations in $G(r,b), G(r,t)$ and $G(r,c)$ act on $\{1,\ldots, b\}, \{b+1,\ldots, b+t\},$ $\{b+t+1,\ldots, b+t+c\}$, respectively. Analogously, $G(r,m)$ and $G(r,n)$ have the parabolic subgroups
	$G(r,c)\times G(r,t)\times G(r,a)$ and $G(r,a)\times G(r,t)\times G(r,b)$, respectively. So we have 
	the following parabolic subgroup $P$ of $G$
	\begin{align*}
		\big(G(r,b)\times G(r,t)\times G(r,c)\big)\times \left(
		\big(G(r,c)\times G(r,t)\times G(r,a)\big)\times 
		\big(G(r,a)\times G(r,t)\times G(r,b)\big)\right)^{\mathrm{op}}.
	\end{align*}
	
	%Given $(h,\sigma)$, let $\tilde{h}=h_{\sigma^{-1}}$ and $\hat{h}=(\tilde{h})^{-1}$.
	%Note that  $(\hat{h},\sigma^{-1})$ is the inverse of $(h,\sigma)$.

	Recall that $H(r,t)=(C_r\times C_r)\wr S_t$. The group $$L:=G(r,a)\times G(r,b)\times G(r,c)\times H(r,t)$$ 
	embeds inside the group $P$ by sending the element $\big((\mathsf{x},\sigma),(\mathsf{y},\tau),(\mathsf{z},\eta), ((\mathsf{w},\mathsf{u}),\xi)\big)$
	of $L$ to the following element of the group $G$
	\begin{equation}\label{al:embedd}
		\big((\mathsf{y},\tau), (\mathsf{w},\xi),(\mathsf{z},\eta); (\mathsf{z}_{\eta^{-1}}^{-1},\eta^{-1}), (\mathsf{u}_{\xi^{-1}},\xi^{-1}), (\mathsf{x}_{\sigma^{-1}}^{-1},\sigma^{-1}); (\mathsf{x}_{\sigma^{-1}},\sigma^{-1}), (\mathsf{w}_{\xi^{-1}}^{-1}\mathsf{u}_{\xi^{-1}}^{-1},\xi^{-1}), (\mathsf{y}_{\tau^{-1}}^{-1},\tau^{-1}) \big).
	\end{equation}
	The map \eqref{al:embedd} is an embedding of groups since  for $s\in\ZZ_{> 0}$, the maps  $\Inv_{s}$, $\bar{\Inv}_{s}$ (see Section~\ref{sec:twist}) and the following maps are group homomorphisms:
	\begin{align}
		\psi_{1}:& H(r,t) \to G(r,t)\label{al:1}\\
		&((\mathsf{w},\mathsf{u}),\xi)\mapsto (\mathsf{w},\xi);\nonumber\\
		\psi_{2}:&H(r,t) \to G(r,t)\label{al:2}\\
		&((\mathsf{w},\mathsf{u}),\xi) \mapsto (\mathsf{u},\xi);\nonumber\\
		\psi_{3}:&H(r,t) \to G(r,t)\label{al:3}\\
		&((\mathsf{w},\mathsf{u}),\xi) \mapsto (\mathsf{w}\mathsf{u},\xi).\nonumber
	\end{align}
	It follows that the stabilizer of the element in $X^t_{l,m,n}$ whose parts are given by
	\begin{align*}
		&\underbrace{\big(\{(c+t+1)', 1''\},e\big),\ldots, \big(\{(c+t+a)',a''\},e\big)}_{a \text{ parts }},\\
		& \underbrace{\big(\{1,(a+t+1)''\}, e\big), \ldots, \big(\{b, (a+t+b)''\},e\big)}_{b \text{ parts}},\\
		&\underbrace{\big(\{(b+t+1), 1'\},e\big) \ldots, \big(\{(b+t+c), c'\},e\big)}_{c \text{ parts}},\\
		&\underbrace{\big(\{b+1,(c+1)',(a+1)''\}, e\big),\ldots, \big(\{b+t, (c+t)',(a+t)''\}, e\big)}_{t \text{ parts}}
	\end{align*}
	is the subgroup $L$ of $P$. The group homomorphisms (\ref{al:1} -- \ref{al:3}) are surjective, so the pulled back representations of irreducible representations are again irreducible. Recall that the irreducible representations of $H(r,t)$ are indexed by the elements of the set \eqref{al:set}.

	For $\OV{\lambda}\in \mathcal{P}_{r,t}$, we identify the indexes of the irreducible representations $S(\OV{\lambda})^{\psi_1},S(\OV{\lambda})^{\psi_{2}}$ and $S(\OV{\lambda})^{\psi_3}$. A proof of this identification of the indexes can be obtained directly from the construction of irreducible representations discussed in Section~\ref{sec:wreath}.
	\begin{enumerate}
		\item[(i)] $S(\OV{\lambda})^{\psi_1}$ is indexed by the following element $ \left(\mu^{(p,q)}\right)_{1\leq p,q\leq r}$ in the set \eqref{al:set}
		\begin{align*}
			\mu^{(p,q)}=\begin{cases}
				\lambda^{(p)}, & \text{ if } q=1;\\
				\varnothing, & \text{ otherwise}.
			\end{cases}
		\end{align*}
		\item[(ii)] $S(\OV{\lambda})^{\psi_2}$ is indexed by the following element $ \left(\mu^{(p,q)}\right)_{1\leq p,q\leq r}$ in the set \eqref{al:set}
		\begin{align*}
			\mu^{(p,q)}=\begin{cases}
				\lambda^{(q)}, & \text{ if } p=1;\\
				\varnothing, & \text{ otherwise}.
			\end{cases}
		\end{align*}
		\item[(iii)] $S(\OV{\lambda})^{\psi_3}$ is indexed by the following element $ \left(\mu^{(p,q)}\right)_{1\leq p,q\leq r}$ in the set \eqref{al:set}
		\begin{align*}
			\mu^{(p,q)}=\begin{cases}
				\lambda^{(p)}, & \text{ if } q=p;\\
				\varnothing, & \text{ otherwise}.
			\end{cases}
		\end{align*}
	\end{enumerate}
	For $\OV{\delta},\OV{\delta'},\OV{\delta''}\in\mathcal{P}_{r,t}$, note that the action of $H(r,t)$ on $S(\OV{\delta})^{\psi_1}\otimes S(\OV{\delta'})^{\psi_2}$ is diagonal. Let $K^{\OV{\delta''}}_{\OV{\delta},\OV{\delta'}}$ denote the multiplicity of $S(\OV{\delta''})^{\psi_3}$ in $S(\OV{\delta})^{\psi_1}\otimes S(\OV{\delta'})^{\psi_2}$.

	For a module $M$ and a simple module $N$, denote by $[M:N]$ the composition multiplicity of $N$ in $M$.
	The group $G$ acts on the left on the vector space $\CC X^t_{l,m,n}$ and the next lemma computes the multiplicitity of each simple $G$-module in $\CC X^{t}_{l,m,n}$.
	
	\begin{lemma}\label{lm:multi}
		For $l,m,n\in\ZZ_{\geq 0}$ and $t\in\mathcal{D}_{l,m,n}$, let $a,b,c$ be as in Definition \ref{def:add}. Then the multiplicity of the simple $\CC[G(r,l)]\times \CC[G(r,m)^{\op}]\times \CC[G(r,n)^{\mathrm{op}}]$-module $S(\OV{\lambda})\otimes S(\OV{\mu})^{*}\otimes S(\OV{\nu})^{*}$ in the module $\CC X^t_{l,m,n}$, where $\OV{\lambda}\in\mathcal{P}_{r,l}$, $\OV{\mu}\in\mathcal{P}_{r,m}$ and $\OV{\nu}\in\mathcal{P}_{r,n}$, is given by
		\begin{align}\label{al:multi}
			\big[\CC X^t_{l,m,n}:S(\OV{\lambda})\otimes S(\OV{\mu})^{*}\otimes S(\OV{\nu})^{*}\big]=
			\sum_{\OV{\alpha}\in\mathcal{P}_{r,a}} \sum_{\OV{\beta}\in\mathcal{P}_{r,b}}
			\sum_{\OV{\gamma}\in\mathcal{P}_{r,c}}
			\sum_{\OV{\delta},\OV{\delta'},\OV{\delta''}\in\mathcal{P}_{r,t}}
			\LR^{\OV{\lambda}}_{\OV{\beta},\OV{\delta},\OV{\gamma}} \LR^{\OV{\mu}}_{\OV{\gamma},\OV{\delta'},\OV{\alpha}}
			\LR^{\OV{\nu}}_{\OV{\alpha},\OV{\delta''},\OV{\beta}} K^{\OV{\delta''}}_{\OV{\delta},\OV{\delta'}}.
		\end{align}
	\end{lemma}
	\begin{proof}
		Since $G$ acts transitively on $X^{t}_{l,m,n}$, we can write $\CC X^{t}_{l,m,n}$ as the representation induced from the trivial representation $\CC$ of the stabilizer of an element in $X^{t}_{l,m,n}$. From Section \ref{sec:stab}, we know that $L$ is the stabilizer of an element in $X^{t}_{l,m,n}$. So
		\begin{displaymath}
			\CC X^{t}_{l,m,n}=\Ind^{\CC [G]}_{\CC [L]}(\CC).
		\end{displaymath}
		By Frobenius reciprocity, we have
		\begin{displaymath}
			\Hom_{\CC [G]}\big(\CC X^{t}_{l,m,n},S(\OV{\lambda})\otimes S(\OV{\mu})^{*}\otimes S(\OV{\nu})^{*}\big)=
			\Hom_{\CC [L]}\big(\CC, S(\OV{\lambda})\otimes S(\OV{\mu})^{*}\otimes S(\OV{\nu})^{*}\big).
		\end{displaymath}
		So, the multiplicity $\big[\CC X^{t}_{l,m,n}: S(\OV{\lambda})\otimes S(\OV{\mu})^{*}\otimes S(\OV{\nu})^{*}\big]$ is equal to the dimension of the following $L$-invariant subspace:  $\dim \big(S(\OV{\lambda})\otimes S({\OV{\mu}})^{*}\otimes S({\OV{\nu}})^{*}\big)^{L}$.
		
		Since $L\subset P \subset G$, the restriction directly  from $G$ to $L$ is the same as
		first the restriction from $G$ to $P$ and then the restriction from $P$ to $L$. The restriction of  $S(\OV{\lambda})\otimes S(\OV{\mu})^{*}\otimes S(\OV{\nu})^{*}$ to $P$ decomposes as
		
		\hspace{1cm}\resizebox{16cm}{!}{
			$V:=\displaystyle{\bigoplus_{\OV{\alpha},\OV{\alpha'}\in\mathcal{P}_{r,a}}}\,\,
			\displaystyle{\bigoplus_{\OV{\beta},\OV{\beta'}\in\mathcal{P}_{r,b}}}\,\,
			\displaystyle{\bigoplus_{\OV{\gamma},\OV{\gamma'}\in\mathcal{P}_{r,c}}}\,\,
			\displaystyle{\bigoplus_{\OV{\delta},\OV{\delta'},\OV{\delta''}\in\mathcal{P}_{r,t}}}
			\big(S(\OV{\beta})\otimes S(\OV{\delta})\otimes S(\OV{\gamma'}) \big)
			\otimes \big(S(\OV{\gamma})^{*}\otimes S(\OV{\delta'})^{*}\otimes S(\OV{\alpha'})^{*}\big)
			\otimes \big (S(\OV{\alpha})^{*}\otimes S(\OV{\delta''})^{*}\otimes S(\OV{\beta'})^{*}\big)^{\oplus N},$
		}
		
		where $N=\LR^{\OV{\lambda}}_{\OV{\beta},\OV{\delta},\OV{\gamma'}} \LR^{\OV{\mu}}_{\OV{\gamma},\OV{\delta'},\OV{\alpha'}}
		\LR^{\OV{\nu}}_{\OV{\alpha},\OV{\delta''},\OV{\beta'}}$. Now the dimension of the necessary $L$-invariant subspace can be found by considering
		\begin{align}\label{al:hom}
			\Hom_{\CC [L]}\big(\CC,\big(S(\OV{\beta})\otimes S(\OV{\delta})\otimes S(\OV{\gamma'}) \big)
			\otimes \big(S(\OV{\gamma})^{*}\otimes S(\OV{\delta'})^{*}\otimes S(\OV{\alpha'})^{*}\big)
			\otimes \big (S(\OV{\alpha})^{*}\otimes S(\OV{\delta''})^{*}\otimes S(\OV{\beta'})^{*}\big)\big).
		\end{align}
		
		Now, by the definition of $L$,
		see \eqref{al:embedd}, the space \eqref{al:hom} is equal to the tensor product of the following homomorphism spaces
		\begin{itemize}
			\item[(i)] $\Hom_{\CC [G(r,a)]}\big(\CC, (S(\OV{\alpha'})^{*})^{\Inv_a}\otimes (S(\OV{\alpha})^{*})^{\bar{\Inv}_a}\big)\cong \Hom_{\CC [G(r,a)]}\big( S(\OV{\alpha}) , S(\OV{\alpha'} \big)$,
			
			\item[(ii)] $\Hom_{\CC [G(r,b)]}\big(\CC, S(\OV{\beta})\otimes (S(\OV{\beta'})^{*})^{\Inv_b}\big)\cong 
			\Hom_{\CC [G(r,b)]}\big( S(\OV{\beta'}), S(\OV{\beta})\big)$,
			
			\item[(iii)]$\Hom_{\CC [G(r,c)]}\big(\CC, (S(\OV{\gamma})^{*})^{\Inv_c}\otimes S(\OV{\gamma'})\big) 
			\cong \Hom_{\CC [G(r,c)]}\big( S(\OV{\gamma}),  S(\OV{\gamma'})\big)$, 
			
			\item[(iv)] $\Hom_{\CC [H(r,t)]}\big(\CC, S(\OV{\delta})^{\psi_1}\otimes (S(\OV{\delta'})^{*})^{\bar{\Inv}_t\circ\psi_2}\otimes (S(\OV{\delta''})^{*})^{\Inv_t\circ\psi_3}\big)\cong \Hom_{\CC [H(r,t)]}\big( S(\OV{\delta''})^{\psi_3}, S(\OV{\delta})^{\psi_1}\otimes S(\OV{\delta'})^{\psi_2}\big)$,
		\end{itemize}
		where the maps $\Inv_s$ and $\bar{\Inv}_s$ are given in Section~\ref{sec:twist}, for $s=a,b,c,t$, and the maps $\psi_1,\psi_2,\psi_3$ are given in the equations (\ref{al:1} -- \ref{al:3}). The action of the corresponding group in each homomorphism space in (i)-(iv) is diagonal. 
		%Note that the hats over $h$, appearing in \eqref{al:embedd}, play an important role in the computation of (b)-(d).
		
		From Section \ref{sec:twist},  the pull back representations $\big(S(\OV{\alpha'})^{*}\big)^{\Inv_a},  \big(S(\OV{\beta'})^{*}\big)^{\Inv_b}$ and $\big(S(\OV{\gamma})^{*}\big)^{\Inv_c}$ are
		isomorphic to the dual representations corresponding to $S(\OV{\alpha'}), S(\OV{\beta'})$ and  $S(\OV{\gamma})$ respectively. Moreover, again from Section~\ref{sec:twist}, the pull back representation 
		$\big(S(\OV{\alpha'})^{*}\big)^{\bar{\Inv}_a}$ is isomorphic to $S(\OV{\alpha'})$. 
		By Schur's lemma, the homomorphism spaces in (i)-(iii) are nonzero if and only if 
		\begin{align*}
			\alpha=\alpha', \beta=\beta', \gamma=\gamma',
		\end{align*}
		and, in each of these cases, the nonzero homomorphism space is one dimensional. 
		
		As above, the pull back representation $S(\OV{\delta''})^{*})^{\Inv_t\circ\psi_3}$ is isomorphic to the dual representation of corresponding to $S(\OV{\delta''})^{\psi_3}$ and the pull back representation $(S(\OV{\delta'})^{*})^{\bar{\Inv}_t\circ\psi_2}$ is isomorphic to $S(\OV{\delta'})^{\psi_2}$. So the dimension of the homomorphism space in (iv) is $K^{\OV{\delta''}}_{\OV{\delta},\OV{\delta'}}$. Consequently, the dimension of the $L$-invariant subspace inside  $V$ is precisely given by the right hand side of \eqref{al:multi}. 
	\end{proof}
	
	Set 
	\begin{equation}\label{al:reduced}
		R^{\OV{\lambda}}_{\OV{\mu},\OV{\nu}}= \sum_{t\in\ZZ_{\geq 0}}\sum_{\OV{\alpha}\in\mathcal{P}_{r,a}} \sum_{\OV{\beta}\in\mathcal{P}_{r,b}}
		\sum_{\OV{\gamma}\in\mathcal{P}_{r,c}}
		\sum_{\OV{\delta},\OV{\delta'},\OV{\delta''}\in\mathcal{P}_{r,t}}
		\LR^{\OV{\lambda}}_{\OV{\beta},\OV{\delta},\OV{\gamma}} \LR^{\OV{\mu}}_{\OV{\gamma},\OV{\delta'},\OV{\alpha}}
		\LR^{\OV{\nu}}_{\OV{\alpha},\OV{\delta''},\OV{\beta}} K^{\OV{\delta''}}_{\OV{\delta},\OV{\delta'}}.
	\end{equation}
	Let $V$ be a $\cPar^{\sharp}$-module and $\OV{\lambda}\in\mathcal{P}_{r,n}$. Since $\cPar^\sharp$ is negatively graded, $1_n\cPar^{\sharp}1_n\cong \CC[G(r,n)]$ and $1_n$ is an idempotent therefore we have
	\begin{equation}\label{al:simplemult}
		\big[V:S^\sharp(\OV{\lambda})\big]=\big[1_nV:S(\OV{\lambda})\big].
	\end{equation}
	\begin{theorem}\label{thm:multi}
		For $\OV{\lambda}\in\mathcal{P}_{r,l}$, $\OV{\mu}\in\mathcal{P}_{r,m},\OV{\nu}\in\mathcal{P}_{r,n}$ and $R^{\OV{\lambda}}_{\OV{\mu},\OV{\nu}}$ as in \eqref{al:reduced}, we have
		\begin{align*}
			\big[S^{\sharp}(\OV{\mu})\cstar S^{\sharp}(\OV{\nu}):S^{\sharp}(\OV{\lambda})\big]  = R^{\OV{\lambda}}_{\OV{\mu},\OV{\nu}}.
		\end{align*}
	\end{theorem}
	\begin{proof}
		From \eqref{al:simplemult}, we can equivalently compute the multiplicity of $S(\OV{\lambda})$ in $1_{l}\big(S^\sharp(\OV{\mu})\otimes S^\sharp(\OV{\nu})\big)$ as $\CC[G(r,l)]$-modules. 
		From the definition of $\cstar$, we have
		\begin{align*}
			&1_l\big(S^{\sharp}(\OV{\mu})\cstar S^{\sharp}(\OV{\nu})\big)\\
			&=
			1_{l}\cPar^{\sharp}1_{\star}\otimes_{\cPar^{\sharp}\otimes \cPar^{\sharp}}\big(S^{\sharp}(\OV{\mu})\otimes S^{\sharp}(\OV{\nu})\big)\\
			&=1_{l}\cPar^{\sharp}1_{\star}\otimes_{\cPar^{\sharp}\otimes \cPar^{\sharp}}\big(\infl^{\sharp}\CC[G(r,m)]\otimes \infl^{\sharp}\CC [G(r,n)]\big)\otimes_{\CC[G(r,m)]\otimes \CC [G(r,n)]}\big(S(\OV{\mu})\otimes S(\OV{\nu})\big).
		\end{align*}
		Let $M_{l,m,n}$ be the following $\big(\CC [G(r,l)],\CC[G(r,m)]\otimes \CC[G(r,n)]\big)$-bimodule
		\begin{displaymath}
			1_{l}\cPar^{\sharp}1_{\star}\otimes_{\cPar^{\sharp}\otimes \cPar^{\sharp}}\big(\infl^{\sharp}\CC[G(r,m)]\otimes \infl^{\sharp}\CC [G(r,n)]\big).
		\end{displaymath}
		
		For $t\in\mathcal{D}_{l,m,n}$ (see Definition \ref{def:add}), let $M^{t}_{l,m,n}$ be the subbimodule of $M_{l,m,n}$ generated by $f^{t}_{l,m,n}\otimes 1 \otimes 1$. Then, from Lemma \ref{lm:downproj}, we get
		\begin{displaymath}
			M_{l,m,n}=\bigoplus_{t\in \mathcal{D}_{l,m,n}} M^{t}_{l,m,n}.
		\end{displaymath}
		Therefore $\big[S^{\sharp}(\OV{\mu})\cstar S^{\sharp}(\OV{\nu}):S^{\sharp}(\OV{\lambda})\big]$ equals
		\begin{align}\label{al:permmulti}
			\left[1_l(S^{\sharp}(\OV{\mu})\otimes S^{\sharp}(\OV{\nu})): S(\OV{\lambda})\right]=
			\sum_{t\in \mathcal{D}_{l,m,n}} \left[M^{t}_{l,m,n}\otimes_{\CC[G(r,m)]\otimes \CC[G(r,n)]}(S(\OV{\mu})\otimes S(\OV{\nu})): S(\OV{\lambda})\right].
		\end{align}
		We use the usual hom-tensor adjunction and switch from bimodules to left modules. We denote by $\widetilde{M}_{l,m,n}^{t}$
		the $\CC[G(r,l)]\times\CC[G(r,m)^{\op}]\times \CC[G(r,n)^{\op}]$-module corresponding to the
		$\big(\CC[G(r,l)],\CC[G(r,m)]\times \CC[G(r,n)]\big)$-bimodule $M_{l,m,n}^t$. Then, for $t\in\mathcal{D}_{l,m,n}$, the corresponding summand on the right hand side of  \eqref{al:permmulti} is equal to
		\begin{align}\label{al:leftmod}
			\big[\widetilde{M}^{t}_{l,m,n}:S(\OV{\lambda})\otimes S({\OV{\mu}})^{*}\otimes S({\OV{\nu}})^{*}\big]
		\end{align}
		We claim that $\widetilde{M}_{l,m,n}^{t}$ is isomorphic to the permutation module $\CC X^{t}_{l,m,n}$. Recall that $\CC X^{t}_{l,m,n}$ is isomorphic to the induced representation $\Ind^{\CC[G]}_{\CC[L]}(\CC)$. Now, since $\widetilde{M}^{t}_{l,m,n}$ is generated by $f^{t}_{l,m,n}\otimes 1\otimes 1$ and the latter element is fixed by a conjugate of the subgroup $L$ by an element of $G(r,l)\times G(r,m)^{\op}\times G(r,n)^{\op}$, we have a surjective homomorphism from $\Ind^{\CC[G]}_{\CC[L]}(\CC)$ to $\widetilde{M}^{t}_{l,m,n}$.
		By comparing the dimensions, we conclude that $\widetilde{M}^{t}_{l,m,n}$ is isomorphic to $\Ind^{\CC[G]}_{\CC[L]}(\CC)$. Thus the assertion of the theorem follows by combining Lemma~\ref{lm:multi}, \eqref{al:permmulti} and \eqref{al:leftmod}.
	\end{proof}

	\subsection{Alternative construction for the path algebra of the colored downward partition category}\label{sec:downalt}
	The paper \cite{VS} defines a groupoid whose path algebra is isomorphic to the group algebra of $G(r,n)$. Several classical results about the representation theory of $G(r,n)$ can be deduced in an elegant way from the viewpoint of this groupoid. In this section, we define a similar category whose path algebra is isomorphic to the path algebra of the  colored downward partition category.
	
	Let $\Gr$ be a $\CC$-linear category whose objects are all maps $f_k:\{1,\ldots, k\} \to C_r$, where $k\in\ZZ_{\geq 0}$. Alternatively, we can think of the object $f_k$ as
	follows: this object is just $k$ dots, each colored by an element of $C_r$, with repetitions allowed.
	If $k<l$, then the only morphism from $f_k$ to $f_l$ is the zero morphism. If $k\geq l$, then the morphism space $\Hom_{\Gr}(f_k,f_l)$ has a basis consisting of color preserving downward $(l,k)$-partition diagrams. In other words, it is a downward partition diagram such that all vertices belonging to the same part have the same color. In what follows, we always denote a color preserving downward partition diagram by the capital letter $D$ (with some indices when necessary) to distinguish it  from  other (colored) partition diagrams. 
	The composition of the morphism $D_1$ with the morphism $D_2$ is the same as the composition of downward partitions. Note that this is only defined if the colors on the bottom vertices of $D_1$  match with the colors of the top vertices of $D_2$.
	\begin{example}\label{ex:morphgroup}
		Let $r=2$. Let us assume that the elements $\zeta^{0}$ and $\zeta$ of $C_r$ are depicted by the blue color and the red color, respectively (cf. \cite[Fig 1]{VS}). Then 
		\begin{enumerate}[(i)]
			\item the object $f_2(1)=\zeta^{0}, f_2(2)=\zeta^0$ can be depicted as   \hspace{2mm}  $\bluecircle\quad\quad\bluecircle$ 
			\item the object $f_3(1)=\zeta^0, f_3(2)=\zeta, f_{3}(3)=\zeta^0$ can be depicted as\hspace{2mm} $\bluecircle\quad\quad\redcircle\quad\quad\bluecircle$
			
			\item the object $g_3(1)=\zeta^0, g_3(2)=\zeta, g_{3}(3)=\zeta$ can be depicted as\hspace{2mm} $\bluecircle\quad\quad\redcircle\quad\quad\redcircle$
		\end{enumerate}
		Then there are exactly two  diagrams that represent morphisms from $f_3$ to $f_2$, namely:
		
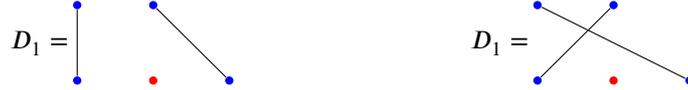
\begin{align*} 
			\begin{tikzpicture}[scale=1,mycirc/.style={circle, minimum size=0.3mm, inner sep = 1.1pt}]
				\node at (-0.5,0.5) {$D_1=$};
				\node[mycirc, fill=blue] (n1) at (0,1) { };
				\node[mycirc, fill=blue] (n2) at (1,1) { };
				\node[mycirc, fill=blue] (n1') at (0,0) {};
				\node[mycirc, fill=red] (n2') at (1,0) {};
				\node[mycirc, fill=blue] (n3') at (2,0) { };
				\path[-, draw](n1) to (n1');
				\path[-,draw](n2) to (n3');
			\end{tikzpicture} \hspace{3cm}	
			\begin{tikzpicture}[scale=1,mycirc/.style={circle, minimum size=0.3mm, inner sep = 1.1pt}]
				\node at (-0.5,0.5) {$D_1=$};
				\node[mycirc, fill=blue] (n1) at (0,1) { };
				\node[mycirc, fill=blue] (n2) at (1,1) { };
				\node[mycirc, fill=blue] (n1') at (0,0) {};
				\node[mycirc, fill=red] (n2') at (1,0) {};
				\node[mycirc, fill=blue] (n3') at (2,0) { };
				\path[-, draw](n1) to (n3');
				\path[-,draw](n2) to (n1');
			\end{tikzpicture}
		\end{align*}
		At the same time, there are no nonzero morphism from $g_3$ to $f_2$ and neither from $f_3$ to $g_3$ (or from $g_3$ to $f_3$).
	\end{example}
	The category $\Gr$ is monoidal. Given objects $f_k$ and $f_l$, the tensor product $f_k\diamond f_l$ of $f_k$ and $f_l$ is the map on $\{1,\ldots,k,k+1,\ldots, k+l\}$ given by
	\begin{equation}
		(f_k\diamond f_l)(i) =   \begin{cases}
			f_k(i), & \text{ if } i\in\{1,\ldots,k\};\\
			f_{l}(i-k), & \text{ if } i\in\{k+1,\ldots,k+l\}.
		\end{cases}
	\end{equation}
	The tensor product of morphisms is best understood diagrammatically. Indeed, for two morphisms $D_1$ and $D_2$ written as diagrams,  the tensor product $D_1\diamond D_2$ is obtained by drawing $D_1$ to the left of $D_2$. 
	\begin{example}\label{ex:tensor}
		For $D_1$ and $D_2$ as in Example~\ref{ex:morphgroup}. The tensor product $D_1\diamond D_2$ is 
		
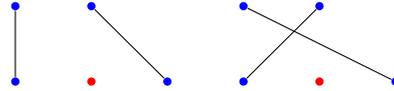
\begin{align*} 
			\begin{tikzpicture}[scale=1,mycirc/.style={circle, minimum size=0.3mm, inner sep = 1.1pt}]
				\node[mycirc, fill=blue] (n1) at (0,1) {};
				\node[mycirc, fill=blue] (n2) at (1,1) {};
				\node[mycirc, fill=blue] (n1') at (0,0) {};
				\node[mycirc, fill=red] (n2') at (1,0) {};
				\node[mycirc, fill=blue] (n3') at (2,0) {};
				\path[-, draw](n1) to (n1');
				\path[-,draw](n2) to (n3');
				\node[mycirc, fill=blue] (n4) at (3,1) {};
				\node[mycirc, fill=blue] (n5) at (4,1) {};
				\node[mycirc, fill=blue] (n4') at (3,0) {};
				\node[mycirc, fill=red] (n5') at (4,0) {};
				\node[mycirc, fill=blue] (n6') at (5,0) {};
				\path[-, draw](n4) to (n6');
				\path[-,draw](n5) to (n4');
			\end{tikzpicture}
		\end{align*}
	\end{example}
	
	For a fixed $k\in\ZZ_{\geq 0}$, let $\mathscr{G}(r,k)$ be the full subcategory of $\Gr$ whose objects are the maps from $\{1,2,\ldots, k\}$ to $C_r$. 
	Then $\mathscr{G}(r,k)$ is exactly the linearization of the groupoid constructed in \cite{VS}.
	It was proved in \cite{VS} that path algebra of $\mathscr{G}(r,k)$ is isomorphic to the group algebra of $G(r,k)$. In Theorem \ref{thm:groupoid}, we extend their result to the colored downward partition category. We have the path algebra $\CC \Gr=\displaystyle{\bigoplus_{k,k'\in\ZZ_{\geq 0}}}\Hom_{\Gr}(f_{k},f_{k'})$. Denote by $1_{f_k}$ the identity morphism on $f_k$.
	
	\begin{theorem}\label{thm:groupoid}
		For a colored partition diagram $d$ in $\cPar^{\sharp}$, let $(B_1,\zeta^{i_1}),\ldots, (B_l,\zeta^{i_l})$ denote the colored parts of $d$. Then the following map 
		\begin{align*}
			\Psi:\cPar^{\sharp}&\longrightarrow \CC \Gr, \quad \text{ given on a basis element }\\
			d & \mapsto \sum_{\substack{(j_1,\ldots,j_l)\\0\leq j_1,\ldots, j_l\leq r-1}}\zeta^{i_1 j_1+\cdots+i_lj_l} D_{(j_1,\ldots,j_l)},
		\end{align*}
		is an isomorphism of algebras. Here $D_{(j_1,\ldots, j_l)}$ is a morphism in $\Gr$ whose connected components are $B_1,\ldots, B_l$
		such that vertices in $B_s$ are colored by $\zeta^{j_s}$, for all $0\leq s\leq r-1$. 
	\end{theorem}
	\begin{proof}
		Let $d$ and $d'$ be two colored $(k,k')$-partition and $(k',k'')$-partition diagrams in $\cPar^{\sharp}$, respectively. Suppose that the parts of $d$ are $(B_1,\zeta^{i_1}),\ldots, (B_l,\zeta^{i_l})$ and the parts of $d'$ are $(C_1,\zeta^{p_1}),  \ldots, (C_m,\zeta^{p_m})$.
		
		We first show that $\Psi$ is a algebra homomorphism. For this, we analyse the parts of $dd'$ and their colors.
		
		Since $d'$ is a colored downward partition diagram, no two parts of $d$ can be combined in the multiplication $dd'$. However, parts of $d'$ can get combined in the multiplication $dd'$. For $t\in\{1,2,\dots,l\}$,
		denote by $I_t$ the set of 
		all $q\in\{1,2,\dots,m\}$
		such that $C_q\cap B_t\neq\varnothing$.
		Then the diagram $dd'$ has
		the following parts.
		
		First we have $l$ 
		parts of the form
		\begin{displaymath}
			\big((B_t \cup \bigcup_{a\in I_t}
			C_a)\setminus\{1',2',\dots,k'\}, \zeta^{i_tb_t}\big),
		\end{displaymath}
		where $t\in\{1,2,\dots,l\}$
		and $b_t=\displaystyle{\sum_{t\in I_{{t}}}}p_t$. Then, for each
		$s\in\{1,\ldots, m\}\setminus\displaystyle{\bigcup_{t=1}^{l}}I_{{t}}$, we have the part $(C_s,p_s)$.

		By the definition of $\Psi$, we have
		\begin{align*}
			\Psi(d)&= \sum_{\substack{(j_1,\ldots,j_l)\\0\leq j_1,\ldots, j_l\leq r-1}}\zeta^{i_1 j_1+\cdots+i_lj_l} D_{(j_1,\ldots,j_l)},\\
			\Psi(d')&= \sum_{\substack{(q_1,\ldots,q_m)\\0\leq q_1,\ldots, q_m\leq r-1}}\zeta^{p_1 q_1+\cdots+p_m q_m} D_{(q_1,\ldots,q_m)}.
		\end{align*}
		Now we analyse $\Psi(d)\circ\Psi(d')$. In this composition, a term corresponding to $D_{(j_1,\ldots,j_l)}\circ D_{(q_1,\ldots,q_m)}$ is nonzero if and only if the colors on the vertices of the bottom row of $D_{(j_1,\ldots,j_l)}$ match with the colors on the vertices of the top row of $D_{(q_1,\ldots,q_m)}$.  In particular, for
		$t\in\{1,2,\dots,l\}$, and for any  $a\in I_t$, the color of the part in $D_{(j_1,\ldots,j_l)}$ containing $a$ matches the color of the part in $D_{(q_1,\ldots,q_m)}$ containing $a$.
		
		The composition $D_{(j_1,\ldots,j_l)}\circ D_{(q_1,\ldots,q_m)}$ in $\Psi(d)\circ \Psi(d')$ 
		contributes the coefficient
		\begin{displaymath}
			\zeta^{i_1b_1j_1+\cdots+i_lb_lj_l+\alpha},
		\end{displaymath}
		with $\alpha=\sum p_s q_s$, where $s$ varies over the elements in $\{1,\ldots, m\}\setminus\displaystyle{\bigcup_{t=1}^{l}}I_{{t}}$. It follows from the analysis of parts of $dd'$ as above and the definition of $\Psi$ that this is precisely the coefficient at $D_{(j_1,\ldots,j_l)}\circ D_{(q_1,\ldots,q_m)}$ in $\Psi(dd')$. Thus $\Psi(dd')=\Psi(d)\circ \Psi(d')$.
		
		We know that $\cPar^{\sharp}=\displaystyle{\bigoplus_{k\leq k'}}1_k\cPar^{\sharp} 1_{k'}$. For $k\leq k'$, let
		\begin{displaymath}
			\Gr(k',k)=\displaystyle{\bigoplus_{f_{k},f_{k'}\in\Ob{\Gr}}}\Hom_{\Gr}(f_{k'},f_{k}).
		\end{displaymath}
		Then $\CC \Gr=\displaystyle{\bigoplus_{k,k'\in\ZZ_{\geq 0}}}\Gr(k',k)$. The dimensions of $\Gr(k',k)$ and $1_k\cPar^\sharp 1_{k'}$ are the same. Let $\Psi_{k,k'}$ be the restriction of $\Psi$ to $1_{k}\cPar^{\sharp}1_{k'}$. In order to show that $\Psi$ is an isomorphism, it suffices to show that
		\begin{displaymath}
			\Psi_{k,k'}:1_k\cPar^{\sharp}1_{k'}\rightarrow \Gr(k',k)
		\end{displaymath}
		is an isomorphism. 
		
		Let $k\leq k'$. For $D\in\Hom_{\Gr}(f_{k'},f_{k})$, from the definition of $\Psi$
		it is clear that there exists $d\in 1_{k}\cPar^{\sharp}1_{k'}$ such that $D$ appears in $\Psi(d)$ as a summand (possibly with some coefficient). 
		
		From the proof of \cite[Theorem~4]{VS}, we know that identity morphisms are in the image of $\Psi$. So there exist $x\in \CC [G(r,k)]$ and $y\in\CC [G(r,k')]$ such that $\Psi(x)=1_{f_{k}}$ and $\Psi(y)=1_{f_{k'}}$. Then 
		\begin{align*}
			\Psi_{k,k'}(x d y)&=\Psi(x d y)\\
			&=\Psi(x)\Psi(d)\Psi(y) \quad \text{(since $\Psi$ is an algebra homomorphism)}\\
			&=1_{f_{k}}\Psi(d)1_{f_{k'}}=\zeta^a D.
		\end{align*}
		Here $a\in\{0,1,\dots,r-1\}$.
		The last equality follows from the composition of morphisms in $\Gr$ and the fact that no element of $\Hom_{\Gr}(f_{k'},f_{k})$ other than $D$ can appear in $\Psi(d)$. 
		
		This proves surjectivity and now the bijectivity follows from the fact that the hom spaces are finite-dimensional (and of the same dimension). 
	\end{proof}
	\begin{example}
		Let $r=2$. Consider the following colored downward partition diagram:
		
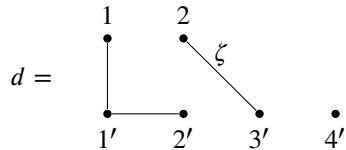
\begin{figure}[h!]
			\centering
			\label{fig:Imagephi}
			\begin{tikzpicture}[scale=1,mycirc/.style={circle,fill=black, minimum size=0.1mm, inner sep = 1.1pt}]
				\node (1) at (-1,0.5) {$d=$};
				\node[mycirc,label=above:{$1$}] (n1) at (0,1) {};
				\node[mycirc,label=above:{$2$}] (n2) at (1,1) {};
				\node[mycirc,label=below:{$1'$}] (n1') at (0,0) {};
				\node[mycirc,label=below:{$2'$}] (n2') at (1,0) {};
				\node[mycirc,label=below:{$3'$}] (n3') at (2,0) {};
				\node[mycirc,label=below:{$4'$}] (n4') at (3,0) {};
				\path[-,draw] (n1)  to (n1');
				\path[-,draw](n1') to (n2');
				\path[-,draw] (n2) edge node[midway,left, above=0.01cm]{$\zeta$} (n3');
			\end{tikzpicture}
		\end{figure}
		
		The colored parts of $d$ are $\big(\{1,1',2'\},\zeta^{0}\big), \big(\{2,3'\},\zeta\big), \big(\{4'\},\zeta^{0}\big)$. Assume the same convention as in Example \ref{ex:morphgroup}, i.e., the blue color correspond to $\zeta^{0}$ and the red color correspond to $\zeta$. Then the image of $d$ under $\Psi$ is the following linear combination:
		\begin{figure}[h!]
			\centering
			\label{fig:Image}
			
			\begin{tikzpicture}[scale=1,mycirc/.style={circle, minimum size=0.1mm, inner sep = 1.1pt}]
				\node (1) at (-0.5,0.5) {$\zeta^{0}$};
				\node[mycirc, fill=blue] (n1) at (0,1) {};
				\node[mycirc,fill=blue] (n2) at (1,1) {};
				\node[mycirc, fill=blue] (n1') at (0,0) {};
				\node[mycirc,fill=blue] (n2') at (1,0) {};
				\node[mycirc,fill=blue] (n3') at (2,0) {};
				\node[mycirc,fill=blue] (n4') at (3,0) {};
				\path[-,draw] (n1)  to (n1');
				\path[-,draw](n1') to (n2');
				\path[-,draw] (n2) to (n3');
			\end{tikzpicture}
			\begin{tikzpicture}[scale=1,mycirc/.style={circle, minimum size=0.1mm, inner sep = 1.1pt}]
				\node (1) at (-0.5,0.5) {$+\hspace{0.2cm}\zeta^0$};
				\node[mycirc, fill=red] (n1) at (0,1) {};
				\node[mycirc,fill=blue] (n2) at (1,1) {};

				\node[mycirc, fill=red] (n1') at (0,0) {};
				\node[mycirc,fill=red] (n2') at (1,0) {};
				\node[mycirc,fill=blue] (n3') at (2,0) {};
				\node[mycirc,fill=blue] (n4') at (3,0) {};

				\path[-,draw] (n1)  to (n1');
				\path[-,draw](n1') to (n2');
				\path[-,draw] (n2) to (n3');
			\end{tikzpicture}
			\begin{tikzpicture}[scale=1,mycirc/.style={circle, minimum size=0.1mm, inner sep = 1.1pt}]
				\node (1) at (-0.5,0.5) {$+ \hspace{0.2cm}\zeta$};
				\node[mycirc, fill=red] (n1) at (0,1) {};
				\node[mycirc,fill=red] (n2) at (1,1) {};

				\node[mycirc, fill=red] (n1') at (0,0) {};
				\node[mycirc,fill=red] (n2') at (1,0) {};
				\node[mycirc,fill=red] (n3') at (2,0) {};
				\node[mycirc,fill=blue] (n4') at (3,0) {};

				\path[-,draw] (n1)  to (n1');
				\path[-,draw](n1') to (n2');
				\path[-,draw] (n2) to (n3');
			\end{tikzpicture}
			\begin{tikzpicture}[scale=1,mycirc/.style={circle, minimum size=0.1mm, inner sep = 1.1pt}]
				\node (1) at (-0.5,0.5) {$+ \hspace{0.2cm}\zeta$};
				\node[mycirc, fill=red] (n1) at (0,1) {};
				\node[mycirc,fill=red] (n2) at (1,1) {};

				\node[mycirc, fill=red] (n1') at (0,0) {};
				\node[mycirc,fill=red] (n2') at (1,0) {};
				\node[mycirc,fill=red] (n3') at (2,0) {};
				\node[mycirc,fill=red] (n4') at (3,0) {};

				\path[-,draw] (n1)  to (n1');
				\path[-,draw](n1') to (n2');
				\path[-,draw] (n2) to (n3');
			\end{tikzpicture}
			\vspace{0.3cm}
			
			\begin{tikzpicture}[scale=1,mycirc/.style={circle, minimum size=0.1mm, inner sep = 1.1pt}]
				\node (1) at (-0.5,0.5) {$+ \hspace{0.2cm}\zeta$};
				\node[mycirc, fill=blue] (n1) at (0,1) {};
				\node[mycirc,fill=red] (n2) at (1,1) {};

				\node[mycirc, fill=blue] (n1') at (0,0) {};
				\node[mycirc,fill=blue] (n2') at (1,0) {};
				\node[mycirc,fill=red] (n3') at (2,0) {};
				\node[mycirc,fill=blue] (n4') at (3,0) {};

				\path[-,draw] (n1)  to (n1');
				\path[-,draw](n1') to (n2');
				\path[-,draw] (n2) to (n3');
			\end{tikzpicture}
			\begin{tikzpicture}[scale=1,mycirc/.style={circle, minimum size=0.1mm, inner sep = 1.1pt}]
				\node (1) at (-0.5,0.5) {$+ \hspace{0.2cm}\zeta$};
				\node[mycirc, fill=blue] (n1) at (0,1) {};
				\node[mycirc,fill=red] (n2) at (1,1) {};

				\node[mycirc, fill=blue] (n1') at (0,0) {};
				\node[mycirc,fill=blue] (n2') at (1,0) {};
				\node[mycirc,fill=red] (n3') at (2,0) {};
				\node[mycirc,fill=red] (n4') at (3,0) {};

				\path[-,draw] (n1)  to (n1');
				\path[-,draw](n1') to (n2');
				\path[-,draw] (n2) to (n3');
			\end{tikzpicture}
			\begin{tikzpicture}[scale=1,mycirc/.style={circle, minimum size=0.1mm, inner sep = 1.1pt}]
				\node (1) at (-0.5,0.5) {$+ \hspace{0.2cm}\zeta^{0}$};
				\node[mycirc, fill=blue] (n1) at (0,1) {};
				\node[mycirc,fill=blue] (n2) at (1,1) {};

				\node[mycirc, fill=blue] (n1') at (0,0) {};
				\node[mycirc,fill=blue] (n2') at (1,0) {};
				\node[mycirc,fill=blue] (n3') at (2,0) {};
				\node[mycirc,fill=red] (n4') at (3,0) {};

				\path[-,draw] (n1)  to (n1');
				\path[-,draw](n1') to (n2');
				\path[-,draw] (n2) to (n3');
			\end{tikzpicture}
			\begin{tikzpicture}[scale=1,mycirc/.style={circle, minimum size=0.1mm, inner sep = 1.1pt}]
				\node (1) at (-0.5,0.5) {$+ \hspace{0.2cm}\zeta^{0}$};
				\node[mycirc, fill=red] (n1) at (0,1) {};
				\node[mycirc,fill=blue] (n2) at (1,1) {};

				\node[mycirc, fill=red] (n1') at (0,0) {};
				\node[mycirc,fill=red] (n2') at (1,0) {};
				\node[mycirc,fill=blue] (n3') at (2,0) {};
				\node[mycirc,fill=red] (n4') at (3,0) {};

				\path[-,draw] (n1)  to (n1');
				\path[-,draw](n1') to (n2');
				\path[-,draw] (n2) to (n3');
			\end{tikzpicture}
			
			\caption{Image of $d$ under $\Psi$}
		\end{figure}
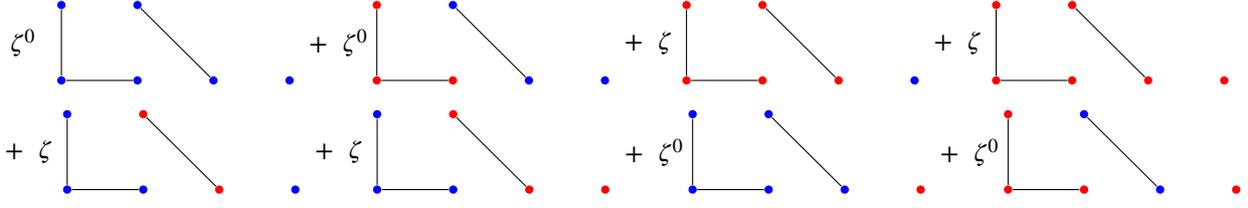
	\end{example}
	
	We have the following commutative diagram 
	\begin{displaymath}
		\xymatrix{	
			\cPar^{\sharp} \ar[r]^{\Psi}
			& \CC \Gr\\
			\displaystyle{\bigoplus_{k\in\ZZ_{\geq 0}}}\CC [G(r,k)] \ar[r]^{\Phi}\ar[u]
			& \displaystyle{\bigoplus_{k\in\ZZ_{\geq 0}}} \CC \mathscr{G}(r,k)\ar[u]
		}
	\end{displaymath}
	where the vertical maps are the inclusion maps and $\Phi$ is the map $\Psi$ restricted to $\displaystyle{\bigoplus_{k\in\ZZ_{\geq 0}}}\CC[G(r,k)]$. In particular, $\Phi$ is also an algebra isomorphism. Also, the map $\Phi$ restricted to $\CC[G(r,k)]$ coincides with the map in \cite[Section 2.1]{VS}.
	
	Let $\OV{\lambda}=(\lambda_1,\ldots,\lambda_r)\in\mathcal{P}_r$. 
	For $1\leq i\leq r$, denote by $\delta_{\lambda_i}:\{1,\ldots,|\lambda_i|\}\to C_r$  the constant function taking the value $\zeta^{i-1}$. Let $e_{\lambda_i}\in \End_{\Gr}(\delta_{\lambda_i})\subset\CC \mathscr{G}(r,|\lambda_i|)$ be an idempotent  such that 
	$$\CC[S_{\lvert \lambda_i\rvert}]e_{\lambda_i}\cong S(\lambda_i).$$ 
	For example, one can take $e_{\lambda_{i}}$ to be the Young symmetrizer in $\CC[S_{\lvert \lambda_i\rvert}]$ corresponding to the partition $\lambda_i$. Then the monoidal product $e_{\lambda_1}\diamond \cdots \diamond e_{\lambda_r}\in\CC\Gr$ is an idempotent endomorphism on the monoidal product $\delta_{\lambda_1}\diamond\cdots\diamond\delta_{\lambda_r}$.

	One of the advantages of the construction of the category $\mathscr{G}(r,k)$ is that the simple $\CC\Gr$-modules are determined by the product of Specht modules for the symmetric groups, see \cite[p. 16]{VS} for details. Then it follows from the construction of simple $\CC\Gr$-modules given in \cite[p. 16]{VS} that $\CC\Gr e_{\lambda_1}\diamond \cdots \diamond e_{\lambda_r}$ is isomorphic to the simple $\CC \Gr$-module corresponding to $\OV{\lambda}$.
	
	Now, since $\Phi$ is an algebra isomorphism, the element $\varepsilon_{\OV{\lambda}}:= \Phi^{-1}(e_{\lambda_1}\diamond \cdots \diamond e_{\lambda_r})$ is an idempotent in $\CC [G(r,\lvert \OV{\lambda}\rvert)]$ such that 
	\begin{equation}\label{eq:lambdagen}
		\CC [G(r,\lvert \OV{\lambda}\rvert)]\varepsilon_{\OV{\lambda}}\cong S(\OV{\lambda}).
	\end{equation}
	Also, define $\varepsilon_{\lambda_i}$
	as $\Phi^{-1}(e_{\lambda_i})$,
	for $i=1,2,\dots,r$, so that $\varepsilon_{\OV{\lambda}}:=\varepsilon_{\lambda_1}\star \cdots \star\varepsilon_{\lambda_r}\in\cPar^{\sharp}.$

	Let $\delta_{k}:\{1,\ldots,k\}\to C_r$  and $\delta_{l}:\{1,\ldots,l\}\to C_r$ be two constant functions both taking the same constant value in $C_r$. Recall that $1_{\delta_k}$ and $1_{\delta_l}$ denote the identity morphisms on their respective objects. Then 
	\begin{equation}\label{eq:homgr}
		1_{\delta_{l}}\CC\Gr1_{\delta_k}\cong\Hom_{\Gr}(\delta_{k},\delta_l)\cong 1_{l}\parti^{\sharp}1_{k}.   
	\end{equation} 
	
	\begin{lemma}\label{lm:tensor}
		For $\OV{\lambda}=(\lambda_1,\ldots,\lambda_r)$ and $\OV{\mu}=(\mu_1,\ldots,\mu_r)\in \mathcal{P}_r$, we have an isomorphism
		of vector spaces as follows:
		\begin{displaymath}
			\varepsilon_{\OV{\lambda}}\cPar^{\sharp} \varepsilon_{\OV{\mu}}\cong \varepsilon_{\lambda_1}\parti^{\sharp} \varepsilon_{\mu_{1}}\otimes \cdots 
			\otimes
			\varepsilon _{\lambda_r}\parti^{\sharp} \varepsilon_{\mu_{r}}.
		\end{displaymath}
	\end{lemma}
	
	\begin{proof}
		We have
		\begin{align*}
			\Phi^{-1}(e_{\lambda_1}\diamond\cdots \diamond e_{\lambda_r})\,\cPar^{\sharp}\,\Phi^{-1}(e_{\mu_1}\diamond\cdots\diamond e_{\mu_{r}})&=
			\Psi^{-1}(e_{\lambda_1}\diamond\cdots \diamond e_{\lambda_r})\,\Psi^{-1}(\CC \Gr)\, \Psi^{-1}(e_{\mu_1}\diamond\cdots \diamond e_{\mu_{r}})\\
			&= \Psi^{-1}(e_{\lambda_1}\diamond \cdots \diamond e_{\lambda_r}\, \CC \Gr \, e_{\mu_1}\diamond \cdots \diamond e_{\mu_{r}} )\\
			&\cong \Psi^{-1}(e_{\lambda_1}\CC \Gr e_{\mu_1}\diamond\cdots\diamond e_{\lambda_r}\CC\Gr e_{\mu_r})\\
			&\cong \varepsilon_{\lambda_1}\parti^{\sharp} \varepsilon_{\mu_{1}}\otimes \cdots 
			\otimes
			\varepsilon _{\lambda_r}\parti^{\sharp} \varepsilon_{\mu_{r}},
		\end{align*}
		where the last isomorphism is a consequence of Equation~\eqref{eq:homgr}. This gives the vector space isomorphism $\varepsilon_{\OV{\lambda}}\cPar^{\sharp} \varepsilon_{\OV{\mu}}\cong \varepsilon_{\lambda_1}\parti^{\sharp} \varepsilon_{\mu_{1}}\otimes \cdots \otimes \varepsilon_{\lambda_r}\parti^{\sharp} \varepsilon_{\mu_{r}}$.
	\end{proof}
	
	\subsection{Various Grothendieck rings}
	In this section we aim to show that both split Grothendieck rings $K_0(\cPar^{\sharp})$ and $K_{0}(\cPar(\textbf{x}))$ are isomorphic to $\displaystyle{\bigotimes_{i=1}^{r}}\Lambda[X^{(i)}]$, see Section~\ref{sec:wreathfun}.

	{\bf Basis from indecomposable projectives.} For $\OV{\lambda}\in\mathcal{P}_r$, let $P^{\sharp}(\OV{\lambda})=\Ind^{\cPar^{\sharp}}_{\cSym}(S(\OV{\lambda}))$. Then $P^{\sharp}(\OV{\lambda})$ is a  finite-dimensional module and it is the projective cover of the simple module $S^{\sharp}(\OV{\lambda})$. The set $\{[P^{\sharp}(\OV{\lambda})]\mid \OV{\lambda}\in\mathcal{P}_r\}$ is a $\ZZ$-basis of the free $\ZZ$-module $K_{0}(\cPar^{\sharp})$. Since the functor $\Ind_{\cSym}^{\cPar^{\sharp}}$ is exact and monoidal, we have the following isomorphism of the Grothendieck rings:
	\begin{align*}
		K_{0}(\cSym)&\rightarrow K_{0}(\cPar^{\sharp})\\
		[S(\OV{\lambda})]&\mapsto [P^{\sharp}(\OV{\lambda})]. 
	\end{align*}
	Also, $K_{0}(\cSym)\cong \displaystyle{\bigotimes_{i=1}^{r}}\Lambda[X^{(i)}]$ which sends the basis element $[S(\OV{\lambda})]$ to the wreath product Schur function $s_{\OV{\lambda}}$ (see Section~\ref{sec:wreathfun}). In particular, using Proposition~\ref{thm:basesymlittle}, we get the following.
	\begin{proposition}\label{prop:projinde}
		The structure constants for the basis $\{[P^{\sharp}(\OV{\lambda})]\mid \OV{\lambda}\in\mathcal{P}_r\}$ are given by the product of the Littlewood--Richardson coefficients
		for each factor $\Lambda[X^{(i)}]$.
	\end{proposition}
	{\bf The Cartan matrix.} For $\OV{\lambda},\OV{\mu}\in\mathcal{P}_{r}$, let $B_{\OV{\lambda},\OV{\mu}}:=[P^{\sharp}(\OV{\lambda}):S^{\sharp}(\OV{\mu})]$. The latter composition factor multiplicity is also equal to the dimension of $\Hom_{\cPar^{\sharp}}(P^{\sharp}(\OV{\mu}),P^{\sharp}(\OV{\lambda}))$. The matrix $B_r=(B_{\OV{\lambda},\OV{\mu}})_{\OV{\lambda},\OV{\mu}\in \mathcal{P}_r}$ is the Cartan matrix for $\cPar^\sharp$. The following lemma describes $B_r$. We omit the proof as it is exactly the same as the proof in the spacial case $r=1$ given in
	\cite[Lemma~3.9]{Brv}. 
	
	\begin{lemma}\label{lm:cartan}
		Let $\OV{\lambda},\OV{\mu}\in\mathcal{P}_{r}$. Then
		\begin{enumerate}[$($a$)$]
			\item  $B_{\OV{\lambda},\OV{\mu}}=\dim \varepsilon_{\OV{\mu}}\,\cPar^{\sharp}\, \varepsilon_{\OV{\lambda}}$.
			
			\item $B_{\OV{\lambda},\OV{\lambda}}=1$.
			
			\item $B_{\OV{\lambda},\OV{\mu}}=0$ if $\lvert \OV{\mu}\rvert > \lvert \OV{\lambda}\rvert$ or if $\lvert \OV{\mu}\rvert = \lvert \OV{\lambda}\rvert$ and $\OV{\mu}\neq \OV{\lambda}$.
		\end{enumerate}
		Consequently,  $B_r$ is unitriangular matrix, in particular, it is invertible.
	\end{lemma}
	
	\begin{corollary}\label{cor:cartantensor}
		The Cartan matrix $B_r$ is the $r$-fold tensor power of the Cartan matrix $B_1$, i.e., $B_r=B_1^{\otimes r}$. 
	\end{corollary}
	
	\begin{proof}
		This is a direct consequence of Lemma \ref{lm:tensor} and Lemma \ref{lm:cartan}.
	\end{proof}
	
	{\bf Basis from simple modules}: Let $K_0'(\cPar^\sharp)$ denote the Grothendieck group of $\cPar^\sharp\BMod_{\text{fd}}$.
	The set $$\{[S^\sharp(\OV{\lambda})]\mid \lambda\in\mathcal{P}_r\}$$ is a $\ZZ$-basis of $K_0'(\cPar^\sharp)$. 
	From Lemma \ref{lm:downproj}, we have that $\cPar^{\sharp}1_{\star}$ is a right $\cPar^\sharp\otimes\cPar^\sharp$-module. The tensor product $\cstar$ is biexact and hence induces a ring structure on $K_0'(\cPar^\sharp)$. 
	
	Since $B_r$ is invertible,  we have the linear 
	an isomorphism  $K_0(\cPar^{\sharp})\rightarrow K_0'(\cPar^{\sharp})$
	induced by the natural inclusion functor $$\cPar^\sharp-\text{Proj}\to \cPar^\sharp\BMod_{\text{fd}},$$ where $\cPar^{\sharp}-\text{Proj}$ denotes the category of finitely generated projective $\cPar^\sharp$-modules. Since the inclusion functor is clearly monoidal the aforementioned linear isomorphism is, in fact, a ring isomorphism.
	
	\begin{proposition}\label{prop:struct}
		The structure constants for the basis $\{[S^{\sharp}(\OV{\lambda})]\mid \OV{\lambda}\in\mathcal{P}_r\}$ are given by 
		the product of the reduced Kronecker coefficients of the form $\G^{\lambda_1}_{\mu_1,\nu_1}\cdots \G^{\lambda_r}_{\mu_r,\nu_r}$.
	\end{proposition}
	
	\begin{proof}
		We know that the Cartan matrix $B_1$ is the transformation matrix from the basis $\{s_{\lambda}\mid\lambda\in\mathcal{P}\}$ of Schur functions in $\Lambda$ to the basis $\{\tilde{s}_{\lambda}\mid \lambda\in\mathcal{P}\}$ of deformed Schur functions. Let $A_{1}$ be the inverse of $B_1$.  Then, from Corollary \ref{cor:cartantensor}, we have that  $A_{1}^{\otimes r}$ is the inverse of the Cartan matrix $B_r$. We have the following commutative diagram
		\begin{displaymath}
			\xymatrix{	
				K_0(\cPar^{\sharp})\ar[dr]	& K_0'(\cPar^\sharp)\ar@[->][d]\ar[l] \\
				& \displaystyle{\bigotimes_{i=1}^{r}}\Lambda[X^{(i)}] ,
			}
		\end{displaymath}
		where the matrix of the top map (which is a ring homomorphism) is $A_1^{\otimes r}$ and the ring homomorphism  $K_0(\cPar^{\sharp})\to\displaystyle{\bigotimes_{i=1}^{r}}\Lambda[X^{(i)}]$ is given by 
		\begin{align*}
			[P^{\sharp}(\OV{\lambda})]& \mapsto s_{\OV{\lambda}},
		\end{align*}
		where $s_{\OV{\lambda}}$ is the wreath product Schur function.
		
		Since the map $K_0'(\cPar^\sharp)\rightarrow \displaystyle{\bigotimes_{i=1}^{r}}\Lambda[X^{(i)}]$ is the composition of previous ring homomorphisms and $A_1^{\otimes r}$ is the transformation matrix from the basis given by the deformed wreath product Schur functions to the basis given by wreath product Schur function, it is 
		given by
		\begin{align}\label{al:basesim}
			[S^{\sharp}(\OV{\lambda})]& \mapsto \tilde{s}_{\OV{\lambda}},\nonumber
		\end{align}
		where $\tilde{s}_{\OV{\lambda}}$ is the deformed wreath product Schur function. Now the result follows from  Proposition~\ref{thm:basesym}(b).
	\end{proof}
	
	\begin{theorem}\label{thm:formula}
		For $\OV{\lambda}=(\lambda_1,\ldots,\lambda_r)$ and $ \OV{\mu}=(\mu_1,\ldots,\mu_r)\in \mathcal{P}_r$, we have
		$\G^{\lambda_1}_{\mu_1,\nu_1}\cdots \G^{\lambda_r}_{\mu_r,\nu_r}= R^{\OV{\lambda}}_{\OV{\mu},\OV{\nu}}$.
	\end{theorem}
	
	\begin{proof}
		This is immediate form Theorem \ref{thm:multi} and Proposition \ref{prop:struct}.
	\end{proof}
	
	\begin{example}
		Suppose $r=3$. Let $$\OV{\lambda}=(\lambda_1,\lambda_2,\lambda_3)=\big(\varnothing,\ytableausetup{boxsize=0.7em}\begin{ytableau}
			\empty \\
			\empty\\
		\end{ytableau}, \varnothing\big), 
		\OV{\mu}=(\mu_1,\mu_2,\mu_3)=\big(\varnothing, \ytableausetup{boxsize=0.7em}\begin{ytableau}
			\empty & \empty\\\end{ytableau},\varnothing\big) \text{ and } \OV{\nu}=(\nu_1,\nu_2,\nu_3)=\big(\varnothing, \ytableausetup{boxsize=0.7em}\begin{ytableau}
			\empty & \empty\\\end{ytableau},\varnothing\big).$$ Then $\G^{\lambda_1}_{\mu_1,\nu_1}\G^{\lambda_2}_{\mu_2,\nu_2}\G^{\lambda_3}_{\mu_3,\nu_3}=1$.
		
		Now we compute $R^{\OV{\lambda}}_{\OV{\mu},\OV{\nu}}$, see~\eqref{al:reduced}. From~\eqref{eq:LRtriple} and~\cite[Section~9]{Mac}, we compute the following coefficients:
		\begin{displaymath}
			\resizebox{17cm}{!}{$
				\LR^{\OV{\lambda}}_{\OV{\beta},\OV{\delta},\OV{\gamma}}=  \begin{cases}
					1, & \text{ if } \OV{\beta}=\OV{\gamma}=(\varnothing,\varnothing,\varnothing), \OV{\delta}=\OV{\lambda};\\
					1, & \text { if } \OV{\delta}=\OV{\beta}=(\varnothing,\varnothing,\varnothing), \OV{\gamma}=\OV{\lambda};\\
					1, & \text { if } \OV{\delta}=\OV{\gamma}=(\varnothing,\varnothing,\varnothing), \OV{\beta}=\OV{\lambda};\\
					1, & \text { if } \OV{\beta}=\OV{\delta}=(\varnothing,\ytableausetup{boxsize=0.7em}\begin{ytableau}
						\empty \\
					\end{ytableau}, \varnothing), \OV{\gamma}=(\varnothing,\varnothing,\varnothing) ;\\
					1, & \text { if } \OV{\gamma}=\OV{\delta}=(\varnothing,\ytableausetup{boxsize=0.7em}\begin{ytableau}
						\empty \\
					\end{ytableau}, \varnothing), \OV{\beta}=(\varnothing,\varnothing,\varnothing);\\
					1, & \text { if } \OV{\beta}=\OV{\gamma}=(\varnothing,\ytableausetup{boxsize=0.7em}\begin{ytableau}
						\empty \\
					\end{ytableau}, \varnothing), \OV{\delta}=(\varnothing,\varnothing,\varnothing) ;\\
					0, & \text{ otherwise}.
				\end{cases}\,\,\,
				\LR^{\OV{\mu}}_{\OV{\gamma},\OV{\delta'},\OV{\alpha}}=  \begin{cases}
					1, & \text{ if } \OV{\gamma}=\OV{\alpha}=(\varnothing,\varnothing,\varnothing), \OV{\delta'}=\OV{\mu};\\
					1, & \text { if } \OV{\delta'}=\OV{\gamma}=(\varnothing,\varnothing,\varnothing), \OV{\alpha}=\OV{\mu};\\
					1, & \text { if } \OV{\delta'}=\OV{\alpha}=(\varnothing,\varnothing,\varnothing), \OV{\gamma}=\OV{\mu};\\
					1, & \text { if } \OV{\gamma}=\OV{\delta'}=(\varnothing,\ytableausetup{boxsize=0.7em}\begin{ytableau}
						\empty \\
					\end{ytableau}, \varnothing), \OV{\alpha}=(\varnothing,\varnothing,\varnothing) ;\\
					1, & \text { if } \OV{\alpha}=\OV{\delta'}=(\varnothing,\ytableausetup{boxsize=0.7em}\begin{ytableau}
						\empty \\
					\end{ytableau}, \varnothing), \OV{\gamma}=(\varnothing,\varnothing,\varnothing);\\
					1, & \text { if } \OV{\gamma}=\OV{\alpha}=(\varnothing,\ytableausetup{boxsize=0.7em}\begin{ytableau}
						\empty \\
					\end{ytableau}, \varnothing), \OV{\delta'}=(\varnothing,\varnothing,\varnothing) ;\\
					0, & \text{ otherwise}.
				\end{cases}\,\,\,
				\LR^{\OV{\nu}}_{\OV{\alpha},\OV{\delta''},\OV{\beta}}=  \begin{cases}
					1, & \text{ if } \OV{\alpha}=\OV{\beta}=(\varnothing,\varnothing,\varnothing), \OV{\delta''}=\OV{\nu};\\
					1, & \text { if } \OV{\delta''}=\OV{\alpha}=(\varnothing,\varnothing,\varnothing), \OV{\beta}=\OV{\nu};\\
					1, & \text { if } \OV{\delta''}=\OV{\beta}=(\varnothing,\varnothing,\varnothing), \OV{\alpha}=\OV{\nu};\\
					1, & \text { if } \OV{\alpha}=\OV{\delta''}=(\varnothing,\ytableausetup{boxsize=0.7em}\begin{ytableau}
						\empty \\
					\end{ytableau}, \varnothing), \OV{\beta}=(\varnothing,\varnothing,\varnothing) ;\\
					1, & \text { if } \OV{\beta}=\OV{\delta''}=(\varnothing,\ytableausetup{boxsize=0.7em}\begin{ytableau}
						\empty \\
					\end{ytableau}, \varnothing), \OV{\alpha}=(\varnothing,\varnothing,\varnothing);\\
					1, & \text { if } \OV{\alpha}=\OV{\beta}=(\varnothing,\ytableausetup{boxsize=0.7em}\begin{ytableau}
						\empty \\
					\end{ytableau}, \varnothing), \OV{\delta''}=(\varnothing,\varnothing,\varnothing) ;\\
					0, & \text{ otherwise}.
				\end{cases}$}
		\end{displaymath}
		From above it follows that all three coefficients are nonzero at the same time only in the following two cases:
		\begin{itemize}
			\item $\OV{\alpha}=\OV{\beta}=\OV{\gamma}=(\varnothing,\ytableausetup{boxsize=0.7em}\begin{ytableau}
				\empty \\
			\end{ytableau},\varnothing)$, $\OV{\delta}= \OV{\delta'}=\OV{\delta''}=(\varnothing,\varnothing,\varnothing)$.
			\item $\OV{\alpha}=\OV{\beta}=\OV{\gamma}=(\varnothing,\varnothing,\varnothing)$, $\OV{\delta}=\OV{\lambda}$, $\OV{\delta'}=\OV{\mu}$ and $\OV{\delta''}=\OV{\nu}$,
			
		\end{itemize}
		In the first case the quantity $K^{\OV{\delta''}}_{\OV{\delta},\OV{\delta'}}$ is evidently one while in the second case we show it is zero. 
		
		Recall from Section~\ref{sec:wreath} that the map $\phi_2:C_3\to \CC^{*}$ denotes the one-dimensional representation of $C_3$ given by $\phi_{2}(\zeta)=\zeta$. Consider the following representations of $(C_3\times C_3)\wr S_2$. Let $((f,h),\sigma)\in (C_3\times C_3)\wr S_2$.
		\begin{enumerate}[$(i)$]
			\item $(C_3\times C_3)\wr S_2\to C_3\wr S_2\to \CC^{*}$ given by $((f,h),\sigma)\mapsto (f,\sigma)\mapsto 
			\phi_2\big(f(1)f(2)f(3)\big)$,
			\vspace{0.1cm}
			
			\item $(C_3\times C_3)\wr S_2\to C_3\wr S_2\to \CC^{*}$ given by $((f,h),\sigma)\mapsto (h,\sigma)\mapsto 
			\phi_2\big(h(1)h(2)h(3)\big)$,
			\vspace{0.1cm}
			
			\item $(C_3\times C_3)\wr S_2\to C_3\wr S_2\to \CC^{*}$ given by $((f,h),\sigma)\mapsto (fh,\sigma)\mapsto 
			\phi_2\big(f(1)h(1)f(2)h(2)f(3)h(3)\big)\sgn(\sigma)$, where $\sgn(\sigma)$ denotes the signature of the permutation $\sigma$. 
		\end{enumerate}
		By definition, $K^{\OV{\delta''}}_{\delta,\delta'}$ is the multiplicity of the representation given in $(iii)$ in the tensor product of the representations given in $(i)$ and $(ii)$. From this, it is now easy to deduce that $K^{\OV{\delta''}}_{\delta,\delta'}$ is zero.
		
		Now substituting all of these in~\eqref{al:reduced}, we get $R^{\OV{\lambda}}_{\OV{\mu},\OV{\nu}}=1$  
	\end{example}
	{\bf Basis from standard modules.} Here we consider the Grothendieck group of yet another subcategory of \linebreak $\cPar(\textbf{x})\Mod$. Let $K_0''(\cPar(\textbf{x}))$ denote the Grothendieck group of the full subcategory  $\cPar(\textbf{x})\BMod_{\Delta}$ of finitely generated $\cPar(\textbf{x})$-modules having a filtration with standard subquotients. Then $\{[\Delta(\OV{\lambda})]\mid \OV{\lambda}\in\mathcal{P}_r\}$ is a $\ZZ$-basis of $K_{0}''(\cPar(\textbf{x}))$. 
	The following proposition asserts that the tensor product of two standard modules has a filtration with standard subquotients. This endows $K_0''(\cPar(\textbf{x}))$ with the natural structure of a ring. The argument in the case $r=1$, given in \cite[Theorem 3.11]{Brv}, works for the general case.
	
	{\bf Degree filtration.}  Both $\parti^{\sharp}$ and $\cPar^\sharp$ are negatively graded. By \cite[p. 21]{Brv},  every $\parti^{\sharp}$-module has the unique degree filtration. Similarly, we can define the degree filtration for every $\cPar^{\sharp}$-module  as well. Specifically, for a $\cPar^{\sharp}$-module $V$, 
	let $V_n:=\displaystyle{\bigoplus_{m\leq n}}1_m V$. Then the following is the degree filtration of $V$:
	\begin{displaymath}
		0=V_{-1}\subset V_{0} \subset V_{1} \subset \cdots \subset V_n\subset \cdots
	\end{displaymath}
	such that the factor $V_n/V_{n-1}$ is isomorphic to $\infl^{\sharp}(1_nV)$, where $1_nV$ is a $\cSym$-module by restriction.

	The following proposition, for $r=1$, is given in~\cite[Corollary 6.2]{SamSnow}. This is also proved in~\cite[Theorem 3.11]{Brv}. The idea of the latter proof works for the general case. Indeed, the first part uses  Corollary~\ref{coro:tringularcat} and Lemma~\ref{coro:tringularcat} together with~\cite[Propositions 4.31-4.32]{SamSnow}; the second part uses the degree filtration of $S^{\sharp}(\OV{\lambda})\Cstar S^{\sharp}(\OV{\nu})$ together with Theorem~\ref{thm:multi}.
	\begin{proposition}\label{prop:delta}
		For $V, W\in \cPar(\mathrm{\mathbf{x}})\BMod_{\Delta}$ and $i\geq 1$, we have $\Tor^{\cPar(\mathrm{\mathbf{x}})}_i(V,W)=0$, in particular, $\cstar$ is biexact on this category. For $\lambda,\mu\in\mathcal{P}_{r}$, there exists a filtration 
		\begin{displaymath}
			0=V_{-1}\subset V_{0}\subset V_{1}\subset \cdots \subset V_{\lvert \OV{\mu}\rvert + \lvert \OV{\lambda}\rvert}=\Delta(\OV{\mu})\cstar \Delta(\OV{\nu})
		\end{displaymath}
		such that $V_i/V_{i-1}\cong \displaystyle{\bigoplus_{\OV{\lambda}\in \mathcal{P}_r}}\Delta(\OV{\lambda})^{\oplus R^{\OV{\lambda}}_{\OV{\mu},\OV{\nu}}}$.
	\end{proposition}
	
	As a consequence, from the above proposition, the Grothendieck group $K_0''(\cPar(\textbf{x}))$ also admits the structure of a ring.
	
	For an exact functor $F$ between two categories, let $[F]$ denote the induced linear map at the level of appropriate Grothendieck groups.

	Since the functor $\Ind_{\cPar^{\sharp}}^{\cPar(\textbf{x})}$ is exact and monoidal and also $\Ind_{\cPar^{\sharp}}^{\cPar(\textbf{x})}(S^{\sharp}(\OV{\lambda}))=\Delta(\OV{\lambda})$, we get an isomorphism of rings: $$[\Ind_{\cPar^{\sharp}}^{\cPar(\textbf{x})}]:K_0'(\cPar^{\sharp})\to K_0''(\cPar(\textbf{x})).$$

	From Proposition~\ref{prop:upperfinite} the category $\cPar(\textbf{x})\Mod$ is an upper finite highest weight category therefore the category of finitely generated projective $\cPar(\textbf{x})$-modules is a full subcategory of $\cPar(\textbf{x})\BMod_{\Delta}$. For $\OV{\lambda}\in\mathcal{P}_r$, let $P(\OV{\lambda})$ denote the finitely generated projective cover of the simple $\cPar(\textbf{x})$-module corresponding to $\OV{\lambda}$. In the  Grothendieck group $K_0''(\cPar(\textbf{x}))$, we have the following
	$$[P(\OV{\lambda})]=[\Delta(\OV{\lambda})]+\sum_{\substack{\OV{\mu}\in\mathcal{P}_r\\ |\OV{\mu}|<|\OV{\lambda}|}}[\Delta(\OV{\mu})].$$
	This gives the map $[\mathrm{Incl}]: K_0(\cPar(\textbf{x}))\to K_0''(\cPar(\textbf{x}))$ induced by the inclusion functor is an isomorphism of vector spaces. 
	
	Notice that all the Grothendieck groups  that we considered in this section are rings. These Grothendieck rings and all the induced maps on them fit into a commutative diagram of vector spaces.
	
	\begin{proposition}\label{prop:lambdaring}
		We have the following commutative diagram of vector spaces:
		%Each of the maps in the following commutative diagram is an isomorphism of rings:
		\begin{displaymath}
			\resizebox{11cm}{!}{
				\xymatrix{	
					& K_0(\cPar^{\sharp}) \ar[r]^{[\mathrm{Incl}]}\ar[dd]^{[\Ind^{\cPar(\mathrm{\mathbf{x}})}_{\cPar^\sharp}]} 	& K_0'(\cPar^\sharp)\ar[dd]^{[\Ind^{\cPar(\mathrm{\mathbf{x}})}_{\cPar^\sharp}]} \\
					\displaystyle{\bigotimes_{i=1}^{r}}\Lambda[X^{(i)}]\cong K_0(\cSym)\ar[dr]^{[\Ind^{\cPar(\mathrm{\mathbf{x}})}_{\cSym}]} \ar[ur]^{[\Ind^{\cPar^{\sharp}}_{\cSym}]}\\
					& K_0(\cPar(\mathrm{\mathbf{x}})) \ar[r]^{[\mathrm{Incl}]} & K_0''(\cPar(\mathrm{\mathbf{x}}))
			}}
		\end{displaymath}
		Moreover, the left vertical map and the map $[\Ind_{\cSym}^{\cPar(\textbf{x})}]$ are isomorphisms and the bottom map is a ring homomorphism. Consequently, all the maps in the above diagram are isomorphism of rings. 
	\end{proposition}
	\begin{proof}
		The commutativity of the diagram is easy to see. Now since the top and the bottom maps and the left vertical maps are isomorphisms, therefore by the commutative of the square in the diagram it follows that the vertical map is also an isomorphism. Likewise, due to the commutativity of the triangle in the above diagram, it follows that the map $[\Ind_{\cSym}^{\cPar(\textbf{x})}]$ is an isomorphism. Since every map in the above commutative diagram is a ring isomorphism, the bottom map is also a ring homomorphism.
	\end{proof}
	\section{Generic semisimplicity}
	Let $\CC[y_0,\ldots,y_{r-1}]$ denote the polynomial ring in $r$-variables $y_0,\ldots,y_{r-1}$. For a polynomial $f\in\CC [y_0,\ldots,y_{r-1}]$, let $V(f)$ denote the set of zeros of $f$. By the generic semismiplcity of  multiparameter colored partition algebras (or modules over them), we mean that there exists a nonzero polynomial $f\in\CC [y_0,\ldots,y_{r-1} ]$ such that the algebra $\cPar_k(\textbf{x})$ is semisimple for any $\textbf{x}=(x_0,\ldots,x_{r-1}) \in \CC^{r}\setminus V(f)$.  
	
	Let $i$ be an integer such that $0\leq i\leq k$. Let $\LL_{k,i}$ be the set consisting of all colored $(k,k)$-partition diagrams with exactly $i$ propagating parts such that the bottom vertices $1',\ldots, i'$ lie in propagating parts and the remaining bottom vertices are singleton parts. Then $\LL_{k,i}=\{d\in\cPar_{k}\mid \cPar_kd=\cPar_k e_i\}$, where recall that $\cPar_k$ denote the colored partition monoid, for the following idempotent element $e_i$ in $\cPar_k$:
	
	\begin{center}
		\begin{tikzpicture}[scale=1,mycirc/.style={circle,fill=black, minimum size=0.1mm, inner sep = 1.1pt}]
			\node (1) at (-1,0.5) {$e_i =$};
			\node[mycirc,label=above:{$1$}] (n1) at (0,1) {};
			\node[mycirc,label=above:{$2$}] (n2) at (1,1) {};
			\node  (n3) at (2,1) {$\ldots$};
			\node[mycirc,label=above:{$i$}] (n4) at (3,1) {};
			\node [mycirc, label=above:{$i+1$}] (n5) at (4,1) { };
			\node (n6) at (5,1) {$\ldots$};
			\node[mycirc, label=above:{$k$}]  (n7) at (6,1) { };
			\node[mycirc,label=below:{$1'$}] (n7'') at (0,0) {};
			\node[mycirc,label=below:{$2'$}] (n8'') at (1,0) {};
			\node (n9'') at (2,0) {$\ldots$};
			\node[mycirc,label=below:{$i'$}] (n10'') at (3,0) {};
			\node[mycirc,label=below:{$(i+1)'$}] (n11'') at (4,0) {};
			\node (n12'') at (5,0) {$\ldots$};
			\node[mycirc, label=below:{$k'$}] (n13'') at (6,0) { };
			\path[-,draw] (n1) to (n7'');
			\path[-,draw] (n2) to (n8'');
			\path[-,draw] (n4) to (n10'');
		\end{tikzpicture}
	\end{center}
	
	Corresponding to the idempotent $e_i$, the maximal subgroup $G(r,i)$ is the group of units of the submonoid $e_i\cPar_k e_i$. Note that the identity element of $G(r,i)$ is $e_i$. We have a natural right action of $G(r,i)$ on $\LL_{k,i}$
	by letting $G(r,i)$ act on the bottom vertices $1',\ldots,i'$ of the elements of $\LL_{k,i}$. Also,  $\cPar_k(\textbf{x})$ acts on the left on $\CC\LL_{k,i}$ as follows. For $d\in\LL_{k,i}$ and a colored $k$-partition diagram $d'$, define
	\begin{align}
		d'\cdot d=
		\begin{cases}
			d'd, & \text{ if } \pn(d'\circ d)=i;\\
			0, & \text{ otherwise};
		\end{cases}
	\end{align}
	where recall from Section~\ref{sec:multicat} that $\pn(d'\circ d)$ denotes the rank of the colored partition diagram $d'\circ d$.
	
	The above two actions commute making $\CC\LL_{k,i}$ a $\cPar_k(\textbf{x})-\CC[G(r,i)]$-bimodule. Recall that $S(\OV{\lambda})$, for $\OV{\lambda}\in\mathcal{P}_{r,i}$, denotes the corresponding simple $\CC[G(r,i)]$-module. Define,
	\begin{align}\label{al:weyl}
		W({\OV{\lambda}}):=\CC\LL_{k,i}\otimes_{\CC[G(r,i)]} S({\OV{\lambda}}).
	\end{align}
	
	\begin{lemma}
		For $k\in \ZZ_{\geq 0}$ and $\OV{\lambda}\in\mathcal{P}_{r,i}$, as $\cPar_k(\mathrm{\mathbf{x}})$-modules, we have 
		\begin{equation*}
			1_k\Delta\big(\OV{\lambda}\big)\cong   \begin{cases}
				W({\OV{\lambda}}), & \text{ if } i\leq k;\\
				0, & \text{ otherwise}.
			\end{cases}
		\end{equation*}
	\end{lemma}
	
	\begin{proof}
		We have
		\begin{align*}
			1_k\Delta(\OV{\lambda})&=1_k\cPar(\textbf{x})\otimes_{\cPar^{\sharp}}S(\OV{\lambda}), \quad \text{ by the definition of $\Delta(\OV{\lambda})$},\\
			&\cong (1_k\cPar^{\flat}\otimes_{\cSym} \cPar^{\sharp})\otimes_{\cPar^{\sharp}}S(\OV{\lambda}), \quad \text{ from \eqref{al:T3}},\\
			&\cong 1_k\cPar^{\flat}\otimes_{\cSym} S(\OV{\lambda}).
		\end{align*}
		If $i>k$, then $1_k\Delta(\OV{\lambda})$ is zero as there is no colored $(k,i)$-partition diagram in $\cPar^{\flat}$ and only elements of $\CC[G(r,i)]$ (as a subalgebra of $\cPar^{\sharp}$) can possibly act nonzero on $S(\OV{\lambda})$.
		
		If $0\leq i\leq k$, then there is a natural bijection between the elements of $\LL_{k,i}$ 
		and colored $(k,i)$-partition diagram in $\cPar^{\flat}$. Tensoring with $S({\OV{\lambda}})$, gives rise to an isomorphism
		as in the formulation.
	\end{proof}
	{\bf Cellularity.}  Partition algebras are cellular over any field. It was proved by many authors, for example,  \cite{Xi}, \cite{DW}, and in each proof the key observation is that the group algebra of the symmetric group over the underlying field $\FF$. In~\cite{Wilcox}, the latter observation was made formal in terms maximal subgroups of twisted semigroup algebras and a criterion for proving cellularity of a twisted semigroup algebras was also given. Now the group algebra of complex reflection groups of type $G(r,n)$ over $\FF$ is cellular~\cite[Theorem~5.5]{GL}. We can observe that $\cPar_k(\textbf{x})$ is cellular over $\FF$ by several methods. For example, one may apply techniques from \cite[Remark~3.4]{DW} to construct an explicit cellular basis where the cell modules are given by $W(\OV{\lambda})$ constructed in \eqref{al:weyl}. Another way, one may realize $\cPar_k(\textbf{x})$ as a twisted semigroup algebra  whose maximal subgroups are cellular and they apply techniques given in~\cite{Wilcox}.
	\begin{proposition}\label{thm:alg}
		Let $k\in\ZZ_{\geq 0}$. 
		Assume that $x_i\neq 0$ for some $i$. Then 
		$\cPar_k(\mathrm{\mathbf{x}})$ is semisimple if and only if $W(\OV{\lambda})$ is simple, for all $\OV{\lambda}\in\mathcal{P}_{r,i}$ and for all $i=0,1,\ldots,k$.
	\end{proposition}
	\begin{proof}
		The condition on parameters ensure that the canonical form on each of $W(\OV{\lambda})$ is nonzero.	Also, it follows from the previous discussion that  $\cPar_k(\textbf{x})$ is cellular. Hence the algebra $\cPar_k(\textbf{x})$ being semisimple is equivalent to each $W(\OV{\lambda})$ being simple. 
	\end{proof}
	After Proposition~\ref{thm:alg}, we need to prove that all modules
	$W({\OV{\lambda}})$ are generically simple. Now we discuss a basis for $W(\OV{\lambda})$ that will be useful for this result. Let $S(k,i)$ be the subset of $\LL_{k,i}$ consisting of all $(k,k)$-colored partition diagrams having the property that the diagram supported at the vertices outside $(i+1)',\dots,k'$ forms a normally ordered upward $(k,i)$-colored partition diagram, see below.
	\begin{example}\label{ex:orbit}
		The following diagram belong to $S(7,3)$:
		
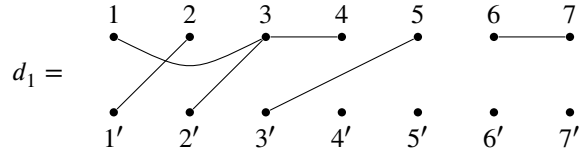
\begin{align*} 
			\begin{tikzpicture}[scale=1,mycirc/.style={circle,fill=black, minimum size=0.1mm, inner sep = 1.1pt}]
				\node (1) at (-1,0.5) {$d_1 =$};
				\node[mycirc,label=above:{$1$}] (n1) at (0,1) {};
				\node[mycirc,label=above:{$2$}] (n2) at (1,1) {};
				\node[mycirc,label=above:{$3$}] (n3) at (2,1) {};
				\node[mycirc,label=above:{$4$}] (n4) at (3,1) {};
				\node[mycirc,label=above:{$5$}] (n5) at (4,1) {};
				\node[mycirc,label=above:{$6$}] (n6) at (5,1) {};
				\node[mycirc,label=above:{$7$}] (n7) at (6,1) {};
				\node[mycirc,label=below:{$1'$}] (n1') at (0,0) {};
				\node[mycirc,label=below:{$2'$}] (n2') at (1,0) {};
				\node[mycirc,label=below:{$3'$}] (n3') at (2,0) {};
				\node[mycirc,label=below:{$4'$}] (n4') at (3,0) {};
				\node[mycirc,label=below:{$5'$}] (n5') at (4,0) {};
				\node[mycirc,label=below:{$6'$}] (n6') at (5,0) {};
				\node[mycirc,label=below:{$7'$}] (n7') at (6,0) {};
				\path[-, draw](n3) to (n2');
				\path[-,draw](n2) to (n1');
				\path[-,draw](n3) to (n4);
				\path[-,draw](n5) to (n3');
				\draw (n1)..controls(1,0.5)..(n3);
				\draw (n6) to (n7);
			\end{tikzpicture}
		\end{align*}
	\end{example}
	One can observe that $S(k,i)$ is a cross-section of the orbits of the right action of $G(r,i)$ on $\LL_{k,i}$. Now, since $\CC[G(r,i)]$ is a semisimple algebra, there exists a primitive idempotent $\epsilon_{\OV{\lambda}}\in\CC[G(r,i)]$ having the property $\CC[G(r,i)]\epsilon_{\OV{\lambda}}\cong S(\OV{\lambda})$. Then there exist $g_1,\ldots, g_{\dim S(\OV{\lambda})}\in G(r,i)$ such that $g_1\epsilon_{\OV{\lambda}},\ldots, g_{\dim S(\OV{\lambda})}\epsilon_{\OV{\lambda}}$ is a basis for $\CC[G(r,i)]\epsilon_{\OV{\lambda}}$. Without loss of generality we can assume $g_1$ is the identity element of $G(r,i)$. This gives rise to the following basis for $W(\OV{\lambda})$:
	\begin{align}\label{al:basis}
		\{d\otimes g_i \epsilon_{\OV{\lambda}}\mid d\in S(k,i) \text{ and } 1\leq i\leq \dim S(\OV{\lambda})\}. 
	\end{align}
	
	By Schur's lemma, $\epsilon_{\OV{\lambda}}\CC[G(r,i)]\epsilon_{\OV{\lambda}}$ is one dimensional and hence it is equal to $\CC \epsilon_{\OV{\lambda}}$. So, we have
	\begin{align}\label{al:bi}
		\epsilon_{\OV{\lambda}}\Inv(g_i)\Inv(d)d'g_j \epsilon_{\OV{\lambda}}=
		\begin{cases}
			\alpha x_0^{b_1}\cdots x_{r-1}^{b_{r-1}} \epsilon_{\OV{\lambda}}, & \text{ if }\Inv(d)d'\in \LL_{k,i};\\
			0, & \text{ otherwise},
		\end{cases}
	\end{align}
	where $\alpha\in\CC$ depends on $\epsilon_{\OV{\lambda}}\Inv(g_i)\Inv(d)d'g_j \epsilon_{\OV{\lambda}}$, all $b_i\in\mathbb{Z}_{\geq 0}$ and the anti-involution $\Inv$ on $\cPar(\textbf{x})$ is given in Section~\ref{sec:anti}. Note that if $\Inv(d)d'\in \LL_{k,i}$ then it actually belongs to the maximal subgroup $G(r,i)\subset \LL_{k,i}$. When $i=j$, the scalar $\alpha$ in~\eqref{al:bi} is equal to one. Moreover, in this case, all the exponents of $x_1,\ldots,x_{r-1}$ are equal to zero. 
	
	Define a bilinear form
	$\beta_{\OV{\lambda}}: W(\OV{\lambda})\times W(\OV{\lambda})\longrightarrow \CC$ by setting
	$\beta_{\OV{\lambda}}(d\otimes g_i\epsilon_{\OV{\lambda}}, d'\otimes g_j\epsilon_{\OV{\lambda}})$ to be 
	the coefficient at $\epsilon_{\OV{\lambda}}$ in \eqref{al:bi} for the element  $\epsilon_{\OV{\lambda}}\Inv(g_i)\Inv(d)d'g_j \epsilon_{\OV{\lambda}}$. Clearly, $\beta_{\OV{\lambda}}$ is nonzero.
	
	Define a bilinear map
	$\mathcal{B}_{\OV{\lambda}}(y_0,\ldots,y_{r-1}): W(\OV{\lambda})\times W(\OV{\lambda})\longrightarrow \CC[y_0,\ldots,y_{r-1}]$ by the replacing the parameters $x_0,\ldots,x_{r-1}$ by the variables $y_0,\ldots,y_{r-1}$, respectively, in the definition of $\beta_{\OV{\gamma}}$. 
	
	\begin{theorem}\label{thm:mod}
		The determinant of $\mathcal{B}_{\OV{\lambda}}(y_0,\ldots,y_{r-1})$ is a nonzero polynomial. 		Consequently, the module $W(\OV{\lambda})$ is generically simple.
	\end{theorem}
	\begin{proof}
		Let $G_{\OV{\lambda}}$ denote the matrix of the bilinear map $\mathcal{B}_{\OV{\lambda}}(y_0,\ldots,y_{r-1})$ with respect to the basis \eqref{al:basis}. The diagonal entries of $G_{\OV{\lambda}}$ are monomials in $y_0$ and, moreover, the highest power of $y_0$ in a column appears only in the corresponding diagonal entry. By expanding the determinant (for example by using Leibniz's formula), we see that in the determinant $\det G_{\OV{\lambda}}$ the highest power of $y_0$ has coefficient equal to $1$. Thus $\det G_{\OV{\lambda}}$ is a nonzero polynomial. Moreover, the evaluation of the polynomial $\det G_{\OV{\lambda}}$ at $(x_0,\ldots,x_{r-1})$ is the determinant $\det\beta_{\OV{\lambda}}$ of the bilinear form $\beta_{\OV{\lambda}}$.
		
		We claim that, for any element $(x_0,\ldots,x_{r-1})$ in $\CC^r \setminus V(\det G_{\OV{\lambda}})$, i.e., when $\beta_{\OV{\lambda}}$ is nondegenerate, the module $W_{\OV{\lambda}}$ is simple.
		Indeed, for any nonzero element in $w\in W_{\OV{\lambda}}$, using the fact that $\beta_{\OV{\lambda}}$ is nondegenerate, we can find an element in $v\in \Inv(\CC \LL_{k,i})$ such that $vw$ is a nonzero scalar multiple of $1_k\otimes \epsilon_{\OV{\lambda}}$. The latter vector generates $W_{\OV{\lambda}}$ proving that the latter is simple.
	\end{proof}

	\begin{corollary}\label{cor:gens}
		For $k\in\ZZ_{\geq 0}$,  multiparameter colored partition algebras $\cPar_k(\mathrm{\mathbf{x}})$ are generically semisimple.
	\end{corollary}
	
	\begin{proof}
		From Theorem \ref{thm:mod}, when the parameters $\textbf{x}=(x_0,\ldots,x_{r-1})$ lie in the complement of $V\big(\displaystyle{\prod_{\OV{\lambda}}}\det G_{\OV{\lambda}}\big)$, each $W_{\OV{\lambda}}$ is a simple $\cPar_k(\textbf{x})$-module. Thus the result follows from Theorem \ref{thm:alg}.
	\end{proof}

	\textbf{Additive Karoubi envelope.} Let $\Kar(\CPar(\textbf{x}))$ denote the category obtained by  first taking the additive envelope and then the Karoubi envelope of $\CPar(\textbf{x})$. Let $\C$ denote the field of complex numbers. By the generic semisimplicity of $\Kar(\CPar(\textbf{x}))$ over $\C$, for $\textbf{x}$ varying over the elements of  $\C^{r}$, we mean that there exists a measure zero subset $A$ of $\CC^{r}$ such that the category $\Kar(\CPar(\textbf{x}))$ is semisimple for all $\textbf{x}$ in $\CC^{r}\setminus A$ is semisimple. As in the case of the Deligne category (this is the classical case $r=1$, see \cite[Section 3.3]{CO}), the key reduction step to prove generic semisimplicity  of $\Kar(\CPar(\textbf{x}))$ in our case of arbitrary $r$ is to use the generic semisimplicity of multiparameter colored partition algebras. 
	
	It follows from general theory about Karoubi envelops that indecomposable objects of $\Kar(\CPar(\textbf{x}))$ are of the form $(k,e)$ where $e$ is a primitive idempotent of $\cPar_k(\textbf{x})$. 
	We have very similar setup as in the case of the Deligne category and the classification of indecomposable objects in $\Kar(\CPar(\textbf{x}))$ can be obtained by similar methods as in \cite[Section 3.1]{CO} with only minor adjustments. 
	When at least one of the parameters is nonzero, the indecomposable projective objects are in one to one correspondence with 
	multipartitions of all nonnegative integers, and when all the parameters are equal to zero, the indecomposable projective objects are in one to one correspondence with multipartitions of all positive integers. When at least one of the parameters is nonzero, then, for every $(k,e)$, there exists an idempotent $e'\in \cPar_{k+1}(\textbf{x})$ such that $(k,e)\cong (k+1,e')$ 
	
	\begin{theorem}\label{thm:karoubi}
		The category $\Kar(\CPar(\mathrm{\mathbf{x}})$ is generically semisimple.
	\end{theorem}
	\begin{proof}
		We will construct a measure zero subset $A$ of $\C^r$ such that on the complement of this set $A$ the category $\Kar(\CPar(\textbf{x}))$ is semisimple. We can assume that at least one of the parameters is nonzero. Then, from the discussion above, given two indecomposable objects $(n,e_1)$ and $(m,e_2)$, we can find $k\in\ZZ_{\geq 0}$ and primitive idempotents $f_1,f_2\in\cPar_k(\textbf{x})$ such that 
		$(n,e_1)\cong (k,f_1)$ and $(m,e_2)\cong (k,f_2)$.  Now, by the definition of a morphism space in $\Kar(\CPar(\textbf{x}))$, we have 
		\begin{align}\label{al:mor}
			\Hom_{\Kar(\CPar(\textbf{x}))}\big((k,f_1),(k,f_2)\big)\cong f_2\cPar_k(\textbf{x})f_1.
		\end{align} 
		Since $f_1$ and $f_2$ are primitive idempotents in the algebra, whenever $\cPar_k(\textbf{x})$ is semisimple, the morphism space \eqref{al:mor} is either zero ( when $f_1$ and $f_2$ are not equivalent) or $\C$ (when $f_1$ and $f_2$ are equivalent). From Theorem \ref{cor:gens}, the algebra $\cPar_k(\textbf{x})$ is generically semisimple, in particular, it is semisimple on the complement of the set $V(\displaystyle{\prod_{\OV{\lambda}}}\det G_{\OV{\lambda}})$ in $\C^r$. 
		
		So, on the complement to the union $A=\displaystyle{\bigcup_{k\in\ZZ_{\geq 0}}} V\big(\prod \det(G_{\OV{\lambda}})\big)$, the morphism space \eqref{al:mor} is either zero or $\C$, for each $k\in \ZZ_{\geq 0}$. From Theorem \ref{thm:mod}, we know that the polynomial $\displaystyle{\prod_{\OV{\lambda}}}\det(G_{\OV{\lambda}})$ is nonzero, and $A$ is a countable union of such sets, so $A$ is a measure zero subset of $\C^r$. 
		Since $\Kar(\CPar(\textbf{x}))$ is  idempotent split  (by definition) and its homomorphism spaces are finite-dimensional, it is Krull--Schmidt. Now the above discussion about homomorphism spaces between indecomposable objects implies that $\Kar(\CPar(\textbf{x}))$ is semisimple on the complement of $A$. This completes the proof.
	\end{proof}
	
	\section{ Robinson--Schensted type correspondence}\label{sec:RSK}
	Recall that $\cPar_{k}$ denote the monoid of all colored $(k,k)$-partition diagrams. In this section, we 
	compute an exponential generating function for the cardinality of $\cPar_{k}$, where $k$ varies over the nonnegative integers. We also describe two analogues of the
	Robinson--Schensted correspondence  for colored 
	partition diagrams. As an application, we show how the latter can be used to characterize the equivalence classes of Green's left, right and two-sided relations for $\cPar_k$.
	
	\subsection{Generating function for the dimension} Let $B_{k,r}$ denote the cardinality of the set of colored $(k,0)$-partition diagrams.  Then the dimension of $\cPar_{k}(\textbf{x})$ is the cardinality $B_{2k,r}$ of the monoid $\cPar_{k}$. 
	
	For $r=1$, $B_{k,r}$ is the $k$-th Bell number which has the following exponential generating function:
	\begin{displaymath}
		H_1(t):=\exp(\exp t)=\sum_{k\in\ZZ_{\geq 0}} B_{k,1} \frac{t^k}{k!}.
	\end{displaymath}
	
	By a standard counting argument (fixing one part with $l$ elements), we have
	\begin{align}\label{al:recursive}
		B_{k,r}=r \sum_{l=1}^{k} \binom{k-1}{l-1} B_{k-l,r}.
	\end{align}
	
	Below, we give an exponential generating function for for the sequence of numbers $B_{k,r}$ where $k$ varies over all nonnegative integers.
	
	\begin{proposition}
		Let $H_r(t)=\displaystyle{\sum_{k\in\ZZ_{\geq 0}}} B_{k,r} \frac{t^k}{k!}$. Then $H_r(t)=\exp(r\exp t)$. 
	\end{proposition}
	
	\begin{proof}
		Let us compute the derivative  of $H_r(t)$ with respect to $t$:
		\begin{displaymath}
			\begin{array}{rcl}
				\displaystyle   \frac{d}{dt} H_r(t)= \sum_{k\in\ZZ_{\geq 1}} B_{k,r}\, k \frac{t^{k-1}}{k!}
				&=& \displaystyle\sum_{k\in\ZZ_{\geq 1}} B_{k,r} \frac{t^{k-1}}{(k-1)!}\\
				&\overset{\eqref{al:recursive}}{=}& \displaystyle r\sum_{k\in\ZZ_{\geq 1}}\sum_{l=1}^{k} \binom{k-1}{l-1} B_{k-l,r}\frac{t^{k-1}}{(k-1)!} \\
				&=& \displaystyle r\sum_{k\in\ZZ_{\geq 1}}\sum_{l=1}^{k} B_{k-l,r} \frac{t^{l-1}}{(l-1)!}\frac{t^{k-l}}{(k-l)!}\\
				&=& \displaystyle r\sum_{i\in\ZZ_{\geq 0}}\sum_{j\in\ZZ_{\geq 0}} B_{i,r} \frac{t^i}{i!}\frac{t^j}{j!}\\
				&= & \displaystyle r\exp(t) H_r(t).
			\end{array}
		\end{displaymath}
		Thus $H_r(t)=c\exp(r\exp t)$, where $c$ is a constant. It follows directly from the definition of $H_{r}(t)$ that $H_r(0)=1$. So $c=1$ and $H_r(t)=\exp (r\exp t)$.
	\end{proof}

	\subsection{Robinson--Schensted type correspondence}
	For the complex reflection group $G(r,n)$, there are two bijections known, see~\cite{white,Shimozono}. In order to give such bijections for the colored partition monoid $\cPar_k$, we first need the following definitions and notation.

	{\bf Young diagrams and standard tableau.} For a partition $\lambda=(\lambda_1,\ldots,\lambda_n)$, we follow English convention to represent the corresponding Young diagram $\{(i,j)\mid 1\leq j \leq \lambda_i\}$. Here $i$ denotes the row number, which increases from top to bottom and $j$ denotes the column number which increases from left to right. We denote a partition and the corresponding Young diagram by the same symbol. An element of a Young diagram is called a cell. 
	A standard tableau is a filling of the cells of a Young diagram with positive integers such that the entries increase along each column and along each row. For a standard tableau $P$, by the shape $\sh(P)$ of $P$, we mean the corresponding underlying Young diagram.

	{\bf Colored permutation.} Given an element $(h,\sigma)\in G(r,n)$, the associated colored permutation is the following array
	\begin{align}\label{al:colbi}
		\begin{pmatrix}
			i_1 & i_2 & \cdots & i_n\\
			1 & 2 & \cdots & n\\
			\sigma(1) & \sigma(2) & \cdots & \sigma(n)
		\end{pmatrix},
	\end{align}
	where, for $1\leq j\leq n$, we have $\sigma(j)\in\{1,2,\dots,n\}$ and
	$i_j\in\{0,\ldots,r-1\}$ is determined by $h(j)=\zeta^{i_j}$. In \eqref{al:colbi}, the entries of the first row will be called colors and an entry of the form $\binom{j}{\sigma(j)}$ will be called a colored value.

	{\bf Maximum entry order.} Given two nonempty subsets $S_1$ and $S_2$ of $\ZZ_{\geq 0}$, then we write $S_1< S_2$ if and only if $\max S_1<\max S_2$. The order $<$ defined on the subsets is called the maximum entry order.
	
	Given an element of $d\in \cPar_{k}$ such that the propagating parts of $d$ are $(B_1,\zeta^{i_1}),\ldots, (B_n,\zeta^{i_n})$ arranged so that $$B_1^u < \cdots<B_n^u$$ with respect to the maximum entry order. Then consider the following {\em colored set-partition array}
	\begin{equation}\label{al:setpart}
		\begin{pmatrix}
			i_1 & i_2 & \cdots & i_n\\
			B_1^{u} & B_2^{u} & \cdots  & B_n^u\\[3pt]
			B_1^l & B_2^l & \cdots & B_n^l  
		\end{pmatrix}.
	\end{equation}
	Furthermore, let us assume that $(C_1,\zeta^{j_1}),\ldots, (C_m,\zeta^{j_m})$ be the nonpropagating parts on the top row of $d$ and they are arranged so that $C_1<\cdots <C_m$ with respect to the maximum entry order. Likewise, let $(C_1',\zeta^{p_1}),\ldots, (C_{q}',\zeta^{p_{q}})$ be the nonpropagating parts on the bottom row of $d$ and they are arranged so that $C_1'<\cdots <C_{q}'$ with respect to the maximum entry order.

	\subsection{Bijection 1} One can associate to \eqref{al:colbi} a pair of $r$-tuples of standard tableaux. 
	Given an $r$-tuples of partitions $\OV{\lambda}=(\lambda^{(0)},\ldots,\lambda^{(r-1)})\in\mathcal{P}_{r,n}$, let $\Tab_{\OV{\lambda}}$ denote the set of $r$-tuples $(P_0,\ldots, P_{r-1})$ of standard tableaux whose entries exhaust $\{1,2,\ldots,n\}$ such that
	$$\sh(P_s)=\lambda^{(s)}, \text{ for all } s\in\{0,1,\ldots,r-1\}.$$
	
	Suppose that  $s\in\{0,1,\ldots, r-1\}$ appears in the colored permutation \eqref{al:colbi} as a color. Consider all the positions in \eqref{al:colbi} which have color $s$. Let these positions be $j_1<\cdots <j_l$. We have the array 
	\begin{equation*}
		B_s= \begin{pmatrix}
			j_1 & j_2 & \cdots & j_l\\
			\sigma(j_1) & \sigma(j_2) & \cdots & \sigma(j_l)
		\end{pmatrix}.
	\end{equation*}
	If some color does not appear in \eqref{al:colbi}, then, by convention, we take the corresponding array to be empty. So, from \eqref{al:colbi}, we get a unique $r$-tuples of arrays $(B_0,\ldots,B_{r-1})$. 
	
	By applying usual Robinson--Schensted correspondence to each $B_s$ (see \cite{Sch}, \cite[Section~3.1]{Sagan}), one gets a pair of standard tableaux $(P_s,Q_s)$ of the same shape such that the entries of the insertion tableau $P_s$ are the elements of $\{\sigma(j_1),\ldots,\sigma(j_l)\}$ and the entries of the recording tableau $Q_s$ are the elements of $\{j_1,\ldots,j_l\}$. The empty array results into the pair of empty tableaux. 
	This allows us to associate to a colored permutation \eqref{al:colbi} a pair 
	$$\big((P_0,P_1,\ldots,P_{r-1}),(Q_0,Q_1,\ldots, Q_{r-1})\big)\in \Tab_{\OV{\lambda}} \times \Tab_{\OV{\lambda}}$$
	of $r$-tuples of tableaux, where $\sh(P_s)=\sh(Q_s)=\lambda^{(s)}$ for $s\in\{0,1,\ldots,r-1\}$ and $\OV{\lambda}=(\lambda^{(0)},\ldots,\lambda^{(r-1)})\in\mathcal{P}_{r,n}$. We denote the map which defines this bijection by
	\begin{equation}
		\RS: G(r,n)\to \bigsqcup_{\OV{\lambda}\in\mathcal{P}_{r,n}}\Tab_{\OV{\lambda}} \times \Tab_{\OV{\lambda}}.
	\end{equation}
	To our knowledge the bijection $\RS$ has first appeared in \cite[p. 236]{Stanton}. In order to give the similar bijection for colored partition diagrams, we need the following notion of set-partition $r$-tuple tableaux.

	{\bf Set-partition $r$-tuple tableau.} Let $\lambda$ be a Young diagram. If we fill the cells of $\lambda$ by the disjoint nonempty subsets of $\ZZ_{\geq 0}$ and a subset can appear only at once, then the resulting filled Young diagram is called a \emph{set-partition tableau}. A set-partition tableau is called standard if the entries of a set-partition tableau increase along the row when read from left to right and along the column when read from top to bottom with respect to the maximum entry order. A \emph{set-partition $r$-tuple tableau} is an $r$-tuples of set-partition tableau such that entries of each set-partition partition tableau in the $r$-tuple are mutually disjoint.
	
	The set of entries of a set-partition $r$-tuple tableau is called its \emph{content}. 
	
	Let $\STab_{\OV{\lambda}}$ denote the set of pairs $(P,S)$ of standard set-partition $r$-tuple tableaux such that 
	\begin{enumerate}[$(i)$]
		\item $\sh(P):=\big(\sh(P_0),\ldots,\sh(P_{r-1})\big)=\OV{\lambda}$,
		\item shapes of every part in $S$ is either a row or $\varnothing$,
		\item the union of contents of $P$ and $S$ is a set-partition of $\{1,2,\ldots,k\}$,
		\item $0<\lvert \sh(P)\rvert+\lvert \sh(S)\rvert \leq k$.
	\end{enumerate}
	
	We now give first bijection, which we continue to denote by $\RS$, for color partition diagrams.
	
	\begin{proposition}\label{prop:Rs}
		There is a bijection
		\begin{align*}
			\RS: \cPar_{k} &\to \bigcup_{s=0}^{k}\,\,\,\bigsqcup_{\OV{\lambda}\in\mathcal{P}_{r,s}}\STab_{\OV{\lambda}}^{r}\times \STab_{\OV{\lambda}}^{r}.
		\end{align*}
	\end{proposition}
	
	\begin{proof}
		For each $d\in \cPar_k$, below we
		describe how to associate to $d$
		a pair $\big((P,S),(Q,T)\big)$ in the right hand
		side of the above map.
		
		Given $d\in\cPar_k$, we have a colored partition array~\eqref{al:setpart} and also in this proof we use the notation introduced there.
		
		Note that an analogue of the bijection $\RS$ for colored permutation~\eqref{al:colbi} can be easily observed if we replace the set $\{1,2,\ldots,n\}$ by any total ordered set of cardinality $n$, for $n\in\ZZ_{\geq 0}$. In particular,  we work with maximum entry order for the case of the colored array~\eqref{al:setpart} to get standard set-partition $r$-tuple tableaux $P$ and $Q$ such that $\sh(P)=\sh(Q)$. Note that the contents of $P$ and $Q$ are $\{B_1^{l},\ldots,B_n^{l}\}$ and $\{B_1^{u},\ldots,B_n^{u}\}$, respectively. 
		
		The standard set-partition $r$-tuple tableau $S$ is obtained from the nonpropagating parts on the bottom row of $d$. For fixed $0\leq p\leq r-1$, the $p$-th slot of $S$ is $\varnothing$ if there is no set-partition $C_i$ with color $\zeta^p$; otherwise it is a row whose first cell is filled with the smallest $C_i$ among all set-partitions $C_i$ with color $\zeta^{p}$, then the next cell is filled with the next smallest entry, and continuing this fashion until we exhaust all set-partitions $C_i$ with color $\zeta^p$. 
		
		Now working with the nonpropagating parts on the top row of $d$, the tableau $T$ is obtained in the exactly same manner as $S$. Evidently, for $\OV{\lambda}=\sh(P)=\sh(Q)$, we have $(P,S)\in\STab_{\OV{\lambda}}$ and  $(Q, T)\in\STab_{\OV{\lambda}}$. By sending $d$ to $\big((P,S), (Q,T)\big)$ we define the map $\RS$ for colored partition diagrams. It is easy to verify that this map is a bijection.
	\end{proof}
	\begin{example}\label{ex:Rs}
		Let $r=5$ and $k=10$. We demonstrate the procedure given in the proof of Proposition \ref{prop:Rs} for the following colored partition diagram $d$:
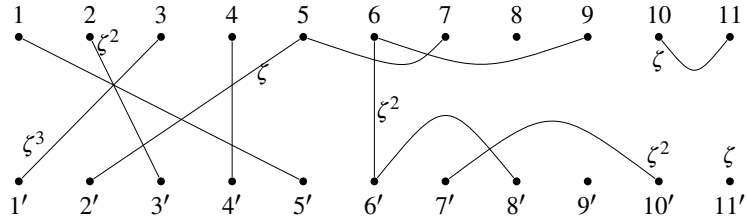
		\begin{align*} 
			\resizebox{10cm}{3cm}{	\begin{tikzpicture}[scale=1,mycirc/.style={circle,fill=black, minimum size=0.1mm, inner sep = 1.1pt}]
					\node[mycirc,label=above:{$1$}] (n1) at (0,1) {};
					\node[mycirc,label=above:{$2$}] (n2) at (1,1) {};
					\node[mycirc,label=above:{$3$}] (n3) at (2,1) {};
					\node[mycirc,label=above:{$4$}] (n4) at (3,1) {};
					\node[mycirc,label=above:{$5$}] (n5) at (4,1) {};
					\node[mycirc,label=above:{$6$}] (n6) at (5,1) {};
					\node[mycirc,label=above:{$7$}] (n7) at (6,1) {};
					\node[mycirc,label=above:{$8$}] (n8) at (7,1) {};
					\node[mycirc,label=above:{$9$}] (n9) at (8,1) {};
					\node[mycirc,label=above:{$10$},label=below:{$\zeta$}] (n10) at (9,1) {};
					\node[mycirc,label=above:{$11$}] (n11) at (10,1) {};
					\node[mycirc,label=below:{$1'$}] (n1') at (0,-1) {};
					\node[mycirc,label=below:{$2'$}] (n2') at (1,-1) {};
					\node[mycirc,label=below:{$3'$}] (n3') at (2,-1) {};
					\node[mycirc,label=below:{$4'$}] (n4') at (3,-1) {};
					\node[mycirc,label=below:{$5'$}] (n5') at (4,-1) {};
					\node[mycirc,label=below:{$6'$}] (n6') at (5,-1) {};
					\node[mycirc,label=below:{$7'$}] (n7') at (6,-1) {};
					\node[mycirc,label=below:{$8'$}] (n8') at (7,-1) {};
					\node[mycirc,label=below:{$9'$}] (n9') at (8,-1) {};
					\node[mycirc,label=below:{$10'$},label=above:{$\zeta^2$}] (n10') at (9,-1) {};
					\node[mycirc,label=below:{$11'$}, label=above:{$\zeta$}] (n11') at (10,-1) {};
					\path[-,draw](n1) to (n5');
					\path[-,draw](n2) edge node[near start, above=0.15cm] {$\zeta^2$} (n3');
					\path[-, draw](n3) edge node[near end, left=0.01cm] {$\zeta^3$} (n1');
					\path[-,draw](n4) to (n4');
					\path[-,draw](n5) edge node[near start, right] {$\zeta$} (n2');
					\draw(n5).. controls(5.5,0.5).. (n7);
					\path[-,draw](n6) edge node[midway, right=-0.09cm] {$\zeta^2$}(n6');
					\draw(n6)..controls(6.5,0.5)..(n9);
					\draw(n6')..controls(6,0.2)..(n8');
					\draw(n7')..controls(7.5,0.1)..(n10');
					\draw(n10)..controls(9.5,0.4)..(n11);
			\end{tikzpicture}}
		\end{align*}
		By considering the propagating parts we have the following colored set-partition array
		\begin{equation*}
			B=\begin{pmatrix}
				0 & 2 & 3 & 0 & 1& 2\\
				\{1\} & \{2\} & \{3\} & \{4\} & \{5,7\} & \{6,9\}\\
				\{5'\} & \{3'\} & \{1'\} & \{4'\} & \{2'\} & \{6',8'\}
			\end{pmatrix}.
		\end{equation*}
		By collecting the parts of $B$ which have the same color, we get following sequence of arrays:
		\begin{equation*}
			B_0=   \begin{pmatrix}
				\{1\} & \{4\}\\
				\{5'\} & \{4'\}\\
			\end{pmatrix},
			B_1=  \begin{pmatrix}
				\{5,7\}\\
				\{2'\}
			\end{pmatrix},
			B_2= \begin{pmatrix}
				\{2\} & \{6,9\}\\
				\{3'\} & \{6',8'\}
			\end{pmatrix},
			B_3=   \begin{pmatrix}
				\{3\} \\
				\{1'\}
			\end{pmatrix},
			B_4=\varnothing.
		\end{equation*}
		By applying usual Robinson--Schensted correspondence to $B_0,B_1,B_3,B_4$, where the order is now the maximum entry order, we get the following pairs of standard set-partition tableaux:
		\begin{small}
			\begin{align*}
				&B_0 \to (P_0,Q_0)= \begin{pmatrix}
					\hspace{0.1cm}\ytableausetup{boxsize=2em} \begin{ytableau}
						\{4'\} \\ \{5'\} 
					\end{ytableau}\,, \hspace{0.1cm}\ytableausetup{boxsize=2em} \begin{ytableau}
						\{1\} \\ \{4\} 
					\end{ytableau}\hspace{0.1cm}
				\end{pmatrix}, 
				& B_1\to (P_1,Q_1)= \begin{pmatrix}
					\hspace{0.1cm}\ytableausetup{boxsize=2.5em} \begin{ytableau}
						\{2'\} 
					\end{ytableau}\,, \hspace{0.1cm}\ytableausetup{boxsize=2.5em} \begin{ytableau}
						\{5,7\}
					\end{ytableau}\hspace{0.1cm}
				\end{pmatrix}\\
				& B_2\to (P_2,Q_2)= \begin{pmatrix}
					\hspace{0.1cm}\ytableausetup{boxsize=3em} \begin{ytableau}
						\{3'\} & \{6',8'\}
					\end{ytableau}\,, \hspace{0.1cm}\ytableausetup{boxsize=3em} \begin{ytableau}
						\{2\}&\{6,9\}
					\end{ytableau}\hspace{0.1cm}
				\end{pmatrix},
				& B_3\to (P_3,Q_3)= \begin{pmatrix}
					\hspace{0.1cm}\ytableausetup{boxsize=2em} \begin{ytableau}
						\{1'\} 
					\end{ytableau}\,, \hspace{0.1cm}\ytableausetup{boxsize=2em} \begin{ytableau}
						\{3\}
					\end{ytableau}\hspace{0.1cm}
				\end{pmatrix}\\
				&B_4\to (P_4,Q_4)= \begin{pmatrix}
					\hspace{0.1cm}\varnothing, \varnothing\hspace{0.1cm}
				\end{pmatrix}.
			\end{align*}
		\end{small}
		
		From the bottom and the top nonpropagating parts, we get
		$$T= \begin{pmatrix}
			\resizebox{2cm}{!}{\ytableausetup{boxsize=3.5em} \begin{ytableau}
					\hspace{0.1cm}\{8\}
				\end{ytableau}\,,\hspace{0.1cm} \ytableausetup{boxsize=3.5em} \begin{ytableau}
					\{10,11\}
			\end{ytableau}}\, ,\hspace{0.1cm} \varnothing,\hspace{0.1cm} \varnothing,\hspace{0.1cm} \varnothing \hspace{0.1cm}
		\end{pmatrix},\quad S= \begin{pmatrix}
			\resizebox{3cm}{!}{\ytableausetup{boxsize=4em} \begin{ytableau}
					\hspace{0.1cm}\{9'\}
				\end{ytableau}\,,\hspace{0.1cm} \ytableausetup{boxsize=4em} \begin{ytableau}
					\{11'\}
				\end{ytableau}\, ,\hspace{0.1cm} 
				\ytableausetup{boxsize=4em} \begin{ytableau}
					\{7',10'\}
			\end{ytableau}}\,, \hspace{0.1cm} \varnothing,\hspace{0.1cm} \varnothing \hspace{0.1cm}
		\end{pmatrix}.$$
		Under the bijection discussed in the proof of Proposition~\ref{prop:Rs}, the colored partition diagram $d$ maps to $$\big(((P_0,P_1,P_2,P_3,P_4), S), (Q_0,Q_1,Q_2,Q_3,Q_4), T)\big).$$ 
	\end{example}
	
	\subsubsection{Young's natural bases for standard modules} Motivated by Young's natural bases constructed in \cite{Grood} for simple modules over rook monoids and in \cite{TTN} for simple modules over partition algebras, here we give an analogous basis for standard modules over the multiparameter colored partition algebra $\cPar_k({\bf x})$. Note that, by Proposition~\ref{thm:mod}, these modules are generically simple. For $1\leq i\leq k$ and $\OV{\lambda}\in\mathcal{P}_{r,i}$, we know that
	\begin{displaymath}
		W(\OV{\lambda})= \CC\mathcal{L}(k,i) \otimes_{\CC G(r,i)}S(\OV{\lambda})=\CC S(k,i)\otimes S(\OV{\lambda}).
	\end{displaymath}
	
	Recall that $S(k,i)$ is a cross-section for the orbit space $\mathcal{L}(k,i)/G(r,i)$. Let $\mathcal{B}$ be a basis for $S(\OV{\lambda})$. Then the following is a basis for $W(\OV{\lambda})$:
	$$\{d\otimes v\mid d\in S(k,i), v\in \mathcal{B}\}.$$
	
	Let $d$ be a colored partition diagram and $d_1$ be an element in $S(k,i)$. If $dd_1\in\mathcal{L}(k,i)$, then there exists $d_2\in S(k,i)$ and $g\in G(r,i)$ such that $dd_1=d_2g$. The action of $d$ on the basis element $d_1\otimes v$ is given as follows:
	\begin{equation}\label{eq:action}
		\begin{cases}
			d_2\otimes gv, & \text{ if } dd_1\in\mathcal{L}(k,i);\\
			0,& \text{ otherwise}.
		\end{cases}
	\end{equation}
	
	Young's natural basis for $S(\OV{\lambda})$ was constructed in~\cite[Theorem 4.2]{Can}. Recall that $\Tab_{\OV{\lambda}}$ denotes the set consisting of all $r$-tuples of standard tableaux of shape $\OV{\lambda}$ and content $\{1,2,\ldots,i\}$. The  set  $\Tab_{\OV{\lambda}}$ is an indexing set for Young's natural basis for $S(\OV{\lambda})$. For $T\in\Tab_{\OV{\lambda}}$, let $v_{T}\in S(\OV{\lambda})$ denote the corresponding basis element. 
	
	Given $d\in S(k,i)$, let $(C_1,\zeta^0), \ldots, (C_i,\zeta^0)$ be the propagating parts in the top row of $d$ and $(D_1,\zeta^j_1),\ldots, (D_l,\zeta^{j_l})$ be the remaining parts appearing in the top row of $d$. Furthermore, we assume that $C_1<\cdots < C_i$ with respect to the maximum entry order. For an $r$-tuple $T=(P_0',\ldots, P_{r-1}')$ of standard tableaux of shape $\OV{\lambda}$, we associate a pair of standard $r$-tuple set-partition tableaux 
	$\big((P_0,\ldots, P_{r-1}), (Q_0,\ldots, Q_{r-1})\big)\in \STab_{\OV{\lambda}}$ to a basis element $d\otimes v_T$ as follows. Note that the content of $T$ is $\{1,2,\ldots,i\}$. \label{pa:pro}
	\begin{itemize}
		\item In $T=(P_0',\ldots, P_{r-1}')$, for all $1\leq j\leq i$, we replace $j$ by the set-partition $C_j$ to obtain the standard  set-partition $r$-tuple tableau $(P_0,\ldots, P_{r-1})$.
		
		\item For $0\leq s\leq r-1$, if there is no nonpropagating part $(D_a,\zeta^{j_a})$ in the top row of $d$ such that $\zeta^{a}=\zeta^s$, then we let the shape of $Q_s$ be empty. Otherwise, if $(D_{a_1},\zeta^s),\ldots, (D_{a_b},\zeta^{s})$ are  all nonpropagating parts in the top row of $d$ that have the color $\zeta^s$, then we let $Q_j$ be the unique standard set-partition tableau whose shape is just one row and whose content is $\{D_{a_1},\ldots,D_{a_b}\}$.
	\end{itemize}
	\begin{example}\label{ex:tab}
		For $r=4$, let $d\in S(14,5)$ be the following diagram:
		
		\begin{center}	\resizebox{12cm}{3cm}
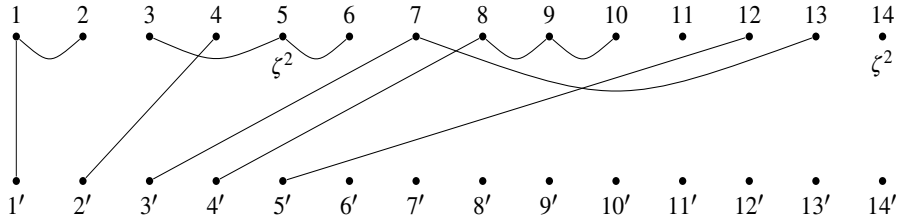
{	\begin{tikzpicture}[scale=1,mycirc/.style={circle,fill=black, minimum size=0.1mm, inner sep = 1.1pt}]
					\node[mycirc,label=above:{$1$}] (n1) at (0,1) {};
					\node[mycirc,label=above:{$2$}] (n2) at (1,1) {};
					\node[mycirc,label=above:{$3$}] (n3) at (2,1) {};
					\node[mycirc,label=above:{$4$}] (n4) at (3,1) {};
					\node[mycirc,label=above:{$5$}] (n5) at (4,1) {};
					\node[mycirc,label=above:{$6$}] (n6) at (5,1) {};
					\node[mycirc,label=above:{$7$}] (n7) at (6,1) {};
					\node[mycirc,label=above:{$8$}] (n8) at (7,1) {};
					\node[mycirc,label=above:{$9$}] (n9) at (8,1) {};
					\node[mycirc,label=above:{$10$}] (n10) at (9,1) {};
					\node[mycirc,label=above:{$11$}] (n11) at (10,1) {};
					\node[mycirc,label=above:{$12$}] (n12) at (11,1) {};
					\node[mycirc,label=above:{$13$}] (n13) at (12,1) {};
					\node[mycirc,label=above:{$14$},label=below:{$\zeta^2$}] at (13,1) {};
					\node[mycirc,label=below:{$1'$}] (n1') at (0,-1) {};
					\node[mycirc,label=below:{$2'$}] (n2') at (1,-1) {};
					\node[mycirc,label=below:{$3'$}] (n3') at (2,-1) {};
					\node[mycirc,label=below:{$4'$}] (n4') at (3,-1) {};
					\node[mycirc,label=below:{$5'$}] (n5') at (4,-1) {};
					\node[mycirc,label=below:{$6'$}] (n6') at (5,-1) {};
					\node[mycirc,label=below:{$7'$}] (n7') at (6,-1) {};
					\node[mycirc,label=below:{$8'$}] (n8') at (7,-1) {};
					\node[mycirc,label=below:{$9'$}] (n9') at (8,-1) {};
					\node[mycirc,label=below:{$10'$}] (n10') at (9,-1) {};
					\node[mycirc,label=below:{$11'$}] (n11') at (10,-1) {};
					\node[mycirc,label=below:{$12'$}] (n12') at (11,-1) {};
					\node[mycirc,label=below:{$13'$}] (n13') at (12,-1) {};
					\node [mycirc,label=below:{$14'$}] at (13,-1) {};
					\node at (4,0.6) {$\zeta^2$};

					\path[-,draw] (n1) to (n1');
					\path[-,draw] (n4) to (n2');
					\path[-,draw] (n7) to (n3');
					\path[-,draw] (n8) to (n4');
					\path[-,draw] (n12) to (n5');
					\draw (n1).. controls(0.5,0.6).. (n2); 
					\draw (n3).. controls(3,0.6).. (n5);
					\draw (n5)..controls(4.5,0.6).. (n6);
					\draw (n7).. controls(9,0).. (n13);
					\draw (n8).. controls(7.5,0.6).. (n9);
					\draw (n9).. controls(8.5,0.6)..(n10);
			\end{tikzpicture}}
		\end{center}
		
		The propgating parts in the top row of $d$ are arranged with respect to the maximum entry order as follows:
		\begin{displaymath}
			\{1,2\}<\{4\}<\{8,9,10\}<\{12\}<\{13\}.
		\end{displaymath}
		
		There is only one nonpropagating part with color $\zeta^0$ in the top row of 
		$d$, namely $\{11\}$. There are only two nonpropagating parts with color $\zeta^2$ in the top row of $d$, namely $\{3,5,6\}$ and $\{14\}$. 
		
		Let $T$ be following $4$-tuple of standard tableaux:
		
		$$\begin{pmatrix}
			\hspace{0.1cm}\ytableausetup{boxsize=1em} \begin{ytableau}
				1 & 3 
			\end{ytableau}\,, \hspace{0.1cm}\ytableausetup{boxsize=1em} \begin{ytableau}
				2 \\ 4 
			\end{ytableau}\,,
			\hspace{0.1cm}
			\ytableausetup{boxsize=1em} \begin{ytableau}
				5\end{ytableau}\, , \hspace{0.1cm}\varnothing\hspace{0.1cm}
		\end{pmatrix}.$$
		
		Then the basis element $d\otimes v_{T}$ corresponds to the following tuple:
		
		$$\begin{pmatrix}\begin{pmatrix}
				\hspace{0.1cm}\resizebox{2cm}{!}{\ytableausetup{boxsize=4em} \begin{ytableau}
						\{1,2\} & \{8,9,10\} 
				\end{ytableau}}\,, \hspace{0.1cm}\resizebox{1 cm}{!}{\ytableausetup{boxsize=4em} \begin{ytableau}
						\{4\} \\ \{12\} 
				\end{ytableau}}\,,
				\hspace{0.1cm}
				\resizebox{1.2cm}{!}{ \ytableausetup{boxsize=4em} \begin{ytableau}
						\{7,13\}\end{ytableau}}\, , \hspace{0.1cm}\varnothing\hspace{0.1cm}
			\end{pmatrix}\,, \hspace{0.1cm} \begin{pmatrix}
				\hspace{0.1cm}\resizebox{1.2cm}{!}{\ytableausetup{boxsize=4em} \begin{ytableau}
						\{11\}  
				\end{ytableau}}\,, \varnothing\,,\hspace{0.1cm}
				\resizebox{2cm}{!}{ \ytableausetup{boxsize=4em} \begin{ytableau}
						\{3,5,6\} & \{14\}
				\end{ytableau}}\,,
				\varnothing\hspace{0.1cm}
			\end{pmatrix}
		\end{pmatrix}.
		$$
	\end{example}
	
	For $1\leq i\leq k$ and $\OV{\lambda}\in\mathcal{P}_{r,i}$, let $\mathcal{ST}_{\OV{\lambda}}$ be the set consisting of all pairs $\big((P_0,\ldots,P_{r-1}),(Q_0,\ldots,Q_{r-1})\big)$ of set-partition $r$-tuple tableaux such that
	\begin{enumerate}
		\item $\big(\sh(P_0),\ldots,\sh(P_{r-1}))\big)=\OV{\lambda}$,
		
		\item for $0\leq j\leq r-1$, the shape $\sh(Q_j)$ is either empty or a row,
		
		\item $(Q_0,\ldots,Q_{r-1})$ is a standard set-partition $r$-tuple tableau,
		
		\item the union of the contents of $(P_0,\ldots, P_{r-1})$ and $(Q_0,\ldots,Q_{r-1})$ is a set-partition of $\{1,2,\ldots,k\}$.
	\end{enumerate}
	Recall that $\STab_{\OV{\lambda}}$ denote the subset of $\mathcal{ST}_{\OV{\lambda}}$ such that for an element $\left((P_0,\ldots,P_{r-1}),(Q_0,\ldots,Q_{r-1})\right)\in\STab_{\OV{\lambda}}$ the set-partition $r$-tuple tableau $(P_0,\ldots,P_{r-1})$ is standard.
	
	From the discussion just before Example~\ref{ex:tab}, we have the following proposition.
	
	\begin{proposition}
		For $d\in S(k,i)$ and $T\in\Tab_{\OV{\lambda}}$, the assignment of $d\otimes v_{T}$ to an element of $\STab_{\OV{\lambda}}$ defines a bijection between $\{d\otimes v_{T}\mid d\in S(k,i), T\in\Tab_{\OV{\lambda}}\}$ and $\STab_{\OV{\lambda}}$.
	\end{proposition}
	
	Define $W'(\OV{\lambda})$ to be the vector space which has a basis indexed by $\STab_{\OV{\lambda}}$. 
	Next goal is to give an action of $\cPar_k(\textbf{x})$ on $W'(\OV{\lambda})$. To do this, we give some definitions motivated by \eqref{eq:action}. These definitions will be used later to describe an  action of $\cPar_k({\bf x})$ on $W'(\OV{\lambda})$. 
	
	Let $U=\left((P_0,\ldots, P_{r-1}),(Q_0,\ldots,Q_{r-1})\right)\in \mathcal{ST}_{\OV{\lambda}}$.   First, construct a colored $(k,0)$-partition diagram $\pi$ whose parts are the contents of $(P_0,\ldots,P_{r-1})$ and $(Q_0,\ldots, Q_{r-1})$. Moreover, the parts arising from $(P_0,\ldots,P_{r-1})$ have their colors equal to the identity; while the parts arising from $Q_j$ have their colors equal to $\zeta^{j}$, for $0\leq j\leq r-1$. Let $d$ be a colored partition diagram in $\cPar_{k}({\bf x})$. Then the diagram concatenation $d\circ \pi$ is a colored $(k,0)$-partition diagram.  
	Now, perform the following procedure to get $dU:=\left((P_0',\ldots,P_{r-1}'),(Q_0',\ldots,Q_{r-1}')\right)$.
	
	\begin{enumerate}
		\item For $0\leq j\leq r-1$, the entry of a cell $c$ of $P_j'$ is obtained by replacing the entry of the same cell $c$ of $P_j$ by the parts it is connected to in the top row of $d\circ \pi$. We multiply the colors of all the parts that get combined in the concatenation and let $\zeta^{s_{c}}$ be the resulting color. Define the following function $\phi$ of $d$ and $U$, which takes value in $C_r$,
		\begin{equation}\label{eq:phi}
			\phi(d,U)=\prod_{j=0}^{r-1}\left(\zeta^{j}\prod_{c \text{ is a cell in $P_{j}'$}} \zeta^{s_{c}}\right).
		\end{equation}
		
		\item Collect the parts of the top row of $d\circ \pi$ which are nonpropagating and the parts of the top row which are connected to only parts that arise from the content of $(Q_0,\ldots,Q_{r-1})$. In this collection $\mathcal{C}$, parts can possibly have non-trivial colors. For $0\leq j\leq r-1$, if there is no part in $\mathcal{C}$ that has color $\zeta^j$, then let $Q_j'$ be empty. Otherwise, let $Q_j'$ be the standard set-partition tableau whose shape is a row and the entries are the parts in $\mathcal{C}$ that have color $\zeta^{j}$.
		
		Note that $(P_0',\ldots,P_{r-1}')$ need not be a set-partition $r$-tuple tableau (in particular, $dU\notin\mathcal{ST}_{\OV{\lambda}}$) and this occurs if and only if one of the following arises
		\begin{enumerate}[($a$)]
			\item two parts of $\pi$ coming from the content of $(P_0,\ldots,P_{r-1})$ get connected in $d\circ \pi$,
			\item a part of $\pi$  coming from the content of $(P_0,\ldots,P_{r-1})$ get connected with a nonpropagating part of $d$ in the diagram concatenation $d\circ \pi$.
		\end{enumerate}
	\end{enumerate}
	We illustrate the above definitions by the following example.
	\begin{example}\label{ex:mult}
		For $r=4$, consider the following element $U$ of $\mathcal{ST}_{\OV{\lambda}}$, where 
		$\OV{\lambda}=\begin{pmatrix}
			\hspace{0.1cm}\ytableausetup{boxsize=1em} \begin{ytableau}
				\text{ } & \text{ } 
			\end{ytableau}\,, \hspace{0.1cm}\ytableausetup{boxsize=1em} \begin{ytableau}
				\text{ }  \\ \text{ }  
			\end{ytableau}\,,
			\hspace{0.1cm}
			\ytableausetup{boxsize=1em} \begin{ytableau}
				\text{ }
			\end{ytableau}\, , \hspace{0.1cm}\varnothing\hspace{0.1cm}
		\end{pmatrix}$:
		
		$\begin{pmatrix}\begin{pmatrix}
				\hspace{0.1cm}\resizebox{2cm}{!}{\ytableausetup{boxsize=4em} \begin{ytableau}
						\{1,2\} & \{8,9,10\} 
				\end{ytableau}}\,, \hspace{0.1cm}\resizebox{1 cm}{!}{\ytableausetup{boxsize=4em} \begin{ytableau}
						\{4\} \\ \{12\} 
				\end{ytableau}}\,,
				\hspace{0.1cm}
				\resizebox{1.2cm}{!}{ \ytableausetup{boxsize=4em} \begin{ytableau}
						\{7,13\}\end{ytableau}}\, , \hspace{0.1cm}\varnothing\hspace{0.1cm}
			\end{pmatrix}\,, \hspace{0.1cm} \begin{pmatrix}
				\hspace{0.1cm}\resizebox{1.2cm}{!}{\ytableausetup{boxsize=4em} \begin{ytableau}
						\{11\}  
				\end{ytableau}}\,, \varnothing\,,\hspace{0.1cm}
				\resizebox{2cm}{!}{ \ytableausetup{boxsize=4em} \begin{ytableau}
						\{3,5,6\} & \{14\}
				\end{ytableau}}\,,
				\varnothing\hspace{0.1cm}
			\end{pmatrix}
		\end{pmatrix}
		$ 
		
		The element $U$ corresponds to a colored $(14,0)$-partition diagram $\pi$. For the sake of visibility, the parts arising from the first tuple of $U$ are colored in red and the parts arising from the second tuple of $U$ are colored in blue.
		The diagram concatenation $d\circ \pi$ is the following graph:
		
		
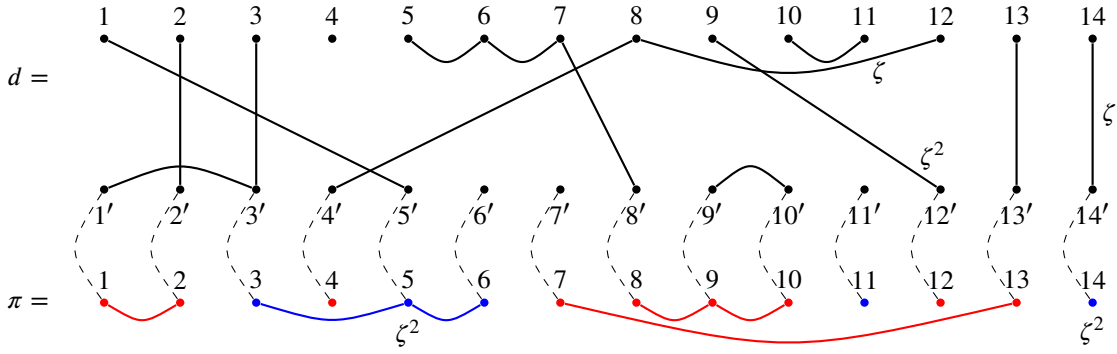
\begin{tikzpicture}[scale=1,mycirc/.style={circle,fill=black, minimum size=0.1mm, inner sep = 1.1pt}]
			\node (1) at (-1,0.5) {$d =$};
			\node[mycirc,label=above:{$1$}] (n1) at (0,1) {};
			\node[mycirc,label=above:{$2$}] (n2) at (1,1) {};
			\node[mycirc,label=above:{$3$}] (n3) at (2,1) {};
			\node[mycirc,label=above:{$4$}] (n4) at (3,1) {};
			\node[mycirc,label=above:{$5$}] (n5) at (4,1) {};
			\node[mycirc,label=above:{$6$}] (n6) at (5,1) {};
			\node[mycirc,label=above:{$7$}] (n7) at (6,1) {};
			\node[mycirc,label=above:{$8$}] (n8) at (7,1) {};
			\node[mycirc,label=above:{$9$}] (n9) at (8,1) {};
			\node[mycirc,label=above:{$10$}] (n10) at (9,1) {};
			\node[mycirc,label=above:{$11$}] (n11) at (10,1) {};
			\node[mycirc,label=above:{$12$}] (n12) at (11,1) {};
			\node[mycirc,label=above:{$13$}] (n13) at (12,1) {};
			\node[mycirc,label=above:{$14$}] (n14) at (13,1) {};
			\node[mycirc,label=below:{$1'$}] (n1') at (0,-1) {};
			\node[mycirc,label=below:{$2'$}] (n2') at (1,-1) {};
			\node[mycirc,label=below:{$3'$}] (n3') at (2,-1) {};
			\node[mycirc,label=below:{$4'$}] (n4') at (3,-1) {};
			\node[mycirc,label=below:{$5'$}] (n5') at (4,-1) {};
			\node[mycirc,label=below:{$6'$}] (n6') at (5,-1) {};
			\node[mycirc,label=below:{$7'$}] (n7') at (6,-1) {};
			\node[mycirc,label=below:{$8'$}] (n8') at (7,-1) {};
			\node[mycirc,label=below:{$9'$}] (n9') at (8,-1) {};
			\node[mycirc,label=below:{$10'$}] (n10') at (9,-1) {};
			\node[mycirc,label=below:{$11'$}] (n11') at (10,-1) {};
			\node[mycirc,label=below:{$12'$}] (n12') at (11,-1) {};
			\node[mycirc,label=below:{$13'$}] (n13') at (12,-1) {};
			\node [mycirc,label=below:{$14'$}] (n14') at (13,-1) {};
			\node at (10.2,0.55) {$\zeta$}; 
			\node at (10.9,-0.55) {$\zeta^2$};
			
			\path[thick,draw] (n2) to (n2');
			\path[thick,draw] (n3) to (n3');
			\path[thick,draw] (n1) to (n5');
			\path[thick,draw] (n7) to (n8');
			\path[thick,draw] (n9) to (n12');
			\path[thick,draw] (n8) to (n4');
			\path[thick,draw] (n13) to (n13');
			\draw[thick] (n1').. controls(1,-0.6).. (n3'); 
			\draw[thick] (n5)..controls(4.5,0.6).. (n6);
			\draw[thick] (n6).. controls(5.5,0.6).. (n7);
			\draw[thick] (n8).. controls(9,0.4).. (n12);
			\draw[thick] (n10).. controls(9.5,0.6).. (n11);
			\draw[thick] (n9').. controls(8.5,-0.6)..(n10');
			\draw[thick] (n14) edge node[midway,right] {$\zeta$} (n14'); 
			
			\node at (-1,-2.5) {$\pi=$};
			\node[mycirc,label=above:{$1$},red] (m1) at (0,-2.5) {};
			\node[mycirc,label=above:{$2$},red] (m2) at (1,-2.5) {};
			\node[mycirc,label=above:{$3$},blue] (m3) at (2,-2.5) {};
			\node[mycirc,label=above:{$4$},red] (m4) at (3,-2.5) {};
			\node[mycirc,label=above:{$5$},blue] (m5) at (4,-2.5) {};
			\node[mycirc,label=above:{$6$},blue] (m6) at (5,-2.5) {};
			\node[mycirc,label=above:{$7$},red] (m7) at (6,-2.5) {};
			\node[mycirc,label=above:{$8$},red] (m8) at (7,-2.5) {};
			\node[mycirc,label=above:{$9$},red] (m9) at (8,-2.5) {};
			\node[mycirc,label=above:{$10$},red] (m10) at (9,-2.5) {};
			\node[mycirc,label=above:{$11$},blue] (m11) at (10,-2.5) {};
			\node[mycirc,label=above:{$12$},red] (m12) at (11,-2.5) {};
			\node[mycirc,label=above:{$13$},red] (m13) at (12,-2.5) {};
			\node[mycirc,label=above:{$14$},label=below:{$\zeta^2$},blue] (m14) at (13,-2.5) {};
			\node at (4,-2.9) {$\zeta^2$}; 
			
			\draw[dashed] (n1') .. controls(-0.5,-1.8)..(m1);
			\draw[dashed] (n2') .. controls(0.5,-1.8)..(m2);
			\draw[dashed] (n3') .. controls(1.5,-1.8)..(m3);
			\draw[dashed] (n4') .. controls(2.5,-1.8)..(m4);
			\draw[dashed] (n5') .. controls(3.5,-1.8)..(m5);
			\draw[dashed] (n6') .. controls(4.5,-1.8)..(m6);
			\draw[dashed] (n7') .. controls(5.5,-1.8)..(m7);
			\draw[dashed] (n8') .. controls(6.5,-1.8)..(m8);
			\draw[dashed] (n9') .. controls(7.5,-1.8)..(m9);
			\draw[dashed] (n10') .. controls(8.5,-1.8)..(m10);
			\draw[dashed] (n11') .. controls(9.5,-1.8)..(m11);
			\draw[dashed] (n12') .. controls(10.5,-1.8)..(m12);
			\draw[dashed] (n13') .. controls(11.5,-1.8)..(m13);
			\draw[dashed] (n14') .. controls(12.5,-1.8)..(m14);
			\draw[red,thick] (m1).. controls(0.5,-2.8).. (m2); 
			\draw[blue,thick] (m3).. controls(3,-2.8).. (m5);
			\draw[blue,thick] (m5)..controls(4.5,-2.8).. (m6);
			\draw[red,thick] (m7).. controls(9,-3.2).. (m13);
			\draw[red,thick] (m8).. controls(7.5,-2.8).. (m9);
			\draw[red,thick] (m9).. controls(8.5,-2.8)..(m10);
		\end{tikzpicture}
		
		\begin{enumerate}
			\item the part $\{1,2\}$ is replaced by $\{1,2,3\}$ and this carries the color $\zeta^{2}$;
			\item The part $\{8,9,10\}$ is replaced by $\{5,6,7\}$ and this carries the trivial color;
			\item the part $\{4\}$ is replaced by $\{8,12\}$ and this carries the color $\zeta$;
			\item the part $\{12\}$ is replaced by $\{9\}$ and this carries the color $\zeta^2$;
			\item the part $\{7,13\}$ is replaced by $\{13\}$ and this carries the trivial color;
		\end{enumerate}  
		By replacing the above mentioned parts in the first tuple of $U$, we get the following $(P_0',P_1',P_2',P_3')$:
		$$\begin{pmatrix}
			\hspace{0.1cm}\resizebox{2cm}{!}{\ytableausetup{boxsize=4em} \begin{ytableau}
					\{1,2,3\} & \{5,6,7\} 
			\end{ytableau}}\,, \hspace{0.1cm}\resizebox{1.2 cm}{!}{\ytableausetup{boxsize=4em} \begin{ytableau}
					\{8,12\} \\ \{9\} 
			\end{ytableau}}\,,
			\hspace{0.1cm}
			\resizebox{1.2cm}{!}{ \ytableausetup{boxsize=4em} \begin{ytableau}
					\{13\}\end{ytableau}}\, , \hspace{0.1cm}\varnothing\hspace{0.1cm}
		\end{pmatrix}.$$
		
		The value $\phi(d,U)$ is equal to $\zeta^3$.
		
		The part $\{4\}$ is the only nonpropagating and it carries the trivial color. The part $\{14\}$ (only) connects to a part arising from the second tuple of $U$ and the total color it carries is $\zeta^{3}$. So $(Q_0',Q_1',Q_2',Q_3')$ is the following:
		
		$$\begin{pmatrix}
			\hspace{0.1cm}\resizebox{1cm}{!}{\ytableausetup{boxsize=4em} \begin{ytableau}
					\{4\}  
			\end{ytableau}}\,, \varnothing\,,\hspace{0.1cm}\varnothing\,,\hspace{0.1cm}
			\resizebox{1cm}{!}{ \ytableausetup{boxsize=4em} \begin{ytableau}
					\{14\}
			\end{ytableau}}\,
			\hspace{0.1cm}
		\end{pmatrix}$$
	\end{example}
	
	For $U\in\STab_{\OV{\lambda}}$, now we are ready to define an action of $\cPar_{k}({\bf x})$ on a basis element $w_{U}$ of $W'(\OV{\lambda})$. Let $d$ be a colored partition diagram in $\cPar_{k}({\bf x})$. Set
	
	\begin{equation}
		dw_{U}=\begin{cases}
			x_{0}^{n_0}\cdots x_{r-1}^{n_{r-1}}{\phi(d,U)}\tilde{w}_{dU}, &  \text{ if } dU\in\mathcal{ST}_{\OV{\lambda}};\\
			0, & \text{ if } dU\notin\mathcal{ST}_{\OV{\lambda}}.
		\end{cases}
	\end{equation}
	Here we have:
	\begin{enumerate}[($a$)]
		\item The values $n_0,\ldots,n_{r-1}$ are the number of connected components carrying the colors $\zeta^{0},\ldots,\zeta^{r-1}$, respectively, that were removed in the concatenation $d\circ\pi$, moreover, $\phi(d,U)$ is as defined in \eqref{eq:phi}.
		
		\item If $dU\in \STab_{\OV{\lambda}}\subset\mathcal{ST}_{\OV{\lambda}}$, then  
		$dU=\left((P_0',\ldots,P_{r-1}'),(Q_{0}',\ldots,Q_{r-1}')\right)$ is standard and $\tilde{w}_{dU}=w_{dU}$.
		
		\item If $dU\in \mathcal{ST}_{\OV{\lambda}}$ but $dU=\left((P_0',\ldots,P_{r-1}'),(Q_{0}',\ldots,Q_{r-1}')\right)\notin \STab_{\OV{\lambda}}$, then, in particular, $(P_0',\ldots,P_{r-1}')$ is not standard. Using maximum entry order, we uniquely transport $(P_0',\ldots,P_{r-1}')$ to an element, whose entries are nonnegative integers, in the Specht module corresponding to $\OV{\lambda}$ and now we perform the straightening  algorithm~\cite[Theorem 5.6]{Can} to write this element as linear combination of basis elements which are indexed by elements of $\Tab_{\OV{\lambda}}$. Once again using maximum entry order, we transport back from nonnegative integer to set-partitions, and finally this allows us to write $\tilde{w}_{dU}$ as a linear combination of the basis elements $w_{U'}$, where $U'=\left((P_0'',\ldots,P_{r-1}''),(Q_{0}',\ldots,Q_{r-1}')\right)\in\STab_{\OV{\lambda}}$. Note that the algorithm is essentially performed on $(P_0',\ldots,P_{r-1}')$ and the second tuple of $dU$ remains unchanged.
	\end{enumerate}
	
	\begin{example}
		In the setup of Example~\ref{ex:mult}, we have $n_0=1, n_1=n_2=n_3=0$, and $dU\in \mathcal{ST}_{\OV{\lambda}}$ but it is not standard. After performing the algorithm in~\cite[Theorem 5.6]{Can}, we get $\tilde{w}_{dU}=-w_{U'}$, where 
		\begin{equation*}
			U'=\begin{pmatrix}\begin{pmatrix}
					\hspace{0.1cm}\resizebox{2cm}{!}{\ytableausetup{boxsize=4em} \begin{ytableau}
							\{1,2,3\} & \{5,6,7\} 
					\end{ytableau}}\,, \hspace{0.1cm}\resizebox{1.2 cm}{!}{\ytableausetup{boxsize=4em} \begin{ytableau}
							\{9\} \\ \{8,12\} 
					\end{ytableau}}\,,
					\hspace{0.1cm}
					\resizebox{1.2cm}{!}{ \ytableausetup{boxsize=4em} \begin{ytableau}
							\{13\}\end{ytableau}}\, , \hspace{0.1cm}\varnothing\hspace{0.1cm}
				\end{pmatrix},\hspace{0.1cm}\begin{pmatrix}
					\hspace{0.1cm}\resizebox{1cm}{!}{\ytableausetup{boxsize=4em} \begin{ytableau}
							\{4\}  
					\end{ytableau}}\,, \varnothing\,,\hspace{0.1cm}\varnothing\,,\hspace{0.1cm}
					\resizebox{1cm}{!}{ \ytableausetup{boxsize=4em} \begin{ytableau}
							\{14\}
					\end{ytableau}}\,
					\hspace{0.1cm}
				\end{pmatrix}
			\end{pmatrix}.
		\end{equation*}
		Finally, $dw_{U}=-x_0\zeta^{3}w_{U'}$.
	\end{example}
	
	\begin{proposition}
		As $\cPar_{k}({\bf x})$-modules, we have $W(\OV{\lambda})\cong W'(\OV{\lambda})$.
	\end{proposition}
	
	\begin{proof}
		For $d\in S(k,i)$ and $T\in\Tab_{\OV{\lambda}}$, sending the basis element $d\otimes v_{T}$ to $w_{U}$, where $U\in\STab_{\OV{\lambda}}$ is obtained via the procedure described on Page~\pageref{pa:pro}, gives the desired isomorphism. 
	\end{proof}
	
	\begin{remark} We end this section with the following two remarks.
		\begin{enumerate}
			\item Let $r=1$. For partition algebras, the authors \cite{TTN} worked with the elements that they called "$m$-symmetric partition diagram" in the place of the elements of $S(k,m)$. It is easy to see that there is a one-to-one correspondence between elements of $S(k,m)$ and $m$-symmetric partition diagrams.
			
			\item Another important basis of $S(\OV{\lambda})$ is the so-called seminomral basis, see, \cite[p. 169]{TA} and \cite{Ariki}. Following the same ideas as above, one can use this seminormal basis to construct yet another basis for $W(\OV{\lambda})$ in which an explicit action of the generators of $\cPar_k({\bf x})$ can be given.
		\end{enumerate}
	\end{remark}
	
	\subsection{Bijection 2} For the second bijection, we need to recall a result from \cite[Section 4.1]{Shimozono}. For this, we first recall necessary  definitions from \cite[Section 2.1]{Shimozono}.
	
	For partitions $\lambda, \mu$, we write $\mu\subset\lambda$ if the shape of $\mu$ is contained in the shape of $\lambda$, i.e. $\mu_i\leq \lambda_i$, for all $i$. Then the skew shape $\lambda/\mu$ consists of the cells in $\lambda$ which are not in  $\mu$.
	
	A skew shape $\lambda/\mu$ is said to be connected provided that any two cells in it can be connected by a sequence of cells in it for which each pair of neighbours has a common side.
	A connected skew shape $\lambda/\mu$ is called a {\em ribbon shape} if it has no $2 \times 2$ subshape or, alternatively, each diagonal intersects this shape in at most one cell. Such a ribbon shape is called $\lambda$-removable and $\mu$-addable. We write $\mu\lessdot= \lambda$ provided that $\lambda=\mu$ or $\lambda/\mu$ is a ribbon shape. 
	If $\mu\lessdot=\lambda$ and $\lambda\neq\mu$, we write $\mu\lessdot\lambda$.
	An $r$-ribbon is a ribbon shape containing exactly  $r$ cells.

	For an $r$-ribbon $h$, the head of $h$ is its northeastmost cell and the tail of $h$ is its southwestmost cell. Then the spin of $h$, denoted $\spin(h)$, is defined as the row index of the tail of $h$ minus the row index of the head of $h$. Clearly, $\spin(h)\in\{0,1,\ldots,r-1\}$. 
	
	\begin{definition}[$r$-ribbon tableau] A standard $r$-ribbon tableau of shape $\lambda/\mu$ is a chain of partitions 
		\begin{displaymath}
			\mu=\lambda^{0}\lessdot= \lambda^{1}\lessdot =\cdots \lessdot=\lambda^n=\lambda. 
		\end{displaymath}
		such that, if $\lambda^{i-1}\lessdot \lambda^{i}$, then $\lambda^{i-1}/\lambda^{i}$ is an $r$-ribbon shape.
		Furthermore, for the case $\lambda^{i-1}\lessdot \lambda^{i}$, we say that the corresponding $r$-ribbon tableau contains $i$. A ribbon tableau can be depicted by putting the number $i$ that it contains in the corresponding $r$-ribbon $\lambda^{i-1}/\lambda^{i}$.
	\end{definition}

	\begin{example}
		Let $r=3$. The standard $r$-ribbon tableau determined by chain of partitions
		$$\varnothing\lessdot (2,1)\lessdot = (2,1)\lessdot (3,3)\lessdot=(3,3)\lessdot (3,3,1,1,1)$$
		can be equivalently depicted by the following  
		$$\ytableausetup{boxsize=1.5em} \begin{ytableau}
			1 & 1 & 3 \\
			1 & 3 & 3\\
			5\\
			5\\
			5
		\end{ytableau}$$
	\end{example}
	
	For $\lambda\in \mathcal{P}_{rn}$, let $\Rtab_{\lambda}^r$ denote the set of $r$-ribbon tableaux whose shape is $\lambda$. 
	To each colored permutation as in \eqref{al:colbi}, \cite[p.~298]{Shimozono} associates a pair of $r$-ribbon tableaux, providing a bijection
	\begin{displaymath}
		\SW:G(r,n) \to \bigsqcup_{\lambda\in\mathcal{P}_{rn}}\Rtab_{\lambda}^r\times \Rtab_{\lambda}^r.
	\end{displaymath}
	This bijection is called a ribbon Schensted bijection in \cite[p. 298]{Shimozono}.
	
	Before giving the explicit description of $\SW$, we first recall the following key result, see \cite[Lemma~6]{Shimozono}.
	
	\begin{lemma}\label{lm:addspin}
		Let $\mu$ be a shape and $c$ be a color. Then the following holds:
		\begin{enumerate}
			
			\item There is a $\mu$-addable $r$-ribbon of spin $c$.
			
			\item For any $\mu$-removable $r$-ribbon $h$ with $\spin(h)\leq c$, there is a $\mu$-addable $r$-ribbon of spin $c$ that is strictly southwest to $h$.
			
		\end{enumerate}
		
	\end{lemma}
	
	Assume that we have a setup as in Lemma \ref{lm:addspin}. Let $\firstr(\mu, c)$ (\enquote{the first ribbon}) be the northeastmost $\mu$-addable ribbon of spin $c$. Let $\nextr(\mu,h)$ (\enquote{the next ribbon}) be the northeastmost $\mu$-addable ribbon of spin $c$ that is strictly southwest to $h$. Note that both, $\firstr(\mu,c)$ and $\nextr(\mu,h)$, are well-defined thanks to Lemma~\ref{lm:addspin}. For two ribbons $h_1$ and $h_2$ which are not equal but have a non-empty intersection, let
	\begin{displaymath}
		\bumpout(h_1,h_2)=(h_2\setminus h_1)\cup\{(i+1,j+1)\mid (i,j)\in h_1\cap h_2 \}.
	\end{displaymath}

	{\bf Insertion process.} Let $T$ be an $r$-ribbon tableau and $v$ a value not contained in $T$. We describe the procedure of inserting a colored value $\binom{c}{v}$. Let $v_1<v_2<\cdots< v_m$ be the values contained in $T$ which are bigger than $v$. Let $h_j$ denote the $r$-ribbon that contains the value $v_j$ for $1\leq j\leq m$.  
	\begin{enumerate}[($i$)]
		\item Let $T_0$ be the tableau obtained from $T$ by removing all the ribbons $h_j$, for $1\leq j\leq m$.
		
		\item  Let $P_0$ be the tableau obtained from $T_0$ by adjoining the ribbon $h_0=\firstr(\sh(T_0),c)$ containing the value $v$.
		
		\item Given $P_{j-1}$, let $h_{j-1}'=\sh(P_{j-1})/\sh(T_{j-1})$. Then $P_j$ is obtained from $P_{j-1}$ by adjoining a ribbon of value $v_j$ at the position determined by the following:
		\begin{equation}
			\begin{cases}
				h_{j}, &  \text{ if } h_{j-1}'\cap h_{j}=\varnothing;\\
				\nextr(\sh(P_{j-1}),h_j), & \text{ if } h_{j-1}'=h_{j};\\
				\bumpout(h_{j-1}',h_j),& \text{ otherwise}.
			\end{cases}
		\end{equation}
	\end{enumerate}

	Now we can describe the algorithm $\SW$. Suppose a colored permutation \eqref{al:colbi} is given. Let $P_0=\varnothing$ and $Q_0=\varnothing$. For $1\leq j\leq n$, let $P_j$ be obtained by inserting, using the insertion process described above, the colored value 
	$\binom{i_j}{\sigma(j)}$ to $P_{j-1}$, and let $Q_j$ be obtained from $Q_{j-1}$ by adjoining a ribbon $\sh(P_j)/\sh(P_{j-1})$ which is labeled by $j$. Let $P=P_n$ and $Q=Q_n$. Then $\SW$ sends the colored permutation \eqref{al:colbi} to the pair $(P,Q)$ of $r$-ribbon tableaux of the same shape.
	
	\begin{remark}
		The algorithm $\SW$ can be defined starting with an arbitrary $r$-core $\kappa$, see \cite[Remark 7.1]{Shimozono}. An $r$-core is a
		partition that does not have any $r$-ribbons on its rim. This results in a bijection $\SW_{\kappa}$ from $G(r,n)$ to the set of certain pairs $(P,Q)$ of tableaux of the same shape and with $r$-core $\kappa$. Furthermore, if the $r$-core $\kappa$ is taken to be ``big enough'', then $\SW_\kappa$ is equivalent to the $\RS$ algorithm. (The key point in this case is that in the insertion process there will not be any nontrivial proper intersection of $r$-ribbons and so the case of  bumping out will never occur.)  Indeed, for ``big enough'' $\kappa$, we have following commutative diagram, where $\pr$ denotes projection map (for example, see \cite[Section 6]{Stanton} and \cite[Section 4]{LTT1} for its definition):
		\begin{displaymath}
			\xymatrixcolsep{9pc}
			\xymatrix{	
				G(r,n)\ar[r]^{\SW_\kappa}\ar[dr]^{\RS} 	& \text{ \{pairs of $r$-ribbon tableaux\}} \ar[d]^{\pr\times \pr}\\
				& \text{\{pairs of $r$-tuple tableaux\}}
			}
		\end{displaymath}
	\end{remark}

	The following definition is required to state Proposition \ref{prop:SW}.
	If we apply the $\SW$ bijection to the following element 
	\begin{equation}
		\begin{pmatrix}
			i_1 & i_2 & \cdots & i_n\\
			1 & 2 & \cdots & n\\
			1 & 2 & \cdots & n
		\end{pmatrix},
	\end{equation}
	then both the insertion and the recording tableaux are the same which can be characterized by the following.
	Let $\mu_j$ denote the tableaux obtained at the $j$-th step. Then the $r$-ribbon added at $(j+1)$-th step is $\firstr(\mu_{j},i_{j+1})$. We call such tableau to be of \emph{special type}. Moreover, the spin of the $r$-ribbon added at step $j$ in the recording tableau of special type is precisely the color $i_j$.

	{\bf Set-partition $r$-ribbon tableau.} Let $\lambda$ be an $r$-ribbon tableau containing $i_1<i_2<\cdots< i_l$. Let $\mathcal{S}_l$ be a set consisting of $l$ disjoint subsets of $\ZZ_{\geq 0}$. The set $\mathcal{S}_l$ is a totally order set with respect to the maximum entry order on subsets of $\ZZ_{\geq 0}$.  Then there is a unique order preserving bijection $\phi$ from $\{i_1,i_2,\cdots,i_l\}$ to $\mathcal{S}_l$. A set-partition $r$-ribbon tableau of shape $\lambda$ is obtained by replacing every entry $i_j$ of $\lambda$ by $\phi(i_j)$, for $j=1,\ldots,l$. The set of entries of a  set-partition $r$-ribbon tableau is called its content. 
	
	Let $\SRTab^{r}_{\lambda}$ denote the set of pairs $(P,S)$ of $r$-ribbon tableau such that 
	\begin{enumerate}
		\item $\sh(P)=\lambda$,
		\item $S$ is a set-partition $r$-ribbon tableau of the special type (see above),
		\item the union of contents of $P$ and $S$ is a set-partition of $\{1,2,\ldots,k\}$,
		\item $0<\lvert \sh(P)\rvert + \lvert \sh(S)\rvert \leq rk$.
	\end{enumerate}
	
	\begin{proposition}\label{prop:SW}
		There is a bijection
		\begin{align*}
			\SW: \cPar_{k} &\to \bigsqcup_{s=0}^{k}\bigsqcup_{\lambda\in\mathcal{P}_{rs}}\SRTab_{\lambda}^{r}\times \SRTab_{\lambda}^{r}  .
		\end{align*}
	\end{proposition}
	
	\begin{proof}
		For each $d\in \cPar_k$, below we
		describe how to associate to $d$
		a pair $\big((P,S),(Q,T)\big)$ in the right hand
		side of the above map.
		
		We have the colored partition array \eqref{al:setpart} associated to $d$. Note that an analogue of the bijection $\SW$ for colored permutation~\eqref{al:colbi} can be easily observed if we replace the set $\{1,2,\ldots,n\}$ by any total ordered set of cardinality $n$, for $n\in\ZZ_{\geq 0}$. In particular, by doing this for maximum entry order and the colored partition array~\eqref{al:colbi}, we get the set-partition $r$-ribbon tableaux $P$ and $Q$ of the same shape $\lambda \in \mathcal{P}_{rs}$, where $0\leq s\leq k$. 
		
		Let $(C_1',\zeta^{j'_1}),\ldots, (C_{m}',\zeta^{j_{m}'})$ be the nonpropagating parts at the bottom row of $d$, arranged so that $C_1'<\cdots <C_{m}'$ with respect to the maximum entry order. Consider the colored array
		\begin{equation}
			\begin{pmatrix}
				j_1' & j_2' & \cdots & j_m'\\
				1 & 2 & \cdots & m\\
				1 & 2 & \cdots & m
			\end{pmatrix}
		\end{equation}
		which gives us an $r$-ribbon tableau of special type. Now, replacing the entries $1, 2, \ldots, m$ in this tableau by $C_1', C_2', \ldots, C_m'$, respectively, we get a set-partition $r$-ribbon tableau of special type. This is our $S$. The same procedure applied to the nonpropagating parts at the top row of $d$, we get a set-partition $r$-ribbon tableau of special type. This is our $T$. 
		
		It is easy to see that this defines a bijection as stated.
	\end{proof}
	One may call the set-partition $r$-ribbon tableaux $P$ and $Q$ in the above proof the insertion and recording tableaux. 
	\begin{example}
		Let $r=5$ and $k=10$. We demonstrate the procedure stated in the proof of Proposition~\ref{prop:SW} for the  colored partition diagram $d$ in Example \ref{ex:Rs}. Recall that we have the following colored set-partition array
		\begin{equation}
			\begin{pmatrix}
				0 & 2 & 3 & 0 & 1& 2\\
				\{1\} & \{2\} & \{3\} & \{4\} & \{5,7\} & \{6,9\}\\
				\{5'\} & \{3'\} & \{1'\} & \{4'\} & \{2'\} & \{6',8'\}
			\end{pmatrix}.
		\end{equation}    
		From the procedure stated in the proof of Proposition~\ref{prop:SW}, we get following set-partition $r$-ribbon tableaux, where the tableaux appearing on the left hand side are obtained from the insertion process and then the tableaux appearing on the right hand side are the corresponding recording tableaux.
		
		\begin{small}
			$P_0=\begin{matrix}
				\ytableausetup{boxsize=2em} \begin{ytableau}
					\{5'\} & \{5'\} & \{5'\} & \{5'\} & \{5'\}
				\end{ytableau}
			\end{matrix} \hspace{3cm} Q_0=\begin{matrix}\ytableausetup{boxsize=2em} \begin{ytableau}
					\{1\} & \{1\} & \{1\} & \{1\} & \{1\}
			\end{ytableau}\end{matrix}$
			\vspace{0.3cm}
			
			$P_1=\begin{matrix}
				\ytableausetup{boxsize=2em} \begin{ytableau}
					\{3'\} & \{3'\} & \{3'\} & \{5'\} & \{5'\}\\
					\{3'\} & \{5'\} & \{5'\} & \{5'\}\\
					\{3'\}
				\end{ytableau}
			\end{matrix} \hspace{3cm} Q_1=\begin{matrix}\ytableausetup{boxsize=2em} \begin{ytableau}
					\{1\} & \{1\} & \{1\} & \{1\} & \{1\}\\
					\{2\} & \{2\} & \{2\} & \{2\}\\
					\{2\}
			\end{ytableau}\end{matrix}$
			\vspace{0.3cm}
			
			$P_2=\begin{matrix}
				\ytableausetup{boxsize=2em} \begin{ytableau}
					\{1'\} & \{1'\} & \{3'\} & \{5'\} & \{5'\}\\
					\{1'\} & \{3'\} & \{3'\} & \{5'\}\\
					\{1'\} & \{3'\} & \{5'\} & \{5'\}\\
					\{1'\} & \{3'\}
				\end{ytableau}
			\end{matrix} \hspace{3cm} Q_2=\begin{matrix}\ytableausetup{boxsize=2em} \begin{ytableau}
					\{1\} & \{1\} & \{1\} & \{1\} & \{1\}\\
					\{2\} & \{2\} & \{2\} & \{2\}\\
					\{2\} & \{3\} &  \{3\} & \{3\}\\
					\{3\} & \{3\}
			\end{ytableau}\end{matrix}$
			\vspace{0.3cm}
			
			$P_3=\begin{matrix}
				\ytableausetup{boxsize=2em} \begin{ytableau}
					\{1'\} & \{1'\} & \{3'\} & \{4'\} & \{4'\} & \{4'\} & \{4'\} & \{4'\}\\
					\{1'\} & \{3'\} & \{3'\} & \{5'\} & \{5'\} & \{5'\}\\
					\{1'\} & \{3'\} & \{5'\} & \{5'\}\\
					\{1'\} & \{3'\}
				\end{ytableau}
			\end{matrix} \hspace{1cm} Q_3=\begin{matrix}\ytableausetup{boxsize=2em} \begin{ytableau}
					\{1\} & \{1\} & \{1\} & \{1\} & \{1\} & \{4\} & \{4\} & \{4\}\\
					\{2\} & \{2\} & \{2\} & \{2\} & \{4\} & \{4\}\\
					\{2\} & \{3\} &  \{3\} & \{3\}\\
					\{3\} & \{3\}
			\end{ytableau}\end{matrix}$
			\vspace{0.3cm}
			
			$P_4=\begin{matrix}
				\resizebox{6cm}{!}{\ytableausetup{boxsize=2.4em} \begin{ytableau}
						\{1'\} & \{1'\} & \{2'\} & \{2'\} & \{2'\} & \{4'\} & \{4'\} & \{4'\}\\
						\{1'\} & \{2'\} & \{2'\} & \{3'\} & \{4'\} & \{4'\}\\
						\{1'\} & \{3'\} & \{3'\} & \{3'\}\\
						\{1'\} & \{3'\} & \{5'\} &\{5'\}\\
						\{5'\} & \{5'\} & \{5'\}
				\end{ytableau}}
			\end{matrix}  Q_4=\begin{matrix}
				\resizebox{6cm}{!}{\ytableausetup{boxsize=2.4em} \begin{ytableau}
						\{1\} & \{1\} & \{1\} & \{1\} & \{1\} & \{4\} & \{4\} & \{4\}\\
						\{2\} & \{2\} & \{2\} & \{2\} & \{4\} & \{4\}\\
						\{2\} & \{3\} &  \{3\} & \{3\}\\
						\{3\} & \{3\} & \{5,7\} & \{5,7\}\\
						\{5,7\} & \{5,7\} & \{5,7\}
			\end{ytableau}}\end{matrix}$
			\vspace{0.3cm}

			$P_5=\begin{matrix}
				\resizebox{6cm}{!} {\ytableausetup{boxsize=3em} \begin{ytableau}
						\{1'\} & \{1'\} & \{2'\} & \{2'\} & \{2'\} & \{4'\} & \{4'\} & \{4'\}\\
						\{1'\} & \{2'\} & \{2'\} & \{3'\} & \{4'\} & \{4'\}\\
						\{1'\} & \{3'\} & \{3'\} & \{3'\}\\
						\{1'\} & \{3'\} & \{5'\} &\{5'\}\\
						\{5'\} & \{5'\} & \{5'\}\\
						\{6',8'\} & \{6',8'\} & \{6',8'\}\\
						\{6',8'\}\\
						\{6',8'\}
				\end{ytableau}}
			\end{matrix}  
			Q_5=\begin{matrix}
				\resizebox{6cm}{!}{\ytableausetup{boxsize=3em} \begin{ytableau}
						\{1\} & \{1\} & \{1\} & \{1\} & \{1\} & \{4\} & \{4\} & \{4\}\\
						\{2\} & \{2\} & \{2\} & \{2\} & \{4\} & \{4\}\\
						\{2\} & \{3\} &  \{3\} & \{3\}\\
						\{3\} & \{3\} & \{5,7\} & \{5,7\}\\
						\{5,7\} & \{5,7\} & \{5,7\}\\
						\{6,9\} & \{6,9\} & \{6,9\}\\
						\{6,9\}\\
						\{6,9\}
			\end{ytableau}}\end{matrix}$
		\end{small}
		\vspace{0.4cm}
		
		From the bottom and the top nonpropagating parts of $d$, we get 
		
		\begin{small} $S=\begin{matrix}
				\resizebox{6cm}{!}{ \ytableausetup{boxsize=3.5em} \begin{ytableau}
						\{9'\} & \{9'\} & \{9'\} & \{9'\} & \{9'\}\\
						\{7',10'\} & \{7',10'\} & \{7',10'\} & \{11'\} &\{11'\}\\
						\{7',10'\} & \{11'\} & \{11'\} & \{11'\}\\
						\{7',10'\} 
				\end{ytableau}}
			\end{matrix}$, \quad $T=\begin{matrix}\resizebox{6cm}{!}{\ytableausetup{boxsize=3.5 em} \begin{ytableau}
						\{8\} & \{8\} & \{8\} & \{8\} & \{8\}\\
						\{10,11\} & \{10,11\} & \{10,11\} \\
						\{10,11\} \\
						\{10,11\} 
			\end{ytableau}}\end{matrix}.$
		\end{small}
		
		The bijection in Proposition~\ref{prop:SW} sends $d$ to $\big((P_5,S),(Q_5,T)\big)$.
	\end{example}
	
	\begin{remark} 
		For $r=2$, one can use Garfinkle's algorithm
		from \cite{Gar1,Gar2,Gar3} (see also~\cite{Tom1,Tom2}) to produce yet another RS type correspondence for $\cPar_k$. In fact, one obtains a series of bijective maps, indexed by a $2$-core. The latter are in bijection with nonnegative integers,
		where $i$ corresponds to the $2$-core $(i,i-1,\dots,1)$. If $i$ is big enough, 
		then the three algorithms (\cite{white,Shimozono,Gar2}) are all equivalent.
	\end{remark}

	\subsection{Application: Green's left and right relations for the colored partition monoid}
	We first recall the definition of Green's left, right and two-sided relations 
	$\mathcal{L}$, $\mathcal{R}$ and $\mathcal{J}$ for  a monoid, see \cite{Green}. Let $M$ be a monoid and $m,m'\in M$. Then
	\begin{enumerate}
		\item $(m,m')\in\mathcal{L}$ if and only if $Mm=Mm'$,
		\item $(m,m')\in\mathcal{R}$ if and only if $mM=m'M$,
		\item $(m,m')\in\mathcal{J}$ if and only if $MmM=Mm'M$.
	\end{enumerate}
	
	In~\cite[Proposition~4.2]{East}, Green's relations for the partition monoid were given.
	\begin{proposition}\label{prop-new-45}
		Let $d$ and $d'$ be two colored partition diagrams in $\cPar_k$. Suppose that the correspondence  in Proposition~\ref{prop:Rs} sends $d$  to $\big((P,S),(Q,T)\big)$ and $d'$ to $\big((P',S'),(Q',T')\big)$. Then 
		\begin{enumerate}[(i)]
			\item  $(d,d')\in\mathcal{L}$ if and only if $\cont(P)=\cont(P'), \quad S=S'$;
			\item $(d,d')\in\mathcal{R}$ if and only if 
			$\cont(Q)=\cont(Q'),\quad T=T'$;
			
			\item $(d,d')\in\mathcal{J}$ if and only if $ \lvert \sh(P)\rvert=\lvert \sh(P')\rvert$.
		\end{enumerate}
	\end{proposition}
	
	\begin{proof} Let $d,d'\in\cPar_k$. 
		\begin{enumerate}[$(i)$]
			\item We have $\cPar_{k}d=\cPar_{k}d'$ if and only if 
			the nonpropagating bottom colored set-partitions of $d$ and $d'$ are equal, and each bottom constituent of a propagating part in $d$ is also a bottom constituent of a propagating part in $d'$, but the propagating part itself could carry any color. Since the contents of $P$ and  $P'$ are precisely the sets of bottom propagating parts of $d$ and $d'$, respectively, we deduce from Proposition~\ref{prop:Rs} that $d$ is left related to $d'$ if and only if $\cont(P)=\cont(P')$ and $S=S'$.
			
			\item It follows by the same argument as in the previous paragraph where we now replace the bottom parts by the top parts, $P$ by $Q$ and $P'$ by $Q'$.
			
			\item We have $\cPar_kd\cPar_k=\cPar_kd'\cPar_k$ if and only if $\pn(d)=\pn(d')$. Since the sizes of $\sh(P)$ and $\sh(P')$ are precisely $\pn(d)$ and $\pn(d')$, respectively, from Proposition \ref{prop:Rs}
			we obtain that $d$ is two-sided related to $d'$ if and only if $\lvert\sh(P)\rvert=\lvert\sh(P')\rvert$. 
		\end{enumerate}
	\end{proof}
	
	In the formulation of Proposition~\ref{prop-new-45} one can replace the bijection from
	Proposition~\ref{prop:Rs}
	by the bijection from
	Proposition~\ref{prop:SW}
	and the assertion still holds.

	\textbf{Acknowledgments.}
	The first author is supported by the Swedish Research Council. The authors thank James East for helpful comments.
	
	\bibliographystyle{alphaurl}
	%\bibliography{references}
	
	\newcommand{\etalchar}[1]{$^{#1}$}

\end{document}